\documentclass{article}
\usepackage[utf8]{inputenc}
\usepackage[margin=2cm]{geometry}
\usepackage{amsmath,amsthm,amssymb,bm,graphicx}
\usepackage{natbib}
\usepackage[dvipsnames]{xcolor}
\usepackage{booktabs,makecell,float}
\usepackage{pifont}
\usepackage{xr}
\usepackage{authblk}
\usepackage[OT1]{fontenc}
\usepackage[colorlinks,citecolor=blue,urlcolor=blue]{hyperref}

\newcommand{\W}{(W_1,W_2)}
\newcommand{\V}{(V_1,V_2)}
\newcommand{\X}{(X_1,X_2)}
\newcommand{\Y}{(Y_1,Y_2)}

\renewcommand{\P}{\mbox{P}}
\newcommand{\E}{\mbox{E}}
\renewcommand{\d}{\mbox{d}}
\newcommand{\RV}{\mbox{RV}}

\renewcommand{\bar}{\overline}

\newcommand\ci{\perp\!\!\!\perp}

\newtheorem{prop}{Proposition}
\newtheorem{cor}{Corollary}
\newtheorem{definition}{Definition}
\newtheorem{lemma}{Lemma}
\newtheorem*{lemma*}{Lemma}
\theoremstyle{definition}
\newtheorem{example}{Example}
\theoremstyle{definition}
\newtheorem*{example*}{Example}
\theoremstyle{definition}
\newtheorem{remark}{Remark}

\begin{document}
\title{Extremal dependence of random scale constructions}

\author[1]{Sebastian Engelke}
\author[2]{Thomas Opitz}
\author[3]{Jennifer Wadsworth}
\affil[1]{\small Research Center for Statistics, University of Geneva, Boulevard du Pont d’Arve 40,
1205 Geneva, Switzerland, \texttt{sebastian.engelke@unige.ch}}
\affil[2]{Biostatistics and Spatial Processes, INRA, 84914, Avignon, France, \texttt{thomas.opitz@inra.fr}}
\affil[3]{Department of Mathematics and Statistics, Fylde College, Lancaster University, LA1 4YF, UK, \texttt{j.wadsworth@lancaster.ac.uk}}


\maketitle
\begin{abstract}
A bivariate random vector can exhibit either asymptotic independence or dependence between the largest values of its components. When used as a statistical model for risk assessment in fields such as finance, insurance or meteorology,  it is crucial to understand which of the two asymptotic regimes occurs. Motivated by their ubiquity and flexibility, we consider the extremal dependence properties of vectors with a random scale construction $(X_1,X_2)=R(W_1,W_2)$, with non-degenerate $R>0$ independent of $(W_1,W_2)$. Focusing on the presence and strength of asymptotic tail dependence, as expressed through commonly-used summary parameters, broad factors that affect the results are: the heaviness of the tails of $R$ and $(W_1,W_2)$, the shape of the support of $(W_1,W_2)$, and dependence between $(W_1,W_2)$. When $R$ is distinctly lighter tailed than $(W_1,W_2)$, the extremal dependence of $(X_1,X_2)$ is typically the same as that of $(W_1,W_2)$, whereas similar or heavier tails for $R$ compared to $(W_1,W_2)$ typically result in increased extremal dependence. Similar tail heavinesses represent the most interesting and technical cases, and we find both asymptotic independence and dependence of $(X_1,X_2)$ possible in such cases when $(W_1,W_2)$ exhibit asymptotic independence. 
The bivariate case often directly extends to higher-dimensional vectors and spatial processes, where the dependence is mainly analyzed in terms of summaries of bivariate sub-vectors.
The results unify and extend many existing examples, and we use them to propose new models that encompass both dependence classes.
\end{abstract}



\section{Introduction}
\label{sec:Introduction}

A rich variety of bivariate dependence models have a pseudo-polar representation 
\begin{align}
 \X=R\W, \qquad R > 0,~\mbox{independent of}~ \W\in\mathcal{W} \subseteq \mathbb{R}^2, \label{eq:xconst}
\end{align}
where we term $R$ the radial variable, assumed to have a non-degenerate distribution, and $\W$ the angular variables. Indeed, many well-known copula families, including the elliptical, Archimedean, Liouville, and multivariate Pareto families have such a representation. In this work, our focus is on the upper tail dependence of such constructions. In particular, we examine whether a given $\X$ displays asymptotic dependence or asymptotic independence, and the strength of dependence within these classes. Our results are particularly useful for constructing new models with properties that reflect the challenges of real data in, for instance, finance, meteorology and hydrology. Specifically, it is often ambiguous whether data should be modeled using an asymptotically dependent or asymptotically independent distribution, and most families of distributions only exhibit one type of dependence. A vector $\X$ with $X_j\sim F_{X_j}$ is said to display asymptotic dependence if the limit
\begin{align}
 \chi_X = \lim_{q\to 1} \P\{X_1\geq F_{X_1}^{-1}(q), X_2\geq F_{X_2}^{-1}(q)\}/(1-q) \label{eq:chi}
\end{align}
exists and is positive; a limit of zero defines asymptotic independence. In~\eqref{eq:chi} and throughout, $F_{X_j}^{-1}$ denotes the (generalized) inverse of the distribution function $F_{X_j}$. The parameter $\chi_X$ is termed the (upper) tail dependence coefficient, and the value of $\chi_X\in(0,1]$ summarizes the strength of the dependence within the class of asymptotically dependent variables. Under asymptotic independence, a more useful summary is the rate at which the convergence to zero in equation~\eqref{eq:chi} occurs, and a widely satisfied assumption \citep{LedfordTawn1997} is
\begin{align}
\P\{X_1\geq F_{X_1}^{-1}(q), X_2\geq F_{X_2}^{-1}(q)\} = \ell(1-q)(1-q)^{1/\eta_X}, \qquad \eta_X \in [0,1], \label{eq:eta}
\end{align}
where $\ell:[0,1] \to \mathbb{R}_+$ is slowly varying at zero, i.e., $\lim_{s\to 0} \ell(sx)/\ell(s) = 1$, $x>0$. The parameter $\eta_X$ is termed the residual tail dependence coefficient; positive and negative extremal association are indicated respectively by $\eta_X\in(1/2,1]$ and $\eta_X\in[0,1/2)$, whilst asymptotically dependent variables have $\eta_X=1$ and $\chi_X=\lim_{q\to 1}\ell(1-q)$. A value of $\eta_X = 0$ means that the left-hand side of~\eqref{eq:eta} decays faster than any power of $1-q$, whilst if the left-hand side is exactly zero for some $q<1$, we say that $\eta_X$ is not defined.

Our particular interest in the extremal dependence of constructions of the form~\eqref{eq:xconst} stems not from their novelty, but from their ubiquity and flexibility. As mentioned,~\eqref{eq:xconst} already encompasses many well-known families, and moreover these families display different types of extremal dependence, which may be determined by the distribution of $R$, the distribution of $\W$, or its support $\mathcal{W}$. There is a large body of literature that treats either individual constructions of the form~\eqref{eq:xconst}, or a particular subset of these constructions where $R$ or $\W$ have certain specified properties; this literature will be reviewed in Section~\ref{sec:Examples}. Our aim is to bring this scattered treatment together and more systematically characterize how the extremal dependence of $\X$ is determined by the properties of $R$ and $\W$. By understanding which facets of the construction lead to different dependence properties, we are able to determine dependence models that can capture both types of extremal dependence within a single parametric family; the recent proposals in \citet{Wadsworthetal2017}, \citet{Huseretal2017} and \citet{HuserWadsworth2017} are specific examples of this.  

A broad split in representations of type~\eqref{eq:xconst} is the dimension of $\mathcal{W}$, the support of $\W$. The most common case in the literature is that $\mathcal{W}$ is a one-dimensional subset of $\mathbb{R}^2$, such as the unit sphere defined by some norm or other homogeneous function. Examples include the Mahalanobis norm (elliptical distributions), $L_1$ norm (Archimedean and Liouville distributions), or $L_\infty$ norm (multivariate Pareto distributions). On top of the support $\mathcal{W}$, to obtain distributions within a particular family, $R$ or $\W$ may be specified to have a certain distribution. Where $\mathcal{W}$ is two-dimensional, it may sometimes be reduced to the one-dimensional case by redefining $R$, such as in the Gaussian scale mixtures of \citet{Huseretal2017}; other times, such as for the scale mixtures of log-Gaussian variables in \citet{Krupskiietal2017}, or the model presented in~\citet{HuserWadsworth2017}, this cannot be done. Where $\mathcal{W}$ is two-dimensional, the possible constructions stemming from~\eqref{eq:xconst} form an especially large class, since $\W$ can itself have any copula. In this case, we focus on how the multiplication by $R$ changes the extremal dependence of $\W$, summarized by the coefficients $(\chi_W, \eta_W)$, to obtain the extremal dependence of the modified vector $\X$ in terms of its coefficients $(\chi_X, \eta_X)$. The marginal distributions of $\W$ and $R$ will play a crucial role, since, intuitively, the heavier the tail of $R$ the more additional dependence is introduced in the vector $\X$.

As we are focused on the upper tail of $\X$, we henceforth assume $\W\in\mathbb{R}^2_+$; by the invariance of copulas to monotonic marginal transformations, this also covers random location constructions of the form $ \Y= S+\V$, $S\in\mathbb{R}$, $\V\in\mathcal{V}\subseteq\mathbb{R}^2$. For simplicity of presentation, we will often make the restriction that $W_1$ and $W_2$ have the same distribution, with comments on relaxations of this assumption given in Section~\ref{sec:Conclusions}. Furthermore, whilst our focus on the bivariate case permits simpler notation, many of the results are directly applicable to the bivariate margins of multivariate and spatial models, whose extremal dependence is typically analyzed in terms of the coefficients \eqref{eq:chi} and \eqref{eq:eta}. Examples are given in Section~\ref{sec:Examples}, with further comment on higher dimensions in Section~\ref{sec:Conclusions}.

There is no widely recognized standard for ordering univariate tail decay rates from the slowest to the fastest, although a broad characterization is given by the three domains of attraction of the maximum. We say that the random variable $R$ is in the max-domain of attraction (MDA) of a generalized extreme value distribution if there exists a function $b(t)>0$ such that as $t \to r^\star=\sup\{r: \P(R\leq r) <1\}$,
\begin{align*}
  \P(R \geq t + r/b(t))/ \P(R\geq t) \to (1+\xi r)^{-1/\xi}_+, \qquad r\geq 0,
\end{align*}
for some $\xi \in\mathbb{R}$, where $a_+=\max(a,0)$. The cases $\xi>0, \xi=0, \xi<0$ define respectively the Fr\'{e}chet, Gumbel and negative Weibull domains of attraction; the tail heaviness of $R$ increases with $\xi$. However, the Gumbel limit in particular attracts distributions with highly diverse tail behavior such as finite upper bounds or heavy tails. Overall, this classification is therefore too coarse for our requirements, and it excludes important classes such as superheavy-tailed distributions defined through the property of heavy-tailed log-transformed random variables. In addition to the maximum domains of attraction, we will utilize various commonly used tail classes, which are defined in Section~\ref{sec:Notation}.

We begin in Section~\ref{sec:Constrained} by presenting results concerning the tail dependence of construction~\eqref{eq:xconst} according to the tail behavior of $R$ and the shape of $\mathcal{W}$, in the case where it is a one-dimensional support defined through a norm. We then characterize various cases where $\mathcal{W}$ is two-dimensional, according to the behavior of both $R$ and $\W$, in Section~\ref{sec:Unconstrained}. Section~\ref{sec:Examples} is devoted to literature review and framing a large number of existing examples in terms of our general results, whilst Section~\ref{sec:NewExamples} illustrates the properties of some new examples inspired by the developments in the manuscript. In Section~\ref{sec:Conclusions} we comment on generalizations and conclude. Proofs are presented in Section~\ref{sec:Proofs}.

\subsection{Terminology and notation}
\label{sec:Notation}
For a random variable $Q$, we define its survival function $\bar{F}_{Q}(q)=\P(Q\geq q)$, and distribution function $F_Q(q) = 1-\bar{F}_Q(q)$. If $Q$ represents a bivariate random vector $Q=(Q_1,Q_2)$, we denote the minimum of its margins by $Q_\wedge = Q_1 \wedge Q_2$. For two functions $f$ and $g$ with $g(x)\not=0$ for values $x$ above some threshold value $x_0$, we write $f \sim g$ if $f(x)/g(x)\rightarrow 1$, where the limit is considered for $x\rightarrow\infty$ if not stated otherwise. Similarly, we write $f(x)=o(g(x))$ to indicate that $f(x)/g(x)\rightarrow 0$. The convolution of $X\sim F_X$ and $Y\sim F_Y$ is denoted $F_X \star F_Y=F_{X+Y}$. We recall definitions of upper tail behavior classes for a random variable $X$ with distribution $F$. Key tail parameters for these classes may be given as subscript, such as in $\mathrm{ET}_\alpha$  to refer to exponential-tailed distributions with rate $\alpha$, but we may omit the subscript if the specific value of the parameter is not of interest. 

\begin{definition}[Light-, heavy- and superheavy-tailed distributions] 
	The distribution $F$ is \emph{heavy-tailed} if $\exp(\lambda x) \overline{F}(x)\rightarrow \infty$ as $x\rightarrow \infty$, for any $\lambda>0$. Further, $F$ is \emph{superheavy-tailed} if $F(\exp(\cdot))$  is heavy-tailed. If $F$ is not heavy-tailed, it is \emph{light-tailed}. 
\end{definition}
\begin{definition}[Regularly varying functions and distributions ($\RV_\alpha^0$ and $\RV_\alpha^\infty$)]
 A measurable function $g$ is \emph{regularly varying at infinity or at zero with index $\alpha \in \mathbb{R}$} if $g(tx)/g(t) \to x^{\alpha}$ as $t\rightarrow\infty$ or $t\to 0$ respectively for any $x>0$. We write $g\in \RV_\alpha^\infty$ or $g\in \RV_\alpha^0$ respectively. If $\alpha=0$, then $g$ is said to be \emph{slowly varying}. A probability distribution $F$ with upper endpoint $x^\star=\infty$ is \emph{regularly varying with index $\alpha\geq 0$} if  $\overline{F}\in\RV_{-\alpha}^\infty$. If $x^\star<\infty$, then $F$ is \emph{regularly varying at $x^\star$ with index $\alpha$} if $\overline{F}(x^\star-\cdot)\in\RV_\alpha^0$.
\end{definition}
\begin{definition}[Exponential-tailed distributions ($\mathrm{ET}_\alpha$, $\mathrm{ET}_{\alpha,\beta}$)]
The distribution $F$ with upper endpoint $x^\star=\infty$ is \emph{exponential-tailed with rate $\alpha \geq 0$} if for any $x>0$, $\overline{F}(t+x)/\overline{F}(t)\rightarrow \exp(-\alpha x)$,
 $t\rightarrow\infty$. If $\alpha>0$ and
$\bar{F}(x)=r(x)\exp(-\alpha x)$, $r \in\RV_\beta^\infty$, we write $F\in\mathrm{ET}_{\alpha,\beta}$. 
\end{definition}
\noindent By definition, $F \in \mathrm{ET}_\alpha$ with $\alpha\geq 0$ if and only if $\bar{F}(\log(\cdot)) \in \RV_{-\alpha}^\infty$. The class $\mathrm{ET}_{\alpha,\beta}$ with $\beta>-1$ is referred to as \emph{gamma-tailed distributions}. Another important subclass of $\mathrm{ET}_\alpha$ are the convolution-equivalent distributions.
\begin{definition}[Convolution-equivalent distributions ($\mathrm{CE}_\alpha$)]
	\label{def:conveq}
The distribution $F$ is \emph{convolution equivalent with index $\alpha \geq 0$} if $F\in \mathrm{ET}_\alpha$ and $\overline{F\star F}(x)/\overline{F}(x)\rightarrow 2 \int_{-\infty}^{\infty} \exp(\alpha x) F(\mathrm{d}x)<\infty$.  
We write $F\in\mathrm{CE}_\alpha$. We refer to the class $\mathrm{CE}_0$ as \emph{subexponential distributions}.
\end{definition}
\begin{definition}[Weibull- and log-Weibull tailed distributions ($\mathrm{WT}_\beta$, $\mathrm{LWT}_\beta$)]\label{def:weib}
	The distribution $F$ is  \emph{Weibull-tailed with index $\beta>0$} if there exist $\alpha>0$, $\gamma\in\mathbb R$, and $r\in\RV_\gamma^\infty$  such that $\overline{F}(x)\sim r(x)\exp(-\alpha x^\beta)$. $F$ is \emph{log-Weibull-tailed} with index $\beta>0$ if $F(\exp(\cdot))\in\mathrm{WT}_\beta$. 
\end{definition}

We remark that some authors define heavy tails to be synonymous with regularly varying tails for which the tail index $\alpha>0$ \citep[e.g.][]{Resnick07}. The definition that we use is broader, and includes distributions such as the log-Gaussian, as well as regularly varying tails.
In practice, all of the heavy tailed distributions that we treat belong to the class of subexponential distributions, $\mathrm{CE}_0$.

\section{Constrained angular variables}
\label{sec:Constrained}
We focus firstly on the case where $\mathcal{W}$ is defined by a norm $\nu$; specifically let $\mathcal{W} = \{(w_1,w_2) \in \mathbb{R}^2_+: \nu(w_1,w_2) = 1\}$. Other types of constrained spaces may sometimes be of interest, but norm spheres are a common restriction, and this focus allows greater generality in other aspects. In particular, all components of the vector are bounded in absolute value when the value of the norm is fixed.  We examine the extremal dependence based on the heaviness of the tail of $R$. Because the $\W$ are bounded, and subject to additional mild assumptions, we can classify $R$ according to its MDA in this section.

The case where $R$ belongs to the Fr\'echet MDA is the least delicate: as long as $R$ has a much heavier tail than each of $\W$, results do not depend strongly on other considerations. No equality in distribution is assumed between $W_1,W_2$ in this case. When $R$ is in the Gumbel or negative Weibull MDA, the shape of the norm $\nu$ becomes important, and some minor additional regularity conditions are assumed, detailed in Section~\ref{sec:Gumbel}. 

\subsection{Radial variable in Fr\'{e}chet MDA}
\label{sec:Frechet}
Many of the most familiar results in the literature on extremal dependence concern the case when $R$ is in the Fr\'{e}chet MDA; this is equivalent to regular variation of the tail of $R$, namely $\bar{F}_R \in\RV_{-1/\xi}^\infty$, $\xi>0$, where $\alpha = 1/\xi$ is called the tail index. A classical example of this is the (multivariate) Pareto copula, which can be constructed as in equation~\eqref{eq:xconst} with $R$ standard Pareto, and $\mathcal{W} = \{(w_1,w_2) \in \mathbb{R}^2_+: \max(w_1,w_2) = 1\}$ \citep{FerreiradeHaan2014}. Pareto copulas can be identified with so-called extreme value copulas, which arise as the limiting copulas of suitably normalized componentwise maxima; see e.g.\ \citet{Rootzenetal2017}. The next result provides the general form of the tail dependence coefficient for these models.

\begin{prop}[$R$ in Fr\'echet MDA]
\label{prop:Frechet_chi}
Let $\bar{F}_R \in\RV_{-\alpha}^{\infty}$, $\alpha \geq 0$, $\P(W_1>0)=\P(W_2>0)=1$, and $\E(W_j^{\alpha+\varepsilon})<\infty$, $j=1,2$, for some $\varepsilon>0$. Then $\eta_X = 1$, and
\begin{align}
 \chi_X = \E\left[\min\left\{W_1^\alpha / \E(W_1^\alpha),W_2^\alpha / \E(W_2^\alpha)\right\}\right]. \label{eq:Rrvchi}
\end{align}
\end{prop}
\begin{remark}
 When $\bar{F}_R(r)\sim Cr^{-\alpha}$ for some $C>0$, then the condition $\E(W_j^{\alpha+\varepsilon})<\infty$ can be replaced by $\E(W_j^{\alpha})<\infty$, by Lemma~2.3 of \citet{DavisMikosch2008}.
\end{remark}
\begin{remark}
 The condition $\E(W_j^{\alpha+\varepsilon})<\infty$ is guaranteed when $\mathcal{W}$ is the unit sphere of a norm $\nu$; Proposition~\ref{prop:Frechet_chi} notably also covers the case where $\W \in\mathbb{R}_+^2$.
\end{remark}
\begin{remark}
\label{rmk:SVR}
The result includes the case $\alpha=0$, although the tail of such an $R$ is too heavy to be in any domain of attraction. In this case, $\chi_X=1$, representing perfect upper tail dependence. This case is discussed further in Section~\ref{sec:RSHT}.
\end{remark}
From~\eqref{eq:Rrvchi}, we observe that asymptotic dependence arises since $\P\{\min(W_1^\alpha/\E(W_1^\alpha),W_2^\alpha/\E(W_2^\alpha))>0\}=1$. If the conditions of the Proposition were relaxed to $\P(W_1>0),\P(W_2>0)>0$, then it is possible that for $\alpha>0$, $\P\{\min(W_1^\alpha/\E(W_1^\alpha),W_2^\alpha/\E(W_2^\alpha))>0\}=0$ which would yield asymptotic independence, and then $\eta_X$ would not be defined. The Fr\'{e}chet case with one-dimensional $\mathcal{W}$ is therefore very restricted in its capacity to represent varied asymptotically independent behaviors. A more complete description of  tail dependence is given by the exponent function, defined as
\begin{align}
 V_X(x_1,x_2) = \lim_{t\to \infty} t[1- \P(X_1\le F_{X_1}^{-1}\{1-1/(tx_1)\}, X_2 \le F_{X_2}^{-1}\{1-1/(tx_2)\})], \qquad x_1,x_2>0. \label{eq:expfn}
\end{align}
Small modifications to Proposition~\ref{prop:Frechet_chi} yield
\begin{align}
 V_X(x_1,x_2) = \E\left[\max\left\{W_1^\alpha /(\E(W_1^\alpha)x_1),W_2^\alpha/(\E(W_2^\alpha)x_2)\right\}\right].\label{eq:V}
\end{align}
The link between $\chi_X$ and $V_X(1,1)$ can be obtained simply by inclusion-exclusion arguments; in particular since $\min(a,b) = a+b-\max(a,b)$, $\chi_X=2-V_X(1,1)$. 
With the assumptions of Proposition~\ref{prop:Frechet_chi}, the random vector $(X_1,X_2)$ satisfies the condition of multivariate regular variation in the sense that $\lim_{t\rightarrow\infty} (1-F(tx_1,tx_2))/(1-F(t,t))$ has finite positive limit for any $x_1,x_2>0$; see \citet[][Section 5.4.2]{Resnick1987} for details about the notion of multivariate regular variation. To abstract away from the marginal distributions in $F$,  we can replace $tx_j$ by the quantile function $F_{X_j}^{-1}(1-1/(tx_j))$, $j=1,2$, in this limit. If the latter exists, it is given by $V(x_1,x_2)/V(1,1)$, and existence of the limit is equivalent to  Equation~\eqref{eq:expfn}. While many of the specific examples of random scale constructions presented in this paper satisfy Equation~\eqref{eq:expfn}, our general results focus on the behavior along the diagonal where $x_1=x_2$, and we do not aim to provide specific statements about off-diagonal behavior with $x_1\not=x_2$.

\begin{example}
Let $\bar{F}_R \in \RV_{-1}^\infty$, i.e., $\alpha=1$, and $\W \in \mathcal{W} = \{(w_1,w_2) \in [0,1]^2: w_1+w_2 =1\}$. Taking $\W=(W,1-W)$ then $W \in[0,1]$, with $\E(W) = 1/2$, is the random variable described by the $L_1$ \emph{spectral measure} \citep[e.g.][]{ColesTawn1991}. A simple example is the Gumbel or logistic spectral measure, which has Lebesgue density
\begin{align}
h(w) = \{w(1-w)\}^{1/\theta-2}\{w^{1/\theta} + (1-w)^{1/\theta}\}^{\theta-2}(1-\theta)/(2\theta) \qquad \theta\in(0,1), \label{eq:hlogistic}
\end{align}
$\chi_X = 2-2^{\theta}$ and $V_X(x_1,x_2) = (x_1^{-1/\theta}+x_2^{-1/\theta})^{\theta}$.
\end{example}

\subsection{Radial variable in Gumbel MDA}
\label{sec:Gumbel}
Suppose that $R$ is in the Gumbel MDA, with upper endpoint $r^\star\in (0,\infty]$ i.e.,
\[
 \lim_{t\to r^\star} \bar{F}_R(t+r/b(t))/\bar{F}_R(t) = e^{-r},
\]
where $b(t)$ is termed the auxiliary function. Such distributions can be expressed as
\begin{align}
 \bar{F}_R(r) = c(r)\exp\left\{-\int_z^{r} b(t) \d t\right\}, \label{eq:GumbelMDA}
\end{align}
where $z<r<r^\star$, $c(r) \to c>0$ as $r\to r^\star$, and $a=1/b$ is absolutely continuous with density $a'$ satisfying $\lim_{t\to r^\star}a'(t)=0$ \citep[e.g.][Chapter 3.3]{Embrechtsetal1997}. Several distributions in this domain have mass on $\mathbb{R}_-$, but we suppose here that $R$ is conditioned to be positive, which does not affect the tail behavior. If $r^\star=\infty$, we also have  \citep{Hashorva2012} that for any $\lambda>1, \rho\in\mathbb{R}$,
\begin{align}
 \lim_{r\to\infty} (rb(r))^\rho \bar{F}_R(\lambda r) / \bar{F}_R(r) = 0. \label{eq:DR}
\end{align}

\subsubsection*{Notation and assumptions for $\W$}
Suppose that $W_1 \stackrel{d}{=} W_2  \stackrel{d}{=} W \in[0,1]$ and $\nu\W = 1$. To this end, we assume that $\nu$ is a symmetric norm, i.e., $\nu(x,y)=\nu(y,x)$, and scaled to satisfy $\nu(x,y)\geq \max(x,y)$, such that the unit sphere of $\nu$ is contained in that of $\max$, with $\nu(b,1-b) = b$ for some $b\geq 1/2$. Let $\tau(z)=z/\nu(z,1-z) = 1/\nu(1,1/z-1)$. The random variable $Z = W_1 / (W_1+W_2) \in [0,1]$ has distribution symmetric about $1/2$, and satisfies
\begin{align}
 \W = (Z,1-Z) / \nu(Z,1-Z) = (\tau(Z),\tau(1-Z)). \label{eq:Wnu}
\end{align}
Define $I_{\nu} = [b_1,b_2] \subseteq [1/2,1]$ as the interval such that $\tau(z) = 1$ for all $z\in I_{\nu}$, and $\tau(z)<1$ for $z\not\in I_{\nu}$, and write $\tau(z)=\tau_1(z)$ for $z \in[0,b_1]$, $\tau(z)=1$ for $z \in[b_1,b_2]$, and $\tau(z)=\tau_2(z)$ for $z \in[b_2,1]$, with $\tau_1$ strictly increasing and $\tau_2$ strictly decreasing.
Figure~\ref{fig:nutaukappa} illustrates $\tau$ for a particular $\nu$; further illustrations are given in Appendix~\ref{app:addill}. 
We assume further that 
\begin{figure}
\centering
 \includegraphics[width=0.45\textwidth]{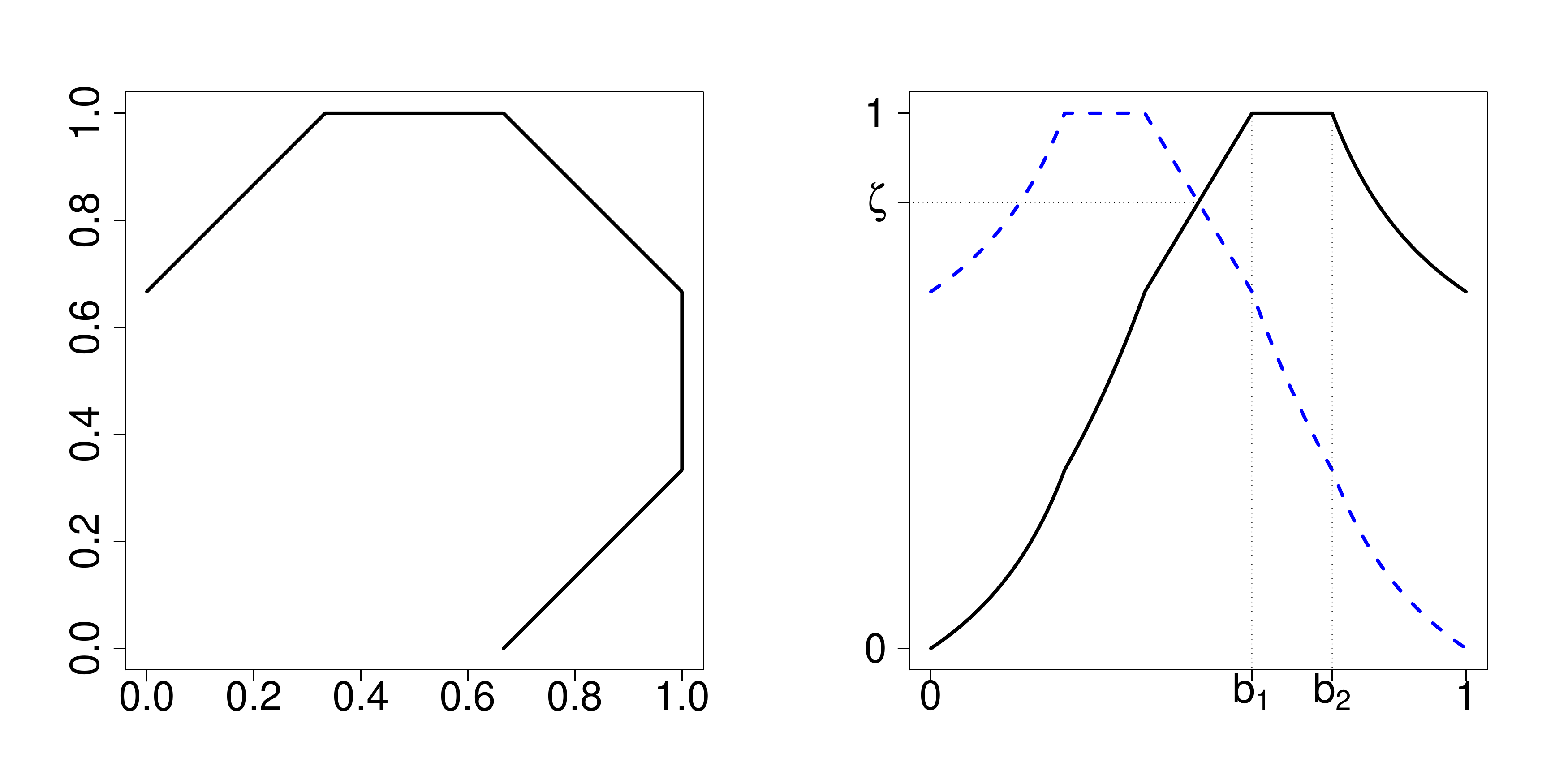}
  \includegraphics[width=0.45\textwidth]{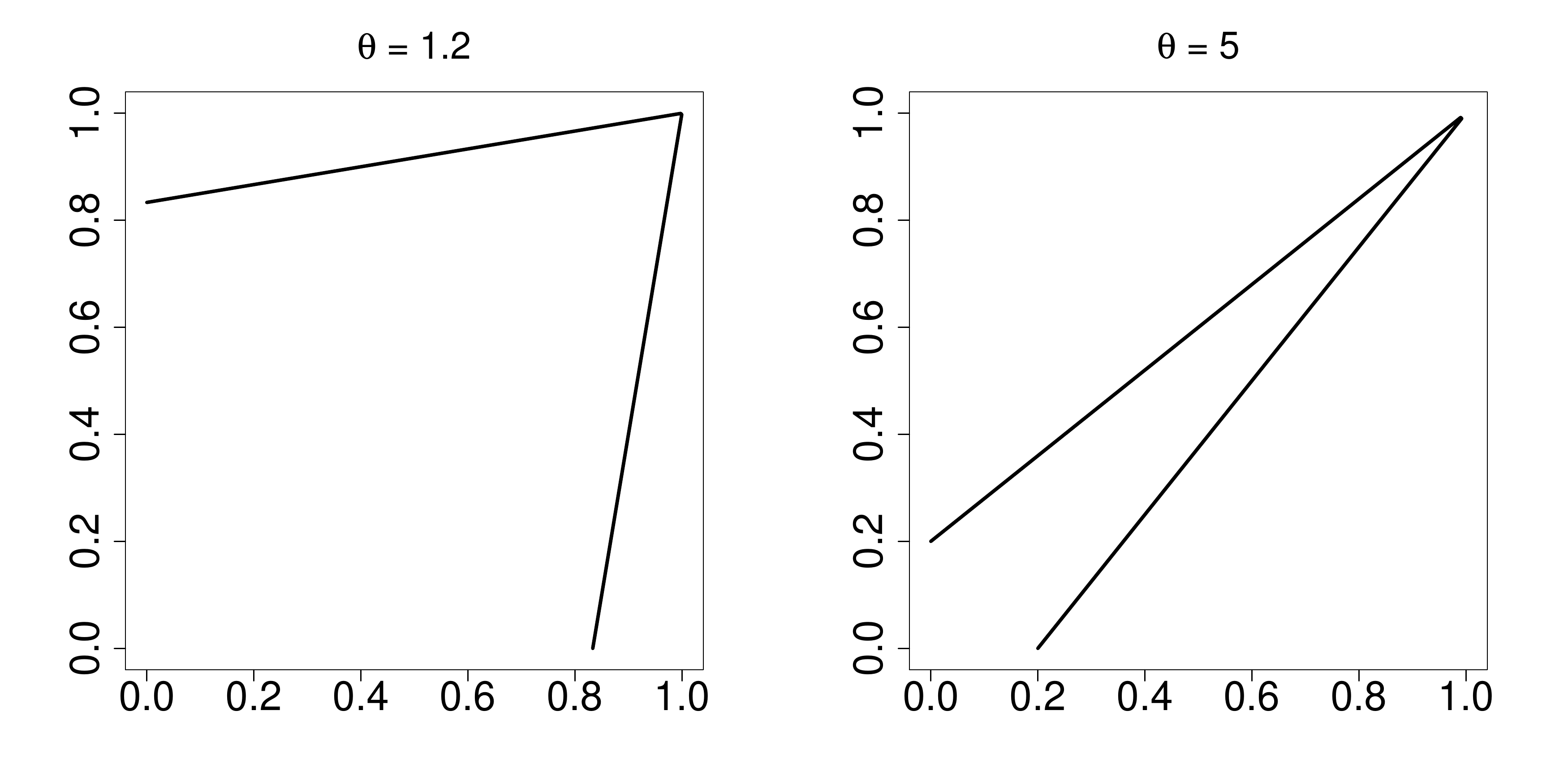}
 \caption{Left: the unit sphere for a particular norm $\nu$; centre-left: illustration of $\tau(z)$ (solid line) and $\tau(1-z)$ (dashed line) for the same $\nu$. Centre-right and right: illustration of the unit sphere of $\nu(x,y) = \theta \max(x,y) + (1-\theta)\min(x,y)$ for two different values of $\theta$.}
 \label{fig:nutaukappa}
\end{figure}

\begin{enumerate}
 \item[(Z1):] $Z$ has a Lebesgue density, $f_Z$, positive everywhere on $(0,1)$, and that its survival function is regularly varying at $1$, with $\bar{F}_Z(1-\cdot) \in \RV_{\alpha_Z}^0$, $\alpha_{Z}>0$, \label{ass:z1}
\end{enumerate}
and make the following mild regularity assumptions on the norm, $\nu$, or equivalently $\tau$:
\begin{enumerate}
 \item[(N1):]The function $\tau$ is twice (piecewise) continuously differentiable except for finitely many points, at which we only require existence of left and right derivatives of first and second order. \label{ass:n1}
 \item[(N2):] $\tau$ is regularly varying as it approaches 1 from either side, i.e., $1-\tau_1(b_1-\cdot) \in \RV_{1/\gamma_1}^0$, and, if $b_2<1$, $1-\tau_2(b_2+\cdot) \in \RV_{1/\gamma_2}^0$, $\gamma_j \in (0,1]$, $j=1,2$. We label $\gamma=\min(\gamma_1,\gamma_2)$ with $\gamma=\gamma_1$ if $b_2=1$.  
 \end{enumerate}
In practice, (N1) and (N2) are satisfied by a wide variety of commonly used norms, and the upper limit of $\gamma_1,\gamma_2\leq 1$ in (N2) is a consequence of convexity of the norm $\nu$; see Lemma~\ref{lem:gamleq1} in Appendix~\ref{app:AProofs}. Note that $\tau(z)\lessgtr\tau(1-z) \Leftrightarrow z \lessgtr 1-z$, so that 
\begin{align*}
\min(\tau(z),\tau(1-z)) = \begin{cases}
                           \tau(z), \qquad & z\in[0,1/2],\\
                           \tau(1-z), \qquad & z\in[1/2,1].
                          \end{cases}  
\end{align*}
 Finally, denote $\zeta= \tau(1/2)\in [1/2,1]$, so that $W_\wedge = \min(\tau(Z),\tau(1-Z)) \in [0,\zeta]$.

\begin{prop}[$R$ in Gumbel MDA] \label{prop:RGumbel}
 Assume $\bar{F}_R$ satisfies~\eqref{eq:GumbelMDA} and that (N1), (N2) and (Z1) hold. Then:
 \begin{enumerate}
  \item If $\zeta<1$, $\chi_X = 0$ and $\eta_X = \lim_{x\to r^\star} \log \bar{F}_R(x)/\log \bar{F}_R(x/\zeta)$, which is defined only for $r^\star = \infty$. \label{RGumbelitem1}
  \item If $\zeta=1$, then $\eta_X = 1$. Further, $b_1=1/2$ and
  \begin{align*}
   \chi_X = \begin{cases}
             0 & \mbox{if } b_2>1/2, i.e., \P(W=1)>0,\\
             \displaystyle\frac{2\tau_2'(1/2_+)}{\tau_2'(1/2_+)-\tau_1'(1/2_-)}  & \mbox{otherwise}.
            \end{cases}
  \end{align*}
 \end{enumerate}
\end{prop}
We observe that asymptotic independence arises for $\zeta<1$, with the residual tail dependence coefficient determined by the properties of $\bar{F}_R$. The following corollary covers an important subclass of distributions in the Gumbel MDA.
\begin{cor}
\label{cor:RWeib}
If $\zeta<1$ and $-\log\bar{F}_R  \in\RV_{\delta}^\infty$, $\delta \geq 0$, then $\eta_X = \zeta^{\delta}$.
\end{cor}
If $\delta = 0$, as in the case of log-normal $R$, then $\eta_X=1$. Another possibility in the Gumbel MDA is $-\log \bar{F}_R(x) \sim \exp(x)$, as in the reverse Gumbel distribution, for which $\eta_X=0$. If $r^\star<\infty$, then when $\zeta<1$ the upper endpoint of $X_\wedge$ is less than that of $X_1$, so $\eta_X$ is not defined.

If $\zeta=1$, then one has asymptotic independence only if $\P(W=1) = \P(Z\in I_{\nu})>0$, which is equivalent to $\nu(x,y)$ behaving locally like the $L_\infty$ norm around the point $x=y$. If $\zeta=1$ and $\P(W=1)=0$, then $b_1=b_2=1/2$, and the ``pointy" shape of such norms induces asymptotic dependence. The following example illustrates this case.

\begin{example}
 \label{ex:AD}
 Let $\nu(x,y) = \theta \max(x,y) + (1-\theta)\min(x,y)$, $\theta>1$; see Figure~\ref{fig:nutaukappa} for an illustration. We have $\zeta=1$ and $b_1=b_2=1/2$, so we can calculate $\chi_X>0$ using Proposition~\ref{prop:RGumbel}, by evaluating $\tau_1'(1/2_+) = 4\theta$ and $\tau_2'(1/2_-) = -4(\theta-1)$. This yields $\chi_X = 2(\theta-1)/(2\theta-1)$, which is an increasing function of $\theta$; in particular $\chi_X \to 0_+$ as $\theta \to 1_+$ and $\chi_X \to 1_-$ as $\theta \to +\infty$.
\end{example}

\begin{example}
 \label{ex:Lp}
  
 Let $\nu(x,y) = (x^p+y^p)^{1/p}$, $p\geq1$, so that  $\zeta=2^{-1/p}<1$. In this case $b_1=b_2=1$ and $1-\tau^{-1}(1-s) = ps^{1/p}[1+o(1)]$, i.e., $\gamma=1/p$. For any $Z$ satisfying (Z1), the conditions of Proposition~\ref{prop:RGumbel} are satisfied, and $\eta_X = \lim_{x\to r^\star}\log\bar{F}_{R}(x)/\log\bar{F}_{R}(2^{1/p}x)$. As a concrete example, if $\bar{F}_R(x) = \exp(-x^\delta)$, then $\eta_X = 2^{-\delta/p}$ by Corollary~\ref{cor:RWeib}.
\end{example}

\subsection{Radial variable in negative Weibull MDA}
\label{sec:Rnwei}
Suppose that $R>0$ is in the negative Weibull MDA with upper endpoint $r^\star > 0$, i.e,
\begin{align*}
 \bar{F}_R(r^\star - s) = \ell(s) s^{\alpha_R},  \qquad \ell \in \RV_{0}^0,~ \alpha_R>0;  \label{eq:NegWeibMDA}
\end{align*}
equivalently $\bar{F}_R(r^\star - \cdot) \in \RV_{\alpha_R}^0$. Note that the distribution of $R$ cannot have a point mass on $r^\star$. The general assumptions for $\W$ are the same as in Section \ref{sec:Gumbel}.

\begin{prop}[$R$ in negative Weibull MDA]\label{prop:RNegWeib}
 Assume $\bar{F}_R(r^\star - \cdot) \in \RV_{\alpha_R}^0$ and that  (N1), (N2) and (Z1) hold. Then:
 \begin{enumerate}
  \item If $\zeta<1$, $\chi_X = 0$ and $\eta_X$ is not defined.
  \item If $\zeta=1$, then $b_1=1/2$ and
  \begin{align*}
   \chi_X = \begin{cases}
             0 & \mbox{if }  \P(W=1)>0,\\ 
                \displaystyle \frac{2\tau_2'(1/2_+)}{\tau_2'(1/2_+)-\tau_1'(1/2_-)} & \mbox{otherwise},
            \end{cases}
               \eta_X = \begin{cases}
             \frac{\alpha_R}{1+\alpha_R} & \mbox{if } \P(W=1)>0,\\
             1 & \mbox{otherwise}.
            \end{cases}
  \end{align*}
  \end{enumerate}
\end{prop}

\begin{example}
\label{ex:max}
Let $\nu(x,y) = \max(x,y)$, so that $\zeta=1$, $b_1=1/2$, $b_2=1$ and $\P(W=1) = 1/2$. For $\bar{F}_R = (1+\lambda r)_+^{-1/\lambda}$, $\lambda <0$ with $r^\star = -1/\lambda$, Proposition~\ref{prop:RNegWeib} gives  $\eta_X=(1-\lambda)^{-1}$, noting $\alpha_R=-1/\lambda$. This represents (part of) a model given in \citet{Wadsworthetal2017}.
\end{example}
Table~\ref{tab:simu-clayton} summarizes the tail dependence for $\X$ using the norms from Examples~\ref{ex:AD}--\ref{ex:max}, under different tail behaviors for $R$.

\begin{table}[h] 
\begin{center}
\begin{tabular}{|l|rr|rr|rr|rr|rr|} \hline
  & \multicolumn{2}{c|}{$\nu = L_p$ norm, $p\geq 1$} & \multicolumn{2}{c|}{$\nu = L_\infty$ norm} & \multicolumn{2}{c|}{$\nu=\theta\max+(1-\theta)\min$, $\theta\geq1$}\\
                                  Radial variable $R$ & $\chi_X$ & $\eta_X$    & $\chi_X$ & $\eta_X$ & $\chi_X$ & $\eta_X$ \\ \hline
Regularly varying &&&&&&\\
\qquad $\bar{F}_R \in \RV_{-\alpha}^{\infty}$, $\alpha>0$ & eq.~\eqref{eq:Rrvchi} & 1  &  eq.~\eqref{eq:Rrvchi} & 1   &  eq.~\eqref{eq:Rrvchi} & 1   \\ 
Log-normal &&&&&&\\
\qquad $-\log\bar{F}_R(r) \sim k(\log r)^2$, $k>0$ & 0 & 1  &  0   & 1 &  $2(\theta-1) / (2\theta-1)$ & 1\\ 
Weibull-like &&&&&&\\
\qquad  $-\log\bar{F}_R(r) \sim k r^\delta$, $k,\delta > 0$ & 0 & $2^{-\delta/p}$  &    0 &  1  &  $2(\theta-1) / (2\theta-1)$   & 1 \\
\ a)  exponential ($\delta=1$) & 0 & $2^{-1/p}$   &   0 & 1  &  $2(\theta-1) / (2\theta-1)$ & 1  \\
\ b)  normal ($\delta=2$)   & 0 & $2^{-2/p}$  &   0 & 1  & $2(\theta-1) / (2\theta-1)$   & 1  \\
Log of exponential &&&&&&\\
\qquad  $-\log \bar{F}_R(r)\sim k \exp(r)$, $k>0$ & 0 & 0 &    0 & 1  & $2(\theta-1) / (2\theta-1)$  & 1  \\
Exponential behavior at $r^\star <\infty$ &&&&&& \\
\qquad  $\bar{F}_R(r^\star - 1/r) \sim k\exp(-r)$, $k>0$ & 0 & ND  &    0& 1   &  $2(\theta-1) / (2\theta-1)$  & 1\\
Negative Weibull &&&&&& \\
\qquad  $\bar{F}_R(r^\star - \cdot)  \in \RV_{\alpha}^0$, $\alpha>0$ & 0 & ND  &    0 & $\displaystyle \alpha/(1 + \alpha)$ &  $2(\theta-1) / (2\theta-1)$  & 1\\
\ a) uniform ($r^\star=1, \alpha = 1$) & 0 & ND  &   0 & 1/2  &  $2(\theta-1) / (2\theta-1)$ & 1  \\ \hline
\end{tabular}
\end{center}
\caption{Values of $\chi_X$ and $\eta_X$ for $\X=R\W$ with different tail decay rates of the variable $R$ and angular variables defined on $L_p$, $L_\infty$ norms, and the norm of Example~\ref{ex:AD}. ND = not defined. } \label{tab:simu-clayton}
\end{table}

\section{Unconstrained angular variables}
\label{sec:Unconstrained}

We now treat the case where the support $\mathcal{W}$ is two-dimensional. As noted in Section~\ref{sec:Introduction}, there are cases where $\W$ itself might have a random scale representation, and by redefining the scaling variable we get back to the situation of one-dimensional $\mathcal{W}$. We thus focus on constructions where this is not necessarily the case. To avoid additional complications we assume throughout this section that $W_1$ and $W_2$ share the common marginal distribution $F_W$. We also generally assume that the tail dependence coefficient $\chi_W$ and the residual tail dependence coefficient $\eta_W$ of $\W$ exist, although some results may still be obtained with the latter undefined.

In Section~\ref{sec:Constrained}, the constraints imposed by $\mathcal{W}$ being a unit sphere gave bounded marginal distributions for $W_j$, $j=1,2$, and deterministic dependence between $\W$. For two-dimensional $\mathcal{W}$, the variety of marginal and dependence behaviors possible for $\W$ means that systematic characterization according only to the MDA of $R$ is more difficult.
In fact, we need to consider different tail decays of both the radial variable $R$ and the angular variable $W$ since the combination of the two is crucial to classify the extremal dependence of $\X = R \W$.  We focus on some interesting sub-classes that still incorporate a wide variety of structures and cover most of the parametric univariate distributions available for $R$ and $W$.

This section is structured according to the tail heaviness assumed for $R$, $W$, or both of them. In decreasing order we consider distributions with superheavy tails, regularly varying distributions, distributions of log-Weibull and Weibull type, and finally distributions with finite upper endpoint in the negative Weibull domain of attraction.
Table~\ref{tab:unconstrained} summarizes the general results developed in the following, and Table \ref{tab:indep} contains the extremal dependence coefficients for all combinations of tail decays of $R$ and $W$ for the specific, yet interesting example where $W_1$ and $W_2$ are independent.

\begin{table}[h]
	\centering
	\begin{tabular}{|l|ccc|} \hline
Radius $R$ &  additional assumptions &  $\chi_X$ & $\eta_X$ \\
\hline 
Superheavy tails &  & &  \\
\ a) $\overline{F}_{W}(x)/\overline{F}_{R}(x)\rightarrow c$ & Prop. \ref{prop:RSHT} & $\frac{1+c\chi_W}{1+c}$ &  $1$ \\
\ b) $\overline{F}_R=o(\overline{F}_{W})$  & $\chi_W>0$ & $\chi_W$ & $1$  \\
 & $\chi_W=0$, $\overline{F}_R(x)\leq C\overline{F}_{W_\wedge}(x)$ & $0$ & $\eta_W$  \\
 & $\chi_W=0$, $\overline{F}_{W_\wedge}=o(\overline{F}_R)$  & $0$ &  \eqref{eq:etalimfrac} \\
\hline
$\mathrm{RV}_{-\alpha_R}^\infty$ & & &  \\
\ a) $\E (W^{\alpha_R+\varepsilon})<\infty$ & $\P(W>0)=1$ & \eqref{eq:Rrvchi} & $1$ \\
\ b) $\bar{F}_{W}\in \RV_{-\alpha_W}^{\infty}$ &  &    &  \\
\ \ (i) $\alpha_R>\alpha_W$  &
&   $\chi_W$  & \eqref{eq:etaWRV}  \\
\ \ (ii) $\alpha_R=\alpha_W$  & Prop. \ref{prop:samealpha} 
&  Prop. \ref{prop:samealpha}  & 1\\
 \hline
$\mathrm{LWT}_{\beta_R>1}$ & $F_W,F_{W_\wedge}\in\mathrm{LWT}_{\beta_R}$  & $\chi_W$ &  \eqref{eq:lightS}  \\
\hline
$\mathrm{WT}_{\beta_R}$ & $F_{W}\in\mathrm{WT}_{\beta_W},F_{W_\wedge}\in\mathrm{WT}_{\beta_{W_\wedge}}$ & Prop. \ref{prop:weib} &   Prop.~\ref{prop:weib}\\\hline
Gumbel & Prop.~\ref{prop:negwei}\eqref{item9.1} & $\chi_W$ &   $1$\\\hline
Negative~Weibull &  &  &  \\
 & Prop.~\ref{prop:negwei}\eqref{item9.2} & $\chi_W$ &   $\eta_W$\\
& Prop.~\ref{prop:negwei}\eqref{item9.3} & $0$ &  \eqref{eq:etanegweib}\\
\hline

\end{tabular}
\caption{Tail dependence summaries $\chi_X$ and $\eta_X$ for $\X=R\W$ with different tail decay rates of the radial variable $R>0$ and unconstrained variables $W_1\stackrel{d}{=}W_2$. 
} 
\label{tab:unconstrained}
\end{table}

\begin{table} 
\begin{center}
\hspace*{-3.5em}
\begin{tabular}{|l|c|c|c|c|c|} \hline
  \multicolumn{1}{|r|}{Angle $W$} & {Super-heavy} & {Reg.~varying} & log-Weibull $(\beta_W > 1)$ & Weibull & Neg.~Weibull  \\ 
                                 Radius $R$ & & & & &\\ \hline
Super-heavy &  $ \displaystyle  \chi_X = (1+c)^{-1}$  &  $\chi_X = 1$  & $\chi_X = 1$ & $\chi_X = 1$ & $\chi_X = 1$\\
 &  $\eta_X$: Prop.~\ref{prop:RSHT}\eqref{item2}& & & &\\ \hline
Reg.~varying & * & 
$\alpha_R < \alpha_W : \chi_X = \eqref{eq:Rrvchi} > 0$
& $\chi_X = \eqref{eq:Rrvchi} > 0$  & $\chi_X = \eqref{eq:Rrvchi} > 0$ & $\chi_X = \eqref{eq:Rrvchi} > 0$\\
&& $\alpha_R = \alpha_W$: Prop.~6  &&&\\ 
&& $\alpha_W < \alpha_R < 2\alpha_W$:&&&\\
&& $\quad \eta_X = \alpha_W/\alpha_R$ &&&\\
&& $\alpha_R > 2 \alpha_W$: $\eta_W = 1/2$ &&& \\\hline
log-Weibull&* & * & $\beta_R = \beta_W: \eta_X = \eqref{eq:lightS}$  & unknown & $\chi_X = 0$\\
($\beta_R > 1$) & &  &  &&$\eta_X = 1$\\\hline
Weibull &*& * & unknown & $\eta_X = 2^{-\beta_R/(\beta_R + \beta_W)}$& $\chi_X = 0$ \\
&& &  && $\eta_X = 1$ \\\hline
Neg.~Weibull &* &* & * & * &  $\displaystyle  \eta_X = \frac{\alpha_W + \alpha_R}{2\alpha_{W} + \alpha_R}$\\
\hline
\end{tabular}
\end{center}
\caption{The values of $\chi_X$ and $\eta_X$ for $\X=R\W$ with $W_1,W_2 \stackrel{d}{=} W$ independent, 
  with different tail decay rates of the radial and angular variables. The *'s indicate that multiplication with $R$ does not change the tail dependence of $\W$, i.e., $\chi_X=\chi_W= 0$ and $\eta_X=\eta_W= 1/2$. The combinations of Weibull and log-Weibull tails remain open problems.  
}\label{tab:indep}
\end{table}

\subsection{Superheavy-tailed variables} 
\label{sec:RSHT}
Suppose that $R$ or $W$ is superheavy-tailed, i.e., $\log R$ or $\log W$ is heavy-tailed. This case naturally arises when considering random location constructions $\log R+(\log W_1,\log  W_2)$; we thus further assume $W>0$ so that $\log W_j$, $j=1,2$, are well defined.  

\begin{prop}[Superheavy-tailed variables]\label{prop:RSHT} \ 
\begin{enumerate}
	\item  \label{item1} If ${F}_{\log R}\in\mathrm{CE}_0$ and $\overline{F}_{W}(x)/\overline{F}_{R}(x)\rightarrow c\geq 0$ as $x\rightarrow\infty$,  
   then $\eta_X=1$ and
	\begin{equation}\label{eq:chiSHT}
	\chi_X=(1+c\,\chi_W)/(1+c)>0.
	\end{equation} 
    \item \label{item2} If $F_{\log W}\in\mathrm{CE}_0$ and $\overline{F}_R=o(\overline{F}_{W})$, then $\chi_X=\chi_W$.  If further $F_{\log R}\in\mathrm{CE}_0$  and 
    \begin{enumerate}
    \item  $F_{\log W_\wedge}\in\mathrm{CE}_0$ with $\overline{F}_R(x)/\overline{F}_{W_\wedge}(x)\leq C$ for a constant $C>0$ as $x\rightarrow\infty$, then $\eta_X=\eta_W$; \label{item2a}
    \item   \label{item2b} $\overline{F}_{W_\wedge}=o(\overline{F}_R)$, then, provided the limit exists,
    \begin{equation*}\label{eq:etalimfrac}
	\eta_X=\lim_{x\rightarrow\infty} \log \overline{F}_{W}(x) / \log \overline{F}_{R}(x).
	\end{equation*}
    \end{enumerate}
\end{enumerate}\end{prop}

\begin{example}[Independence model]\label{ex:indep1}
  In order to illustrate the results of this section, we consider the example 
  where $R$, $W_1$ and $W_2$ are independent.
  In this case $\bar F_{W_\wedge} = (\bar F_W)^2$, $\chi_W = 0$ and $\eta_W = 1/2$.
If $F_{\log R} \in\mathrm{CE}_0$ and $\overline{F}_{W}(x)/\overline{F}_{R}(x)\rightarrow c\geq 0$ as $x\rightarrow\infty$, then Proposition~\ref{prop:RSHT}\eqref{item1} yields asymptotic dependence in $\X$ with $\chi_X = (1+c)^{-1}$. Hence if $F_{\log R}\in\mathrm{CE}_0$ and $W$ has a comparable tail, then $\chi_X \in (0,1)$, whilst if $W$ has tail lighter than superheavy, then $\chi_X=1$.
 On the other hand, if $W$ is superheavy-tailed with $F_{\log W}\in\mathrm{CE}_0$,  then $W_\wedge$ is also superheavy-tailed. If $R$ has lighter tail than $W_\wedge$, then $\overline F_{R}(x) \leq C\overline F_{W_\wedge}(x)$ for large $x$ with some $C>0$, and by Proposition~\ref{prop:RSHT}\eqref{item2a} we have $\chi_X=0$ 
  and $\eta_X = 1/2$. The case $\eta_X\not=\eta_W$ may arise when the tail of $W$ dominates the tail of $R$ and the tail of $R$ dominates the tail of $W_\wedge$. For a concrete example, consider log-Weibull tails in $W$ and $R$ with $\overline{F}_W(\exp (x)) \sim \exp(-x^\beta)$ and $\overline{F}_R(\exp (x)) \sim \exp(-(1+c)x^\beta)$, where $0<\beta,c<1$. Then, $\eta_X=(1+c)^{-1}$ according to Proposition~\ref{prop:RSHT}\eqref{item2b}.  The first row and column of Table \ref{tab:indep} summarize these results. 
\end{example}

\subsection{Regularly varying variables}
\label{sec:RRV}

In this section we consider the case where $R$, $W$ or both of them are regularly varying.
When $R$ is regularly varying with index $\alpha_R > 0$ and $\E(W^{\alpha_R+\varepsilon}) < \infty$ for some $\varepsilon >0$, then the tail dependence coefficient $\chi_X$ is as described in Proposition~\ref{prop:Frechet_chi} in Section~\ref{sec:Frechet}. We firstly consider the case where $W$ is regularly varying with index $\alpha_W> 0$ 
and $R$ is lighter tailed, i.e., either also regularly varying with $\alpha_R > \alpha_W$ or even lighter tailed such as distributions in the Gumbel or negative Weibull domain of attraction. Secondly, we study the case where both $R$ and $W$ are regularly varying with the same index $\alpha_W = \alpha_R$, which turns out to be particularly involved and which requires additional assumptions.

\begin{prop}[$W$ regularly varying with $R$ lighter]\label{prop:WRV}  
Let $\bar{F}_W \in\RV_{-\alpha_W}^\infty$, $\alpha_W \geq 0$, and suppose that either $\bar{F}_R\in\RV_{-\alpha_R}^\infty$ with $\alpha_R > \alpha_W$, or $R$ is in the Gumbel or negative Weibull domain of attraction; denote the latter case by $\alpha_R = +\infty$. Then $\chi_X = \chi_W$ and 
\begin{align}
 \eta_X = \begin{cases}
     \alpha_W/\alpha_R, \quad & \text{if }   \alpha_R  < \alpha_W / \eta_W, \eta_W=0 \mbox{ or $\eta_W$ not defined}, \\
      \eta_W, \quad & \text{if }  \alpha_R  > \alpha_W / \eta_W \text{ or } \alpha_R = +\infty.
   \end{cases} \label{eq:etaWRV}
\end{align}
\end{prop}

The case where $R$ and $W$ are regularly varying with the same index $\alpha>0$ leads to various scenarios for the extremal dependence in $(X_1,X_2)$. Since  $\overline{F}_R,\overline{F}_W\in \RV_{-\alpha}^\infty$ is equivalent to $F_{\log R},F_{\log W}\in\mathrm{ET}_{\alpha}$, and $\mathrm{ET}_\alpha$ is closed under convolutions, we have that $F_{\log X}\in\mathrm{ET}_\alpha$ \citep[Lemma 2.5]{Watanabe.2008}.

\begin{prop}[Regularly varying $R$ and $W$ with the same index] \label{prop:samealpha} 
	Let $\bar F_{R},\bar F_{W}\in\RV_{-\alpha}^\infty$ with $\alpha>0$.
    Then $\eta_X=1$, and we have the following:
 \begin{enumerate}
 	\item \label{6item1} If $F_{\log R}\in\mathrm{CE}_\alpha$, and if $\overline{F}_{W}(x)/\overline{F}_{R}(x)\rightarrow c \geq 0$ as $x\rightarrow\infty$, then 
 		\begin{equation*}
 		\chi_X=
        \frac{\E(W_\wedge^\alpha)+c\,\chi_W\,\E(R^\alpha)}{\E(W^\alpha)+c\,\E(R^\alpha)}.
 		\end{equation*}
 	\item \label{6item2} If $F_{\log W}\in\mathrm{CE}_\alpha$ and $\overline{F}_{R} = o(\overline{F}_{W})$, then $\chi_X=\chi_W$. 
 \item \label{6item3} Let $F_{\log R}\in \mathrm{ET}_{\alpha,\beta_R}$ with $\beta_R\geq -1$ and $\E(R^\alpha)=\infty$ if $\beta_R=-1$,   and let $F_{\log W}\in\mathrm{ET}_{\alpha,\beta_W}$. 
\begin{enumerate}
 \item If $\chi_W>0$ and if either $\beta_W>-1$ or $\beta_W=-1$ and $\E(W^\alpha)=\infty$, then $\chi_X=\chi_W$. \label{item:6a}
 \item   If $\chi_W\geq 0$ and if either $\beta_W<-1$ or $\beta_W=-1<\beta_R$ and $\E(W^\alpha)<\infty$, then 
 $\chi_X= \E (W_\wedge^\alpha) / \E (W^\alpha)$. 
 \item If $\beta_R>-1$, $\beta_W>-1$ and $\E(W_{\wedge}^{\alpha+\varepsilon})<\infty$ for some $\varepsilon>0$, then $\chi_X=0$. \label{item:prop6wminlighter}
 \end{enumerate}
 \end{enumerate}
\end{prop}

\begin{remark}
Proposition~\ref{prop:samealpha} contains certain results of Proposition~\ref{prop:RSHT} as a special case when allowing for $\alpha=0$.  Proposition~\ref{prop:samealpha}\eqref{6item1},\eqref{6item2} treats the case of convolution-equivalent tails in $\log R$ or $\log W$, which are relatively light since the expectation $\E(R^\alpha)$ or $\E(W^\alpha)$ is finite; notice that $\mathrm{ET}_{\alpha,\beta}$ with $\beta<-1$ is an important subclass of $\mathrm{CE}_\alpha$, see Lemma~2.3 of \citet{Pakes.2004}. The tail of $R$ is not dominated by that of $W$ in Proposition~\ref{prop:samealpha}\eqref{6item1}, while it is dominated in Proposition~\ref{prop:samealpha}\eqref{6item2}. Proposition~\ref{prop:samealpha}\eqref{6item3} shifts focus to relatively heavy tails in $R$ with $\E(R^\alpha)=\infty$, such as the gamma tails of $\mathrm{ET}_{\alpha,\beta}$ with $\beta>-1$. 
\end{remark}

\begin{example}[Independence model]\label{ex:indep2}
  We continue Example \ref{ex:indep1}, where now $W_1$ and $W_2$ are independent and 
  regularly varying with index $\alpha_W$, and $R$ is regularly varying with index $\alpha_R$.   If $\alpha_R < \alpha_W$, then we have asymptotic dependence with $\chi_X$ given in \eqref{eq:Rrvchi}. The same is true in general when $W$ has a lighter tail than $R$ that is not necessarily regularly varying.
By Proposition~\ref{prop:WRV}, if $\alpha_W < \alpha_R < 2\alpha_W$, then $\X$ is asymptotically independent with
  $\eta_X = \alpha_W / \alpha_R$, and if $\alpha_R > 2 \alpha_W$, then $\eta_X = 1/2$. In general, if $R$ is even lighter tailed, it does not affect the coefficients $\chi_X$ and $\eta_X$.  
  If $\alpha_R=\alpha_W=\alpha$, then $\eta_X=1$ and different scenarios for $\chi_X$ arise depending on the distributions
  of $R$ and $W$: see Proposition~\ref{prop:samealpha}. Suppose $\alpha>0$; since $\bar{F}_{W_\wedge}\in\RV_{-2\alpha}^\infty$, $\bar{F}_{X_\wedge}\sim \E (W_\wedge^{\alpha}) \bar{F}_R$, and so $\chi_X=\E (W_\wedge^{\alpha})/c>0$ if $\bar{F}_{X}\sim c \bar{F}_{R}$ for some $c>0$. In particular, $c=\E (W^\alpha)$ if $F_W\in\mathrm{ET}_{\alpha,\beta_W}$ and $F_R\in\mathrm{ET}_{\alpha,\beta_R}$ with $\beta_W<\beta_R$ and $\beta_W<-1$. This fills the second row and column of Table \ref{tab:indep}.
\end{example}

\subsection{Log-Weibull-type variables}

In this and the following section we concentrate on radial and angular variables
in the Gumbel domain of attraction. Due to the large variety of distributions in this domain
we consider subsets that include the most commonly used distribution families. We firstly study the case where both $R$ and $W$ are  log-Weibull-tailed; equivalently, 
$\log R$ and $\log W$ are Weibull-tailed. We recall that 
a random variable $Y$ is log-Weibull-tailed if
\begin{equation}\label{eq:logweib}
  \bar{F}_Y(y) = \ell(\log y) (\log y)^\gamma \exp\left(-\alpha (\log y)^{\beta}\right), \qquad \ell \in \RV_{0}^\infty, \gamma \in \mathbb{R}, \alpha,\beta>0, 
\end{equation} 
and we write $F_Y\in \mathrm{LWT}_\beta$. 
The parameter $\beta$ has the predominant influence on the tail decay rate, with  $\beta=1$ if and only if $\bar{F}_Y\in \RV_{-\alpha}^\infty$, while $\beta<1$ gives superheavy-tailed $F_Y$, and $\beta>1$ yields rapid variation of $Y$, i.e., $\bar{F}_Y\in \RV_{-\infty}^\infty$.
In the following, we denote the $\beta$-parameters of $R$ and $W$  by $\beta_R$ and $\beta_W$, respectively. The superheavy-tailed case, $\beta_R<1$ or $\beta_W<1$, is already covered by Section \ref{sec:RSHT}, and the case of regularly varying tails with $\beta_R=1$ or $\beta_W=1$ is treated in Section \ref{sec:RRV}.

We therefore study the remaining case $\beta_R >1$ and $\beta_W >1$, which encompasses important distributions such as the log-Gaussian. 
As in Section \ref{sec:RSHT}, it is more intuitive to consider the random location construction $\log R+\log \W$, where we can apply convolution-based results.
When independent heavy-tailed summands are involved in the convolution, typically only one of the values of summands has a dominant contribution to a high values of the sum, resulting in relatively simple formulas; see Section~\ref{sec:RSHT}.
On the contrary, in the light-tailed set-up all summands may contribute significantly when high values arise in the sum, rendering the tail analysis more intricate. Only relatively few general results on convolutions with tails lighter than exponential are available in the literature. The following lemma will be useful for this and the next section.

\begin{lemma}\label{lem:weib} 
Let $(W_1,W_2)$ be a random vector with $W_j \sim F_W, j=1,2$, such that
both $F_W\in\mathrm{WT}_{\beta_W}$ and $F_{W_\wedge}\in\mathrm{WT}_{\beta_{W_\wedge}}$, with other 
parameters also indexed by $W$ and $W_\wedge$. 
\begin{enumerate}\label{enum3}
	\item \label{chiwlwt} If $\beta_{W_\wedge} =\beta_W$, $\alpha_{W_\wedge}=\alpha_W$, then, provided the limit exists, $\chi_W = \lim_{x\rightarrow \infty} \frac{\ell_{W_\wedge}(x)x^{\gamma_{W_\wedge}}}{\ell_{W}(x)x^{\gamma_{W}}}.$  
	\item If $\beta_{W_\wedge} =\beta_W$, then $\eta_W =\alpha_W/\alpha_{W_\wedge}$.
	\item If $\beta_{W_\wedge}>\beta_W$, then $\log \bar{F}_{W}=o(\log \bar{F}_{W_\wedge})$, and $\eta_W$ is not defined.
\end{enumerate}
\end{lemma}

\begin{remark}
The proof of Lemma~\ref{lem:weib} is straightforward from \eqref{eq:logweib}. It also covers the case where $W$ and $W_\wedge$ are
log-Weibull-tailed, since the
tail and residual tail dependence coefficients are invariant under monotonic marginal transformations.
\end{remark}
We consider the set-up where the components $R$, $W$ and $W_\wedge$ are log-Weibull-tailed with the same coefficient $\beta > 1$ and a simplified form of the slowly varying function $\ell$ by assuming that it is asymptotically constant, i.e., $\ell(x) \sim c>0$.

 \begin{prop}[Light-tailed random location with $F_R\in\mathrm{LWT}_\beta$, $\beta>1$]\label{prop:lightS}
Suppose that $F_R, F_W, F_{W_\wedge} \in \mathrm{LWT}_\beta$ with possibly different parameters $\alpha,\gamma$ indexed by the corresponding $R$, $W$ and $W_\wedge$, but where $\beta = \beta_R = \beta_W = \beta_{W_\wedge} > 1$. Assume that the slowly varying functions $\ell$ behave asymptotically like positive constants. 

\begin{enumerate}
 \item If $\chi_W>0$, then $\chi_X=\chi_W>0.$
 \item If $\chi_W=0$, then $\chi_X=0$ and
 \begin{equation}\label{eq:lightS}
\eta_X=\eta_W\times \left(\frac{\alpha_{W_\wedge}^{1/(\beta-1)}+\alpha_R^{1/(\beta-1)}}{\alpha_{W}^{1/(\beta-1)}+\alpha_R^{1/(\beta-1)}}\right)^{\beta-1},
\end{equation}
where $\eta_W = {\alpha_W}/{\alpha_{W_\wedge}}$, and $\eta_X=\eta_W$ if $\alpha_{W}=\alpha_{W_\wedge}$.
\end{enumerate}
 \end{prop}

\begin{example}[Gaussian factor model]\label{ex:gfm}
Suppose that $\log R$ is univariate standard Gaussian and that $\log \W$ is bivariate standard Gaussian, independent of $R$ and with Gaussian correlation $\rho_W\in(-1,1]$.  Then we have log-Weibull tails with parameters $\beta_R=\beta_W=\beta_{W_\wedge}=2$, $\alpha_R=\alpha_W=1/2$ and $\alpha_{W_\wedge}=1/(1+\rho_W)$ (see Example~\ref{ex:gsm}).  Applying \eqref{eq:lightS} gives $\eta_X=\eta_W\times (3+\rho_W)/(2(1+\rho_W))=(3+\rho_W)/4$.
\end{example}

\begin{example}[Independence model]\label{ex:indep3}
  As in Examples \ref{ex:indep1} and \ref{ex:indep2} we let $R$, $W_1$ and $W_2$ be independent, and we now assume that they are log-Weibull-tailed with equal $\beta$ parameter. 
  By independence, $F_{W_\wedge} \in\mathrm{LWT}_{\beta}$ with $\alpha_{W_\wedge} = 2\alpha_W$.
  Proposition~\ref{prop:lightS} gives $\chi_X=0$ with $\eta_X$ calculated by formula \eqref{eq:lightS}.
\end{example}

\subsection{Weibull-type variables}

We now consider the case where $R$ and $W$ follow a Weibull-type distribution, a rich class in the Gumbel MDA. 
Recall that a variable $Y$ is of Weibull-type, $F_Y \in \mathrm{WT}_\beta$, if
\begin{equation}\label{eq:weib}
  \bar{F}_Y(y) = \ell(y) y^\gamma \exp\left(-\alpha y^{\beta}\right), \qquad \ell \in \RV_{0}^\infty, \gamma \in \mathbb{R}, \alpha,\beta>0.
\end{equation} 
Well-known examples of Weibull-tailed distributions are the Gaussian with $\beta=2$, the gamma with $\beta=1$ or, more generally, the Weibull where $\beta$ is called the Weibull index.

For developing useful results, we further assume that, in addition to $R$ and $W$, $W_\wedge$ also has a Weibull-type tail. As previously, we index the corresponding 
$\ell$ functions and the parameters $\alpha, \gamma$ in \eqref{eq:weib} by the variable name.
We also recall Lemma \ref{enum3} concerning the dependence coefficients of the vector $\W$.

\begin{prop}[Weibull-type variables]\label{prop:weib}
Suppose that $F_R \in \mathrm{WT}_{\beta_{R}}$, $F_W\in \mathrm{WT}_{\beta_{W}}$ and $F_{W_\wedge}\in \mathrm{WT}_{\beta_{W_\wedge}}$. We have the following hierarchy of dependence structures: 

\begin{enumerate}
 \item \label{8item1}
  If $\beta_{W_\wedge} = \beta_W$, $\alpha_{W_\wedge} = \alpha_W$, $\gamma_{W_\wedge} = \gamma_W$, then $\chi_X = \chi_W = \lim_{x\rightarrow\infty}\ell_{W_\wedge}(x) / \ell_{{W}}(x)$,
   if the limit exists, and $\eta_X = \eta_W=1$.
\item \label{8item2}
  If $\beta_{W_\wedge} = \beta_W$, $\alpha_{W_\wedge} = \alpha_W$, $\gamma_{W_\wedge} < \gamma_W$, 
 then $\chi_X=0$ and $\eta_X = \eta_W=1$.
\item \label{8item3}
  If $\beta_{W_\wedge} = \beta_W$, $\alpha_{W_\wedge} > \alpha_W$, then $\chi_X=0$ and 
  \begin{equation*}\eta_X = \eta_W^{\beta_R/(\beta_R + \beta_W)} = \left(\alpha_W / \alpha_{W_\wedge}\right)^{\beta_R/(\beta_R + \beta_W)}.
 \end{equation*}
\item \label{8item4}
  If $\beta_{W_\wedge} > \beta_W$, then $\chi_X=0$ and
  $\eta_X = \eta_W = 0$.
\end{enumerate}
\end{prop}

In all of the cases encompassed by Proposition \ref{prop:weib}, $(X_1,X_2)$ and $(W_1,W_2)$ have the same tail dependence coefficient $\chi$, which can be positive only in case 1. In all other cases the variables are asymptotically independent, and only in case 3~the residual tail dependence coefficient $\eta$ changes under the multiplication of the radial variable $R$. Since $\beta_R/(\beta_R + \beta_W) \in (0,1)$, this always leads to an increase in dependence, that is, $\eta_X > \eta_W$.

\begin{example}[Gaussian scale mixtures]\label{ex:gsm}
  To illustrate the most interesting case 3~in Proposition~\ref{prop:weib} we consider 
  $(W_1,W_2)$ following a bivariate Gaussian distribution with standardized margins and correlation $\rho_W$.
  We have that 
  \begin{align*}
    \bar{F}_W(x)&\sim r_W(x) \exp(- x^2/2),&
    \bar{F}_{W_\wedge}(x) &\sim r_{W_\wedge}(x) \exp\{ - x^2/ (1+\rho_W) \},
  \end{align*}
  where the tail distribution of the minimum follows from bounds on the multivariate Mills ratio \citep[e.g.][]{HashorvaHusler2003}, and $r_W$ and $r_{W_\wedge}$ are regularly varying functions. Therefore, $\eta_W = (1+\rho_W)/2$ and Proposition \ref{prop:weib}  confirms  \citet[][Theorem~2]{Huseretal2017} where 
  $$\eta_X = \eta_W^{\beta_R/(\beta_R + 2)} = \left\{(1+\rho_W)/2\right\}^{\beta_R/(\beta_R + 2)}.$$
\end{example}

\begin{example}[Independence model]
  We continue the example where $R$, $W_1$ and $W_2$ are independent, and they are now assumed
  to be Weibull-tailed. The variable $W_\wedge$ is also Weibull-tailed with $\beta_{W_\wedge} = \beta_W$ and $\alpha_{W_\wedge} = 2\alpha_W$. Therefore, the third part of Proposition~\ref{prop:weib} entails that $\eta_X = 2^{-\beta_R/(\beta_R +\beta_W)}$. This expression tends to $1/2$ if $\beta_R/\beta_W \to \infty$ such that the tail of $W$ dominates strongly with respect to that of $R$; if $\beta_R/\beta_W \to 0$, then $\eta_X \to 1$.   
\end{example}

\subsection{Variables in the negative Weibull domain of attraction}
\label{sec:RWnegwei}

The remaining cases are those where $R$, $W$ or both of them are in the negative Weibull MDA, with finite upper endpoint. Recall that a variable $Y$ is 
in the negative Weibull MDA if
\begin{align}\label{eq:nwei}
 \bar{F}_Y(y^\star - s) =\ell(s) s^{\alpha}, \quad \ell \in \RV_0^0, \alpha>0.
\end{align}

The case where $R$ is superheavy-tailed or regularly varying and 
$W$ satisfies \eqref{eq:nwei} has been covered in part~1 of Proposition~\ref{prop:RSHT} and Proposition~\ref{prop:Frechet_chi}, respectively. On the other hand, the case where the tail of $R$ satisfies \eqref{eq:nwei} and $W$ is superheavy-tailed or regularly varying is treated by part 2 of Proposition~\ref{prop:RSHT} and Proposition~\ref{prop:WRV}, respectively. It remains to study the situation where
one of $R$ or $W$ is of form \eqref{eq:nwei}, and the other is in the Gumbel domain of attraction as defined in \eqref{eq:GumbelMDA}. In this section we focus
on the case where $W_\wedge$ has the same upper endpoint as $W$; for a more detailed study where $W_\wedge$ can have a smaller upper endpoint, and $W$ may have a point mass on its upper endpoint, see Section~\ref{sec:Constrained}.

\begin{prop}[Variables in the negative Weibull MDA]\label{prop:negwei}
\begin{enumerate}
  \item \label{item9.1}
    Suppose that $R$ is in the Gumbel MDA with upper endpoint $r^\star\in (0,\infty]$ and that
    $W$ and $W_\wedge$ satisfy \eqref{eq:nwei} with parameters $\alpha_W$ and $\alpha_{W_\wedge}$, respectively. Then $\chi_X = \chi_W$ and $\eta_X = 1$.
  \item \label{item9.2}
    Let $R$ satisfy \eqref{eq:nwei} and let $W$ and $W_\wedge$ be in the Gumbel MDA with equal upper endpoint $w^\star\in(0,\infty]$ and auxiliary functions $b_W$ and $b_{W_\wedge}$, such that $\lim_{x\to w^\star} b_W(x) / b_{W_\wedge}(x)$ exists. Then $\chi_X = \chi_W$ and $\eta_X = \eta_W$. 
  \item \label{item9.3}
    Let $R$, $W$ and $W_\wedge$ all satisfy \eqref{eq:nwei} with endpoints $r^\star, w^\star, w^\star$ and parameters $\alpha_R, \alpha_W$ and $\alpha_{W_\wedge}$, respectively. If $\alpha_{W_\wedge} = \alpha_W$ then $\chi_X = \chi_W$ and $\eta_X = 1$.
    If $\alpha_{W_\wedge} > \alpha_W$ then $\chi_X = \chi_W = 0$ and 
\begin{equation}\label{eq:etanegweib}
\eta_X = (\alpha_W + \alpha_R)/(\alpha_{W_\wedge} + \alpha_R) > \alpha_{W} / \alpha_{W_\wedge} = \eta_W.
\end{equation}
\end{enumerate} 
\end{prop}

\begin{example}[Independence model]
  Continuing the running independence example, we now suppose that $F_R, F_W$ satisfy~\eqref{eq:nwei} with parameters $\alpha_R,\alpha_W$. Clearly, $F_{W_\wedge}$ also satisfies~\eqref{eq:nwei}, with parameter $\alpha_{W_\wedge} = 2\alpha_W$.  
  The third part of Proposition~\ref{prop:negwei} shows that 
  $$\eta_X = (\alpha_W + \alpha_R)/(2\alpha_{W} + \alpha_R) \in (1/2, 1),$$
  hence by varying the parameters $\alpha_R, \alpha_W > 0$ we can attain the whole
  range of residual tail dependence coefficients related to positive association.
\end{example}

\section{Literature review and examples}
\label{sec:Examples}

Here we present an overview of related literature, detailing how existing examples and results fit into the framework of this paper.

\paragraph{Elliptical copulas} Let $\Sigma$ be a positive-definite covariance matrix with Cholesky decomposition $\Sigma=AA^T$, and $(U_1,U_2)$ be uniformly distributed on the $L_2$ sphere $\{(w_1,w_2):(w_1,w_2)^T (w_1,w_2)=1\}$. Then 
$\X = R A(U_1,U_2)^T$
has an elliptical distribution for any $R> 0$ called the \emph{generator}. Therefore $\W^T = A(U_1,U_2)^T$ lies on the Mahalanobis sphere $\mathcal{W} = \{(w_1,w_2):(w_1,w_2)^T\Sigma^{-1}(w_1,w_2)=1\}$, and the extremal dependence in the upper right orthant is unchanged by taking $(W_j)_+ = \max(W_j,0)$. It is well known that $\X$ is asymptotically dependent if and only if $R$ is in the Fr\'{e}chet MDA \citep[][Theorem 4.3]{HultLindskog02}. In that case, the tail dependence coefficient $\chi_X$ is given by~\eqref{eq:Rrvchi}, with $W_j$ replaced by $(W_j)_+$; see also \citet{Opitz2013}. For $R$ in the Gumbel MDA, the scaling condition on $\nu$ such that $\tau(w) \in[0,1]$ yields $\Sigma$ with diagonal elements 1, off-diagonal elements $\rho\in(-1,1)$, and residual tail dependence coefficient is given by  Proposition~\ref{prop:RGumbel}\eqref{RGumbelitem1} with $\zeta=\tau(1/2) = \{(1+\rho)/2\}^{1/2}$. \citet{Hashorva2010} details calculation of $\eta_X$ assuming $R$ to be in the Gumbel MDA, providing an alternative perspective on the derivation. The spatial model of \citet{Huseretal2017} is also covered by this case.

\begin{example*}[Gaussian]
The Gaussian distribution arises when $\bar{F}_R(r) = e^{-r^2/2}$, so that by Corollary~\ref{cor:RWeib}, $\eta_X = \zeta^2 = (1+\rho)/2$.
\end{example*}

\paragraph{Archimedean and Liouville copulas} Archimedean (respectively Liouville) copulas arise as the survival copula when $\W$ is uniformly (respectively Dirichlet) distributed on the positive part of the $L_1$ sphere $\mathcal{W}=\{(w_1,w_2) \in [0,1]^2 : w_1+w_2=1\}$, and $R> 0$. That is, $\X = R\W$ has an \emph{inverted} Archimedean or Liouville copula, whilst the Archimedean or Liouville copula itself is that of $(t(X_1),t(X_2))$, for a monotonic decreasing transformation $t$. By taking $t(x) = 1/x$, we obtain $1/\X = (\tilde{X_1},\tilde{X_2}) = \tilde{R}(\tilde{W_1},\tilde{W_2})$, so Archimedean copulas have a random scale representation with $(\tilde{W_1},\tilde{W_2})$ constrained by functional dependence that is not represented by a norm.

Archimedean copulas are typically defined in terms of a non-increasing continuous generator function $\psi:[0,\infty)\to[0,1]$, such that $C(u_1,u_2) = \psi(\psi^{-1}(u_1)+\psi^{-1}(u_2))$. The link between $\psi$ and the variable $R\sim F_R$ is given in equation~(3.3) of \citet{McNeilNeslehova09}; for $d=2$ this is 
\begin{align}
\bar{F}_R(r) = \psi(r) -r \psi'(r_+), \qquad r >0, \label{eq:ArchR}
\end{align}
where $\psi'(r_+)$ denotes the right-hand derivative of $\psi$.

Archimedean copulas are a special case of Liouville copulas, whose dependence properties are studied in \citet{BelzileNeslehova17}. For $\X$, their Theorem~1 states that $R$ in the Fr\'echet MDA leads to asymptotic dependence, whilst the Gumbel and negative Weibull MDAs lead to asymptotic independence. The exponent function given in their Theorem~1 matches equation~\eqref{eq:V}. In their Theorem~2, \citet{BelzileNeslehova17} consider the extremal dependence properties of $1/\X = \tilde{R}(\tilde{W_1},\tilde{W_2})$, i.e., the Liouville copula itself. Since the reciprocal of Dirichlet random variables have regularly varying tails, this links with Proposition~\ref{prop:WRV} which states that asymptotic independence arises if $(\tilde{W_1},\tilde{W_2})$ themselves are asymptotically independent and heavier-tailed than $R$. Proposition~\ref{prop:samealpha}\eqref{item:prop6wminlighter} 
is relevant if $\tilde{R}$ and $\tilde{W}$ are regularly varying with the same index.

\begin{example*}[Gumbel and inverted Gumbel copulas]
The Gumbel, or logistic, Archimedean copula arises when $\psi(x) = e^{-x^{\theta}}$, $\theta \in (0,1]$. By~\eqref{eq:ArchR}, $\bar{F}_R(r) = e^{-r^{\theta}}(1+\theta r^{\theta}) \in \mathrm{WT}_\theta$.
The copula of $\X$ is the asymptotically independent inverted Gumbel copula \citep{LedfordTawn1997}. We have $\zeta = \tau(1/2) = 1/2$ and so $\eta_X = 2^{-\theta}$ by Corollary~\ref{cor:RWeib}.
The Gumbel copula is that of $1/\X = \tilde{R}(\tilde{W_1},\tilde{W_2})$, with $F_{\tilde{R}}(r) = e^{-r^{-\theta}}(1+\theta r^{-\theta})$, so $\bar{F}_{\tilde{R}} \in \RV_{-\theta}^\infty$. The dependence structure follows from Proposition~\ref{prop:Frechet_chi} for $\theta<1$ since $\E(\tilde{W}^\theta)<\infty$. Noting that $\tilde{W}_{\wedge}$ is a bounded random variable, $\chi_X=0$ for $\theta=1$, as given by Proposition~\ref{prop:samealpha}\eqref{item:prop6wminlighter}. In fact, for $\theta=1$, the copula is the independence copula.
\end{example*}

\paragraph{Archimax copulas} Bivariate Archimax copulas were introduced by \citet{CaperaaEtAl2000} and extended to the multivariate case with a stochastic representation by \citet{CharpentierEtAl2014}. They are so-called because of a connection to both Archimedean and extreme-value max-stable copulas. Letting $\psi$ be the generator of an Archimedean copula, and $V$ the exponent function defined in~\eqref{eq:expfn}, a bivariate Archimax copula can be expressed as $C(u_1,u_2) = \psi \circ V(1/\psi^{-1}(u_1),1/\psi^{-1}(u_2))$, such that taking $V(x_1,x_2) = 1/x_2+1/x_2$ --- corresponding to the independence max-stable copula --- yields an Archimedean copula, whilst taking $\psi(x)=e^{-x}$ yields a max-stable copula. \citet{CharpentierEtAl2014} show that the vector $(X_1,X_2) = R(W_1,W_2)$ has an inverted Archimax copula if $\bar{F}_R$ is as in~\eqref{eq:ArchR}, and $\P(W_1\geq w_1,W_2 \geq w_2) = \max(0,1-V(1/w_1,1/w_2))$. Hence $(\tilde{X}_1,\tilde{X}_2)=1/(X_1,X_2) = \tilde{R}(\tilde{W}_1,\tilde{W}_2)$ has an Archimax copula. We have $\bar{F}_{\tilde{W}}(w) \in \RV_{-1}^\infty$ and $\chi_{\tilde{W}}=2-V(1,1)$ which is positive unless $V(x_1,x_2) = 1/x_2+1/x_2$. If $\tilde{R}$ has a lighter tail then Proposition~\ref{prop:WRV} gives $\chi_{\tilde{X}}=\chi_{\tilde{W}}$, whilst if $\tilde{R}$ is the same or heavier, the results of Propositions~\ref{prop:Frechet_chi},~\ref{prop:RSHT} or~\ref{prop:samealpha} are relevant. The inverted case follows similarly to the calculations in Section~\ref{sec:Constrained} since the margins of $W_1,W_2$ are uniform and the zero-truncation in the copula for $(W_1,W_2)$ means that there is mass on $\{w \in [0,1]: V(1/w,1/(1-w))=1\}$ whenever $V(1/w_1,1/w_2)>1$, where $V(1/x_1,1/x_2)$ defines a norm.

\begin{example*}[Gumbel Archimax]
Taking $V(x_1,x_2)=(x_1^{-1/\theta}+x_2^{-1/\theta})^\theta$, the exponent function of the logistic, then the corresponding Archimax copula is $C(u_1,u_2) = \psi\{(\psi^{-1}(u_1)^{1/\theta}+\psi^{-1}(u_2)^{1/\theta})^\theta\}$, which is Archimedean with generator $\psi(x^\theta)$ \citep{CharpentierEtAl2014}. If $\psi(x)=e^{-x^\alpha}, \alpha \in (0,1]$, then we obtain the Gumbel copula with parameter $\theta\alpha$. Tail dependence results can then be obtained as in the example above, or considering the Archimax structure. Following the latter, for $\tilde{X}$ Proposition~\ref{prop:Frechet_chi} gives $\chi_{\tilde{X}} = 2-V(1,1)^\alpha = 2-2^{\theta\alpha}$, for $\alpha \in (0,1)$ whilst Proposition~\ref{prop:samealpha}\eqref{item:6a} gives the extension to $\alpha=1$. For the inverted copula $\eta_{X} = (2^{-\theta})^\alpha$, following similar lines to Proposition~\ref{prop:RGumbel}.
\end{example*}

\paragraph{Multivariate ($\rho$-)Pareto copulas} Let $\rho:(0,\infty)^2\to(0,\infty)$ be a positive homogeneous function. Multivariate $\rho$-Pareto copulas arise when $\bar{F}_R(r)=r^{-1}$, i.e.\ standard Pareto, and the random vector $(W_1,W_2)$ is concentrated on $\mathcal{W} = \{(w_1,w_2) \in \mathbb{R}_+^2: \rho(w_1,w_2) = 1\}$ with marginals satisfying $\E(W) < \infty$ \citep{DombryRibatet15}. The case of $\rho(x_1,x_2) = \max(x_1,x_2)$ leads to the multivariate Pareto copula associated to multivariate generalized Pareto distributions \citep{RootzenTajvidi2006,FerreiradeHaan2014,Rootzenetal2017}. Such copulas are asymptotically dependent, except for the case outlined in Section~\ref{sec:Frechet}, with $\chi_X$ given by~\eqref{eq:Rrvchi}. Although we have focused on norms and $\rho$ need not be convex, there is nothing in Proposition~\ref{prop:Frechet_chi} requiring this. 

\begin{example*}[Bivariate Pareto copula associated to the Gumbel copula]
Since the Gumbel copula is a max-stable distribution, it has an associated Pareto copula.   If $Z$ has density $f_Z(z) = h(z)\max(z,1-z)2^{1-\theta}$, where $h$ is given by~\eqref{eq:hlogistic}, then $\X = R\W = R(Z,1-Z)/\max(Z,1-Z)$ with $\bar{F}_{R}(r)=r^{-1}$ leads to the associated bivariate Pareto copula. The distribution function of $\X$ is
\[
\P(X_1 \leq x_1, X_2 \leq x_2) = \left\{V(\min(x_1,1),\min(x_2,1)) - V(x_1,x_2)\right\} / V(1,1),
\]
where $V(x_1,x_2) = (x_1^{-1/\theta} + x_2^{-1/\theta})^\theta$ is the exponent function for the Gumbel distribution. 
\end{example*}

\paragraph{Model of \citet{deh2011}} 
They describe the losses of two banks by a factor model $S_j = C + L_j$, $j=1,2$, where $\bar{F}_{L_j} \in \RV_{-\alpha}^\infty$ and $\bar{F}_C \in \RV_{-\beta}^\infty$. For $\beta<\alpha$ they show that
$S=(S_1,S_2)$ is completely asymptotically dependent, i.e., $\chi_S = 1$, and for $\beta=\alpha$ they obtain $\chi_S \in(0,1)$. If $\alpha < \beta  < 2\alpha$ asymptotic independence arises with $\eta_S = \alpha/\beta$, and if $\beta > 2\alpha$ then $\eta_S = 1/2$. Our proposition 
\ref{prop:RSHT} yields the same results as special cases with $R = \exp(C)$ and $W_j = \exp(L_j)$, $j=1,2$.

\paragraph{Model of \citet{Wadsworthetal2017}} They consider the copula induced by taking $R$ to be generalized Pareto, $\bar{F}_R(r) = (1+\lambda r)^{-1/\lambda}_+$, and $\mathcal{W} = \{(w_1,w_2) \in [0,1]^2: \|(w_1,w_2)\|_* = 1\}$ where $\|\cdot\|_*$ is a symmetric norm subject to certain restrictions. These restrictions mean that $\lambda\leq 0$ corresponds to asymptotic independence; the residual tail dependence coefficient $\eta_X$ is as given in Proposition~\ref{prop:RGumbel} for $\lambda=0$ with $\zeta=\tau(1/2) = \|(1,1)\|_*^{-1}$, and Proposition~\ref{prop:RNegWeib} for $\lambda<0$. We note that if the norm $\nu$ has certain shapes that were excluded in \citet{Wadsworthetal2017}, asymptotic dependence is possible for $\lambda\leq 0$. When $R$ is in the Fr\'echet MDA $(\lambda>0)$ then asymptotic dependence holds with $\chi_X$ given by~\eqref{eq:Rrvchi}.

\paragraph{Model of \citet{Krupskiietal2017}} They consider location mixtures of Gaussian distributions, corresponding to scale mixtures of log-Gaussian distributions. According to their Proposition~1, asymptotic dependence occurs when the location variable is of exponential type, i.e.\ the scale is of Pareto type; the given $\chi_X$ can then be obtained via~\eqref{eq:Rrvchi}. When the location is Weibull-tailed but with shape in $(0,1)$, the scale is superheavy-tailed, with $\bar{F}_R \in \RV_0^\infty$, and perfect extremal dependence ($\chi_X=1$) arises, as noted in Remark~\ref{rmk:SVR} following Proposition~\ref{prop:Frechet_chi}. When the random location is Weibull-tailed with shape in $(1,\infty)$ then the random scale $R$ is in the Gumbel MDA and asymptotic independence arises. 
If $\bar{F}_{\log R} \in \mathrm{WT}_2$ has the same Weibull coefficient $2$ as the standard Gaussian $\log W$ and as $\log W_\wedge$ (provided that $ \rho = \rm{cor}(\log W_1,\log W_2) \in (-1,1]$), then we can  apply Proposition~\ref{prop:lightS} to calculate the value of $\eta_X$ given as
$$
\eta_X=\eta_W\frac{\alpha_{W_\wedge}+\alpha_R}{\alpha_W+\alpha_R}=\frac{1+\rho}{2}\frac{(1+\rho)^{-1}+\alpha_R}{1/2+\alpha_R}=\frac{1+(1+\rho)\alpha_R}{1+2\alpha_R}, 
$$
which extends the results of \citet{Krupskiietal2017}. Specifically, with standard Gaussian $\log R$ we get $\eta_X=(3+\rho)/4$, see Example~\ref{ex:gfm}. 

\paragraph{Model of \citet{HuserWadsworth2017}} They consider scale mixtures of asymptotically independent vectors where both $R$ and $W$ have Pareto margins with different shape parameters. Asymptotic dependence arises when $R$ is heavier tailed; $\chi_X$ is then given by~\eqref{eq:Rrvchi}, whilst asymptotic independence arises when $W$ is heavier tailed and $\eta_X$ is given by~\eqref{eq:etaWRV}. When $R$ and $W$ have the same shape parameter, their assumption $\eta_W<1$ implies that $\E(W_{\wedge}^{\alpha+\varepsilon})<\infty$ for some $\varepsilon>0$, giving asymptotic independence by Proposition~\ref{prop:samealpha}\eqref{item:prop6wminlighter}. 

\bigskip
Various other articles also focus on polar or scale-mixture representations. \citet{Hashorva2012} examines the extremal behavior of scale mixtures when $R$ is in the Gumbel MDA. He considers both one- and two-dimensional $\mathcal{W}$, both with similarities and differences to our set-up. For one-dimensional $\mathcal{W}$, he assumes a functional constraint of the form $\W = (W,\rho W+z^*(W))$ for measurable $z^*:[0,1]\to (0,\infty)$; $\rho\in(-1,1)$ (the specification also allows for negative components, but we focus here on the positive part). Constraints to certain norm spheres, such as the Mahalanobis or $L_p$ norm, could be written in this way, however examples such as the $L_\infty$ norm could not. Where the representations overlap, our results coincide (e.g., Section 4.3 of \citet{Hashorva2012}). In the case of two-dimensional $\mathcal{W}$, $W_1$ is assumed bounded, whilst $W_\wedge$ is in the negative Weibull MDA. Although, differently to \citet{Hashorva2012}, we typically assume symmetry, there are nonetheless some connections between the results in our  Section~\ref{sec:RWnegwei} and that paper.

 \citet{Nolde2014} provides an interpretation of extremal dependence in terms of a \emph{gauge function} (see also \citet{BalkemaNolde2010} and \citet{BalkemaEmbrechts2007}), which, loosely speaking, corresponds to level sets of the density in light-tailed margins. The main result of \citet{Nolde2014} (Theorem 2.1) is presented in terms of Weibull-type margins, such that $-\log\bar{F}_X \in \RV_{\delta}^\infty$, $\delta>0$; in terms of Section~\ref{sec:Constrained}, this corresponds to $-\log\bar{F}_R \in \RV_{\delta}^\infty$. By noting that where the density of $R$, $f_R$, exists, the joint density of $\X = R(\tau(Z),\tau(1-Z))$ is
 \[
 f_{\X}(x_1,x_2) = f_{R}(\nu(x_1,x_2))f_{Z}(x_1/(x_1+x_2))\nu(x_1,x_2) / (x_1+x_2)^2, 
 \]
the gauge function of $\X$ is obtained as $\nu$ when $-\log\bar{F}_R \in \RV_{\delta}^\infty$, using Proposition~3.1 therein. We found $\eta_X = \zeta^{\delta}$, with $\zeta=\nu(1,1)^{-1}$, precisely as in \citet{Nolde2014}.

Various papers focus on extremal dependence arising from certain types of polar representation, but from a conditional extremes perspective \citep{HeffernanTawn2004,HeffernanResnick2007}. This is different to our focus; here we examine the extremal dependence as both variables grow at the same rate. In the conditional approach, different rates of growth may be required in the different components. \citet{Abdousetal2005} examine conditional limits in the context of elliptical copulas, whilst \citet{FougeresSoulier2010} and \citet{Seifert2014} consider the constrained $\mathcal{W}$ case, with $R$ in the Gumbel MDA.

\section{New examples}
\label{sec:NewExamples}

We present two new constructions that have the desirable property of smoothly interpolating between asymptotic dependence and asymptotic independence, whilst yielding non-trivial structures within each class. By smoothness, we mean that the transition between classes occurs at an interior point, $\theta_0$, of the parameter space $\Theta$, and, assuming increasing dependence with $\theta$, $\lim_{\theta \to \theta_{0+}} \chi_X = 0$, $ \lim_{\theta \to \theta_{0-}} \eta_X=1$. To our knowledge, the only other models in the literature with this behavior are (i) that of \citet{Wadsworthetal2017}, where $\nu(x,y) = \max(x,y)$ and $\bar{F}_R(x) = (1+\lambda x)_+^{-1/\lambda}$, $\lambda \in\mathbb{R}$, and (ii) that of \citet{HuserWadsworth2017} where $\bar{F}_R(x) = x^{-1/\delta}$, $\bar{F}_W(x) = x^{-1/(1-\delta)}$, $\delta \in(0,1)$, and $\eta_W<1$. The first example is constructed using constrained $\W$ (Section~\ref{sec:Constrained}), where the required ingredients are $\bar{F}_R$, $\nu$, and $F_Z$, whilst the second uses unconstrained $\W$ (Section~\ref{sec:Unconstrained}) with ingredients $\bar{F}_R$, $F_W$ and the dependence structure of $\W$.

\subsection{Model 1}
In Propositions~\ref{prop:RGumbel} and~\ref{prop:RNegWeib}, it was demonstrated how the shape of $\nu$ affects the tail dependence of $\X$ when $R$ is in the Gumbel or negative Weibull MDA. In Example~\ref{ex:AD}, a particular norm that yields asymptotic dependence was given; here we extend the parameterization of this norm and use our results to present a new asymptotically (in)dependent copula. Since the cases where $R$ has finite upper endpoint $r^\star<\infty$ lead to undefined $\eta_X$, we focus on $r^\star=\infty$.

\begin{prop}
\label{prop:constrainedmodel}
 Let $R$ be in the Gumbel MDA with $r^\star=\infty$ and let $\nu(x,y) = \theta\max(x,y)+(1-\theta)\min(x,y)$, $\theta\geq 1/2$. Then for $\W$ defined through~\eqref{eq:Wnu} and $Z$ satisfying the conditions in Section~\ref{sec:Gumbel}, 
 \begin{align*}
  \chi_X &= \max(2(\theta-1)/(2\theta-1),0) & \eta_X &= \lim_{x\to\infty} \log \bar{F}_R(x)/ \log \bar{F}_R(x/\min(\theta,1)).                                                
 \end{align*}
\end{prop}

We note that if $- \log \bar{F}_R \in \RV_\delta$, for example, then we have a continuous parametric family exhibiting asymptotic dependence for $\theta>1$ with $\chi_X = 2(\theta-1)/(2\theta-1)$, and asymptotic independence for $\theta \leq 1$ with $\eta_X = \theta^\delta$. To make things concrete, we propose the following model.

\newtheorem{model}{Model}

\begin{model}
\label{mod:1}
\begin{align*}
\bar{F}_R(r) &=  e^{-r^\delta}  & \nu(x,y)&=\theta\max(x,y)+(1-\theta)\min(x,y),~~ \theta \geq 1/2 & Z &\sim \mbox{Beta}(\alpha,\alpha).
\end{align*}
\end{model}

The set of models defined in Proposition~\ref{prop:constrainedmodel}, exemplified by Model~\ref{mod:1}, has some rather interesting behavior in the extremes. Whilst the limiting quantities $\chi_X$, $\eta_X$ are given in Proposition~\ref{prop:constrainedmodel}, the subasymptotic behavior of $\X$, in particular the behavior of the slowly varying function $\ell$ in~\eqref{eq:chi}, is not prescribed by any of the propositions in this paper. Combining equations~\eqref{eq:chi} and~\eqref{eq:eta}, define \[\chi_X(q) = \Pr(X_1 \geq F_{X_1}^{-1}(q),X_1 \geq F_{X_1}^{-1}(q))/(1-q) = \ell(1-q)(1-q)^{1/\eta_X - 1},\] so that for $\eta_X =1$, $\chi_X(q) = \ell(1-q)$. For Model~\ref{mod:1} we find that $\chi_X(q)$ is not necessarily monotonic, and may decrease before increasing to a positive limit value. Figure~\ref{fig:chis} shows $\chi_X(q)$ for various parameterizations of the model. This non-monotonic behavior appears uncommon; to our knowledge there are no well-known theoretical examples of this.

\begin{figure}[htb]
 \centering
 \includegraphics[width=0.3\textwidth]{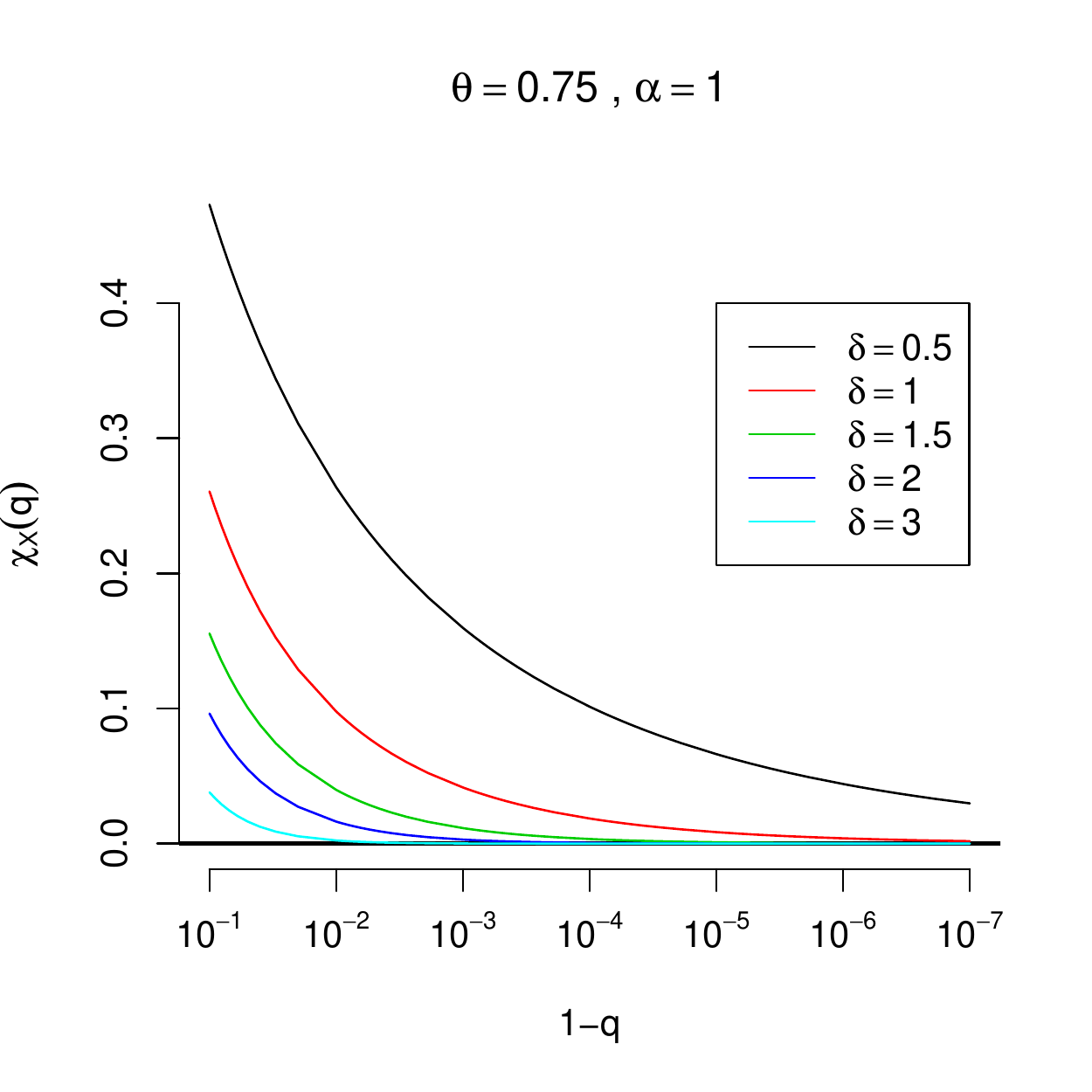}
 \includegraphics[width=0.3\textwidth]{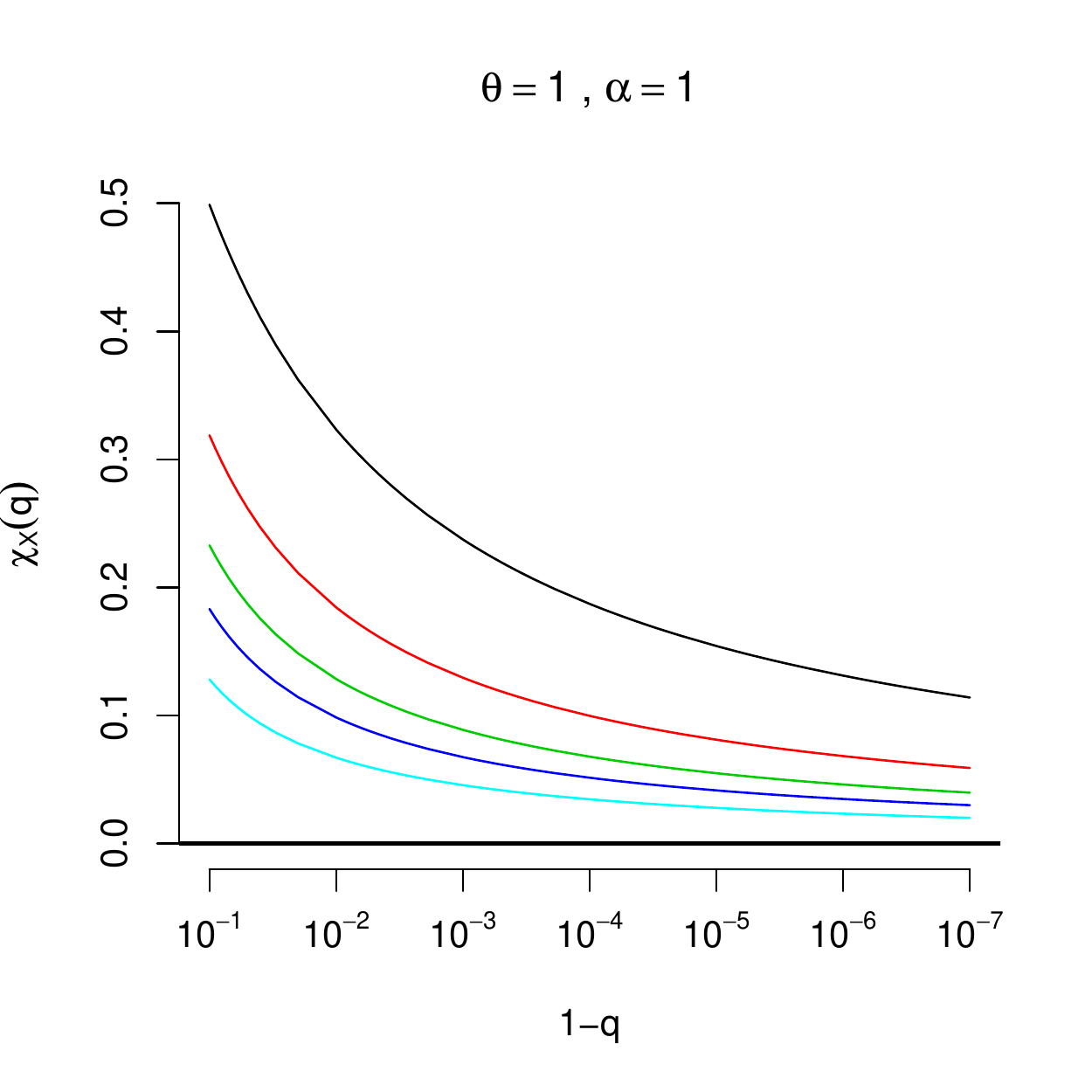}
 \includegraphics[width=0.3\textwidth]{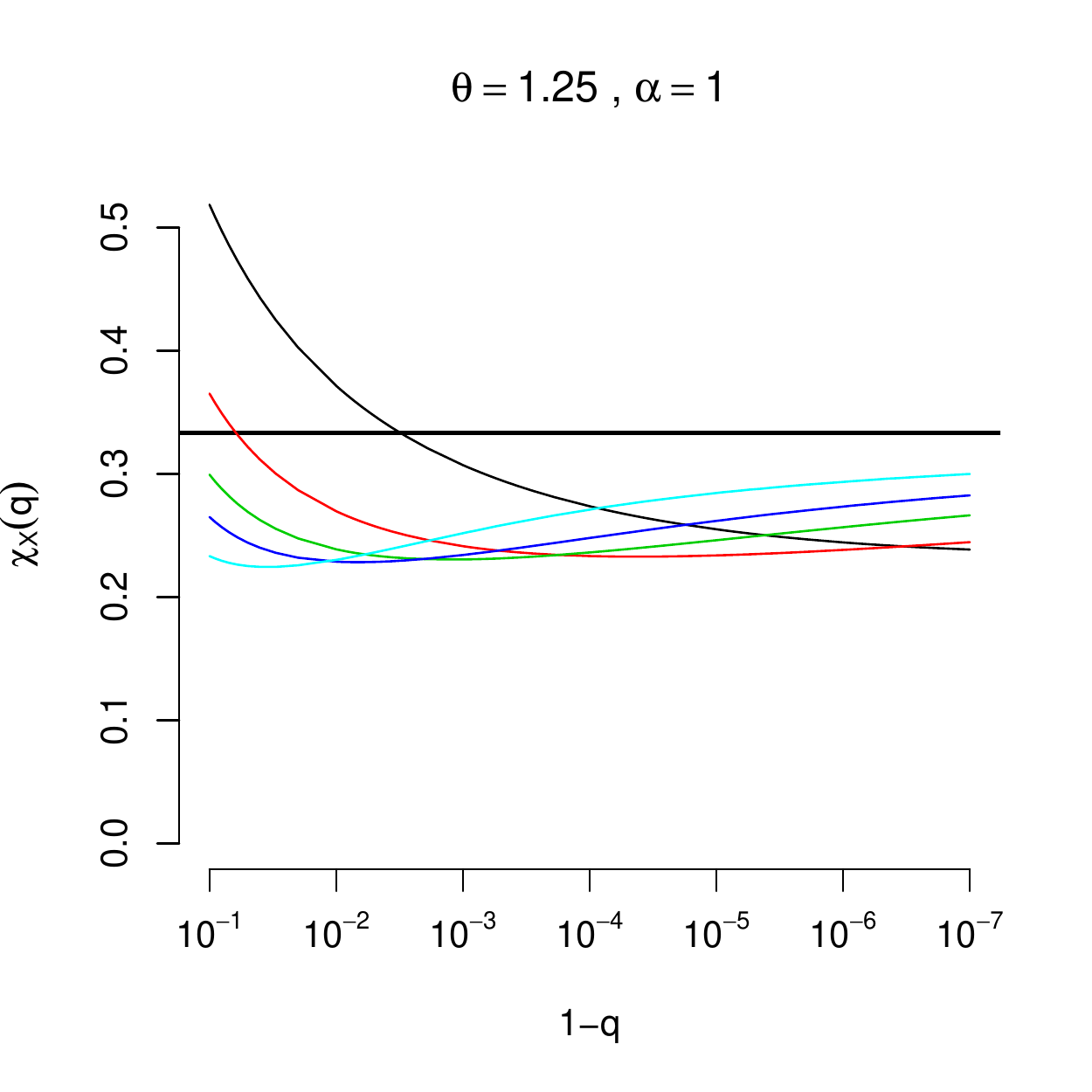}\\
 \includegraphics[width=0.3\textwidth]{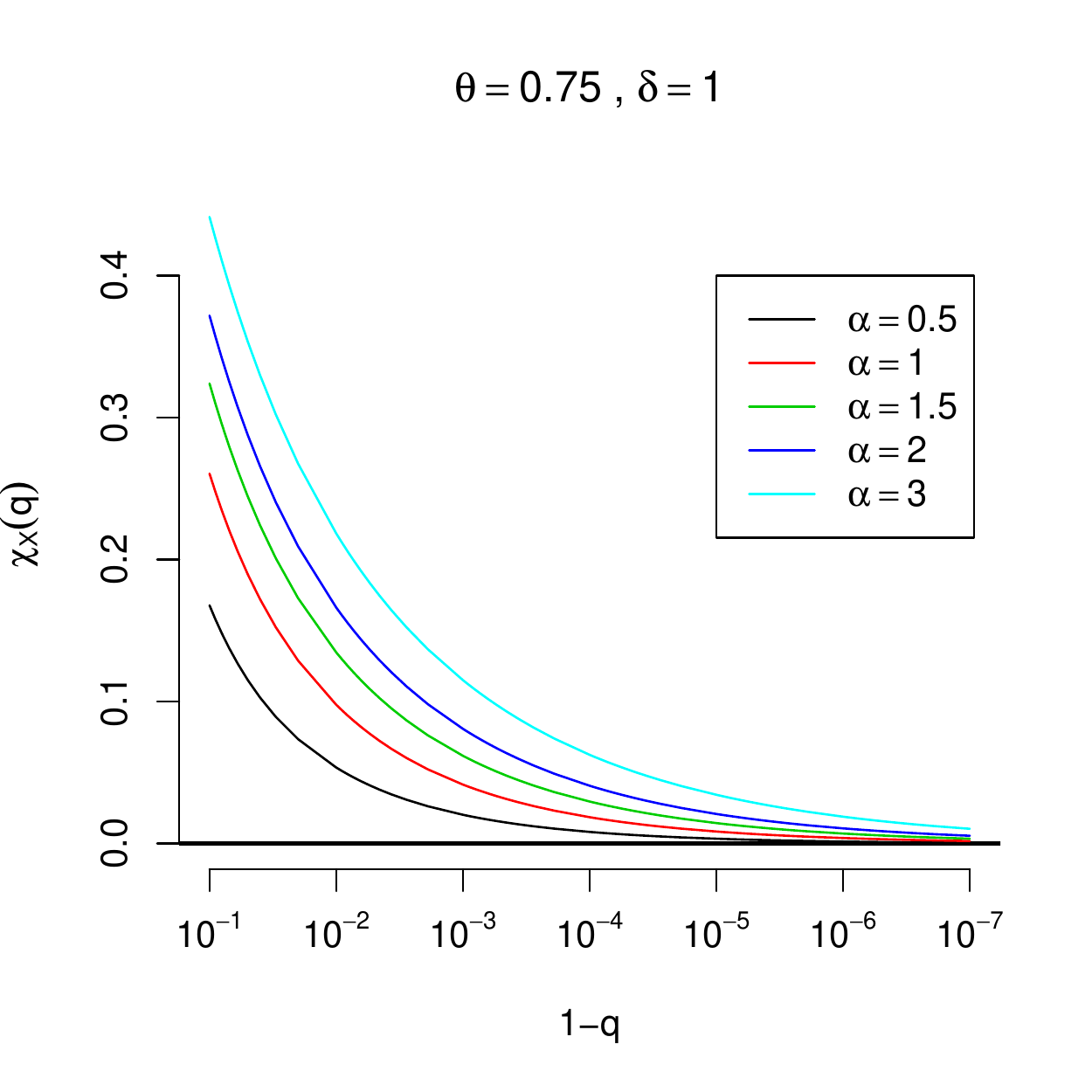}
 \includegraphics[width=0.3\textwidth]{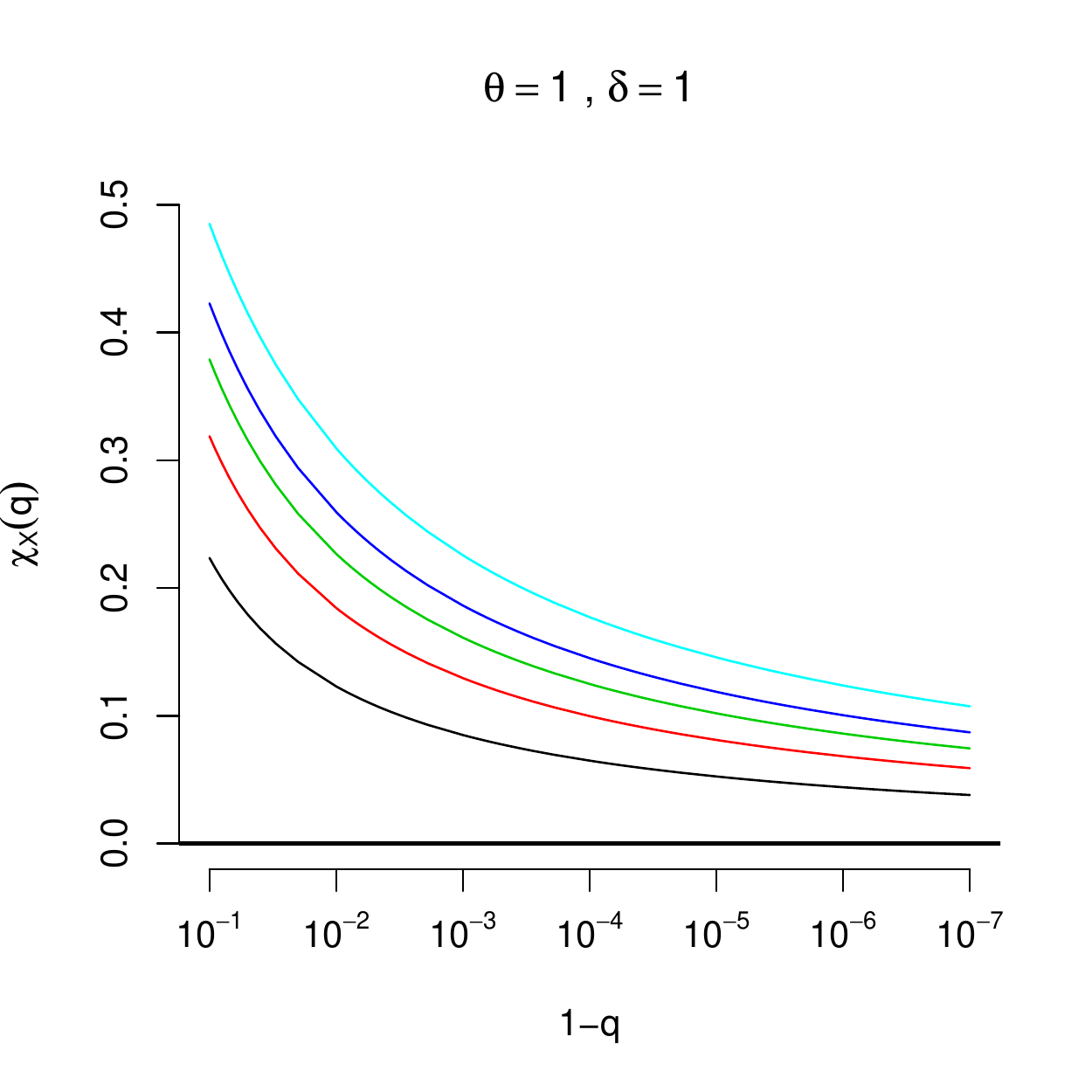}
 \includegraphics[width=0.3\textwidth]{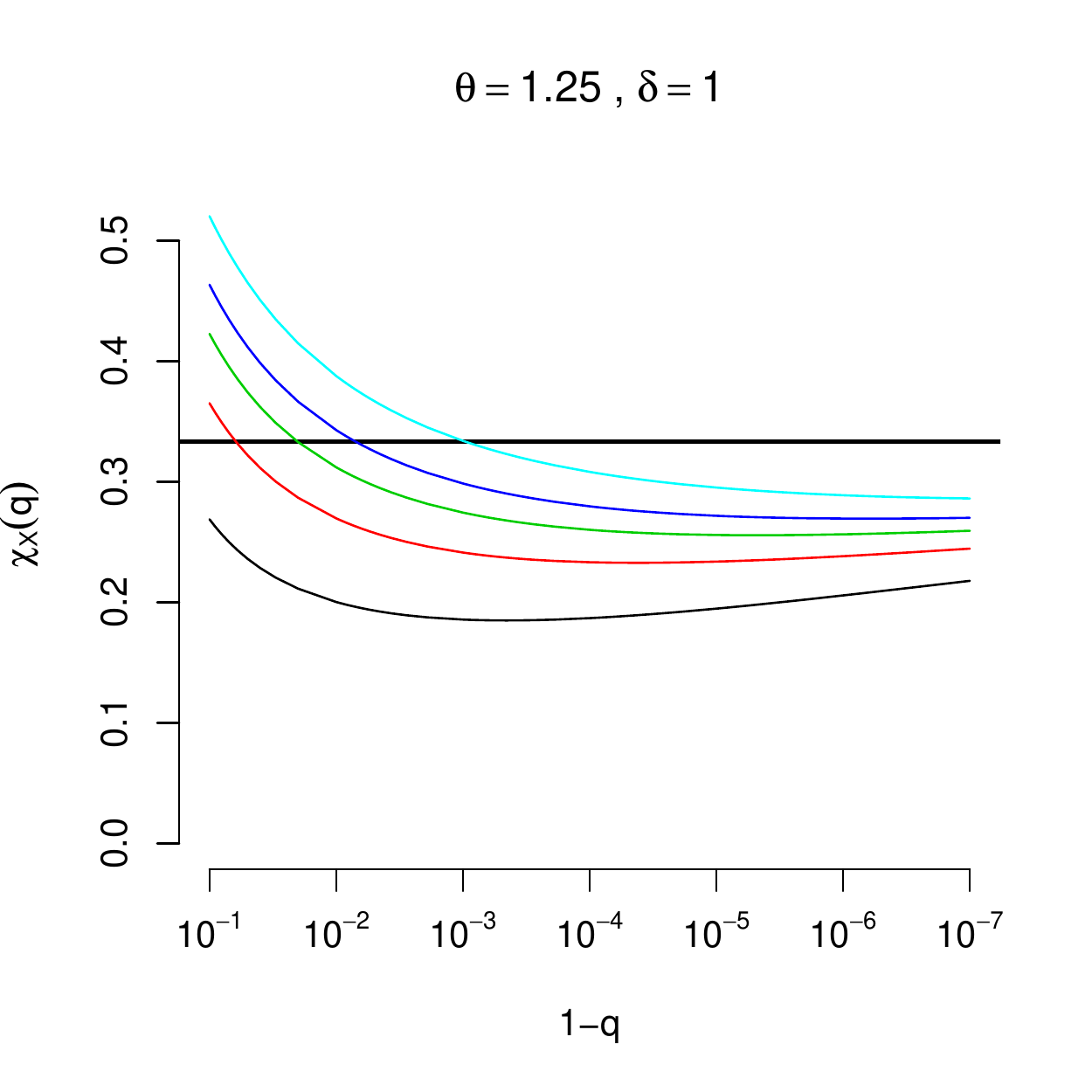}
 \caption{Theoretical $\chi_X(q)$ for Model~\ref{mod:1} plotted against $1-q$ on a logarithmic scale for $q \in [1-10^{-1},1-10^{-7}]$. Different columns show different values of $\theta$; thick horizontal lines show the true limiting values of $\chi_X = (0,0,1/3)$ (left-right). Top row: $\delta$ varies within a panel; bottom row: $\alpha$ varies within a panel.} \label{fig:chis}
\end{figure}

\subsection{Model 2}

The following proposition collates results from Propositions~\ref{prop:Frechet_chi} and~\ref{prop:negwei}, and provides a general principle for constructing new dependence models permitting both asymptotic dependence and asymptotic independence.
\begin{prop}
\label{prop:unconstrainedmodel}
 Let $R$ be in the MDA of a generalized extreme value distribution with shape parameter $\xi\in\mathbb{R}$, and let $\W$ with $W_1 \overset{d}{=} W_2 \overset{d}{=} W$ have $\chi_W=0$, well-defined $\eta_W\in(0,1)$, and $\bar{F}_W(w^\star - \cdot) \in \RV_{\alpha_W}^0$, $\alpha_W>0$. Then
 \begin{enumerate}
  \item For $\xi>0$, $\chi_X = \E(W_\wedge^{1/\xi})/\E(W^{1/\xi})$, $\eta_X=1$,  
  \item For $\xi=0$, $\chi_X=0$, $\eta_X=1$,
  \item For $\xi<0$, $\chi_X=0$, $\eta_X=(1-\xi\alpha_W) /(1- \xi \alpha_{W}/\eta_W)$.
 \end{enumerate}
\end{prop}

The model construction opportunities from Proposition~\ref{prop:unconstrainedmodel} are quite varied; specifically taking $\bar{F}_R$ that permits all three tail behaviors produces a flexible range of models spanning the two dependence classes. We therefore propose the following concrete model, based on our running independence example.

\begin{model}
\label{mod:2}
\begin{align*}
\bar{F}_R(r) &=  (1+\xi x)^{-1/\xi}_+,~~ \xi\in\mathbb{R}  & W_1 &\ci W_2 & W &\sim \mbox{Beta}(\alpha,\alpha).
\end{align*}
\end{model}

For the special case $\alpha=1$, i.e, $W \sim \mbox{Unif}(0,1)$, one can explicitly calculate $\chi_X$ and $V_X$ as well as $\eta_X$. By Proposition~\ref{prop:unconstrainedmodel}, for $\xi<0$, $\eta_X = (1-\xi)/(1-2\xi)$ with $\lim_{\xi\to 0_{-}} \eta_X=1$, and $\chi_X = 0$; for $\xi=0$, $\chi_X=0$, $\eta_X=1$, and for $\xi>0$, $\chi_X = 2\xi/(2\xi+1)$, whilst
\begin{align}
 V_X(x_1,x_2) = \min(x_1,x_2)^{-1} + \frac{1}{2\xi+1} \left(\frac{\min(x_1,x_2)}{\max(x_1,x_2)}\right)^{\xi} \max(x_1,x_2)^{-1}. \label{eq:mod2V}
\end{align}
The limits of~\eqref{eq:mod2V} as $\xi\to 0$ and $\xi\to\infty$ are $x_1^{-1}+x_2^{-1}$ and $\min(x_1,x_2)^{-1}$, corresponding to independence and perfect dependence \citep[e.g.][Ch. 8]{Beirlantetal2004}. Figure~\ref{fig:model2chis} displays the function $\chi_X(q)$ for Model~\ref{mod:2} across a range of different $\alpha$ and $\xi$ values. 
\begin{figure}[h]
 \centering
 \includegraphics[width=0.3\textwidth]{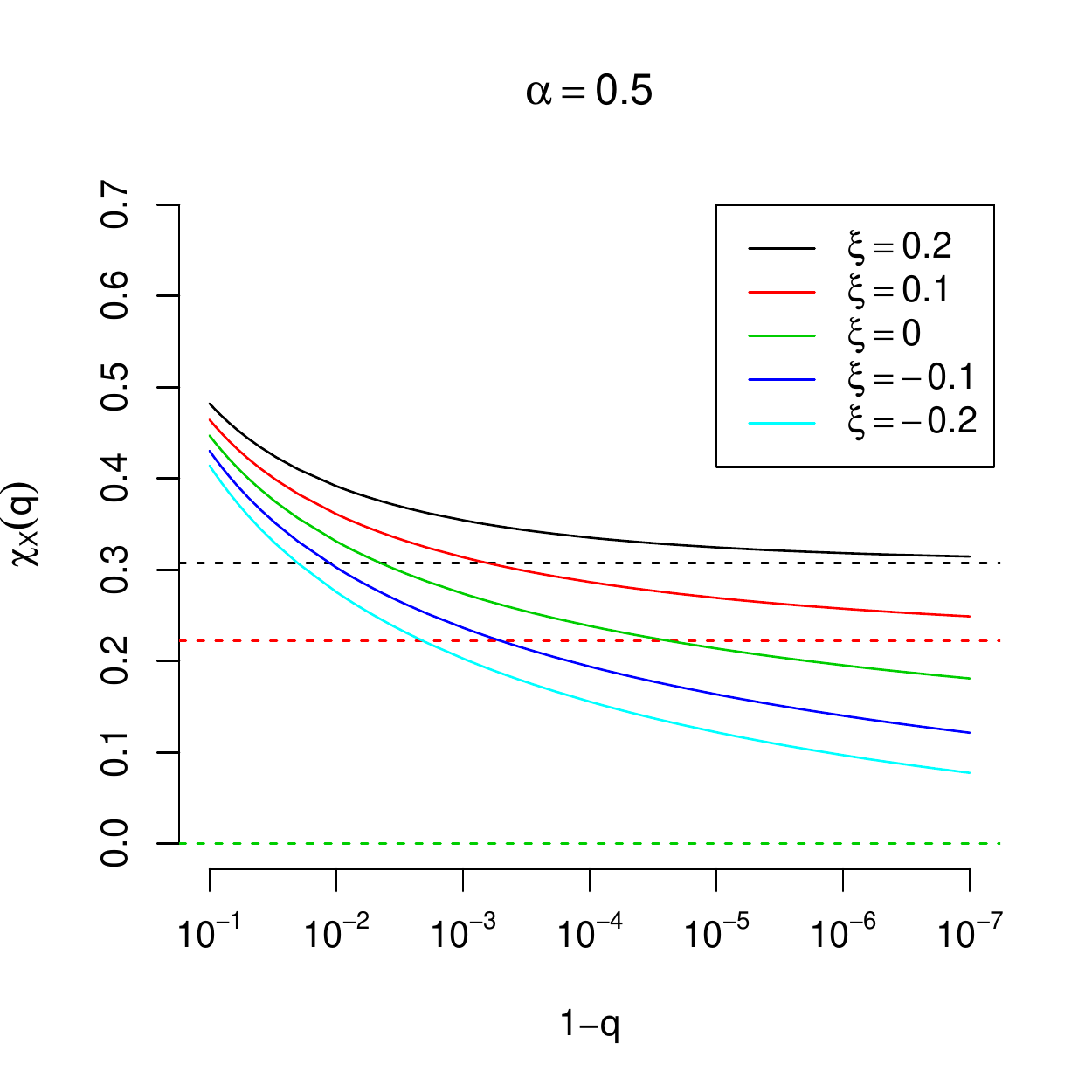}
 \includegraphics[width=0.3\textwidth]{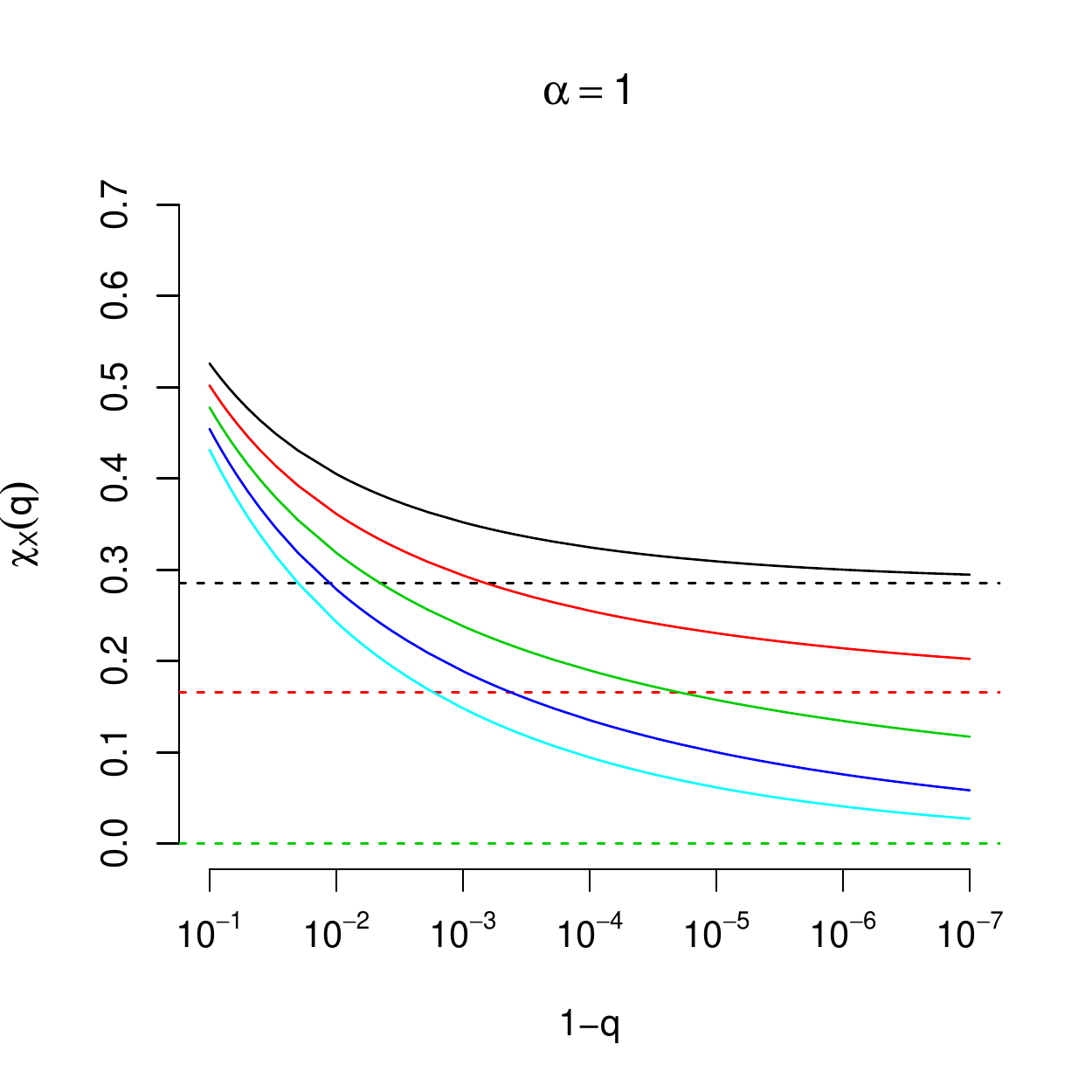}
 \includegraphics[width=0.3\textwidth]{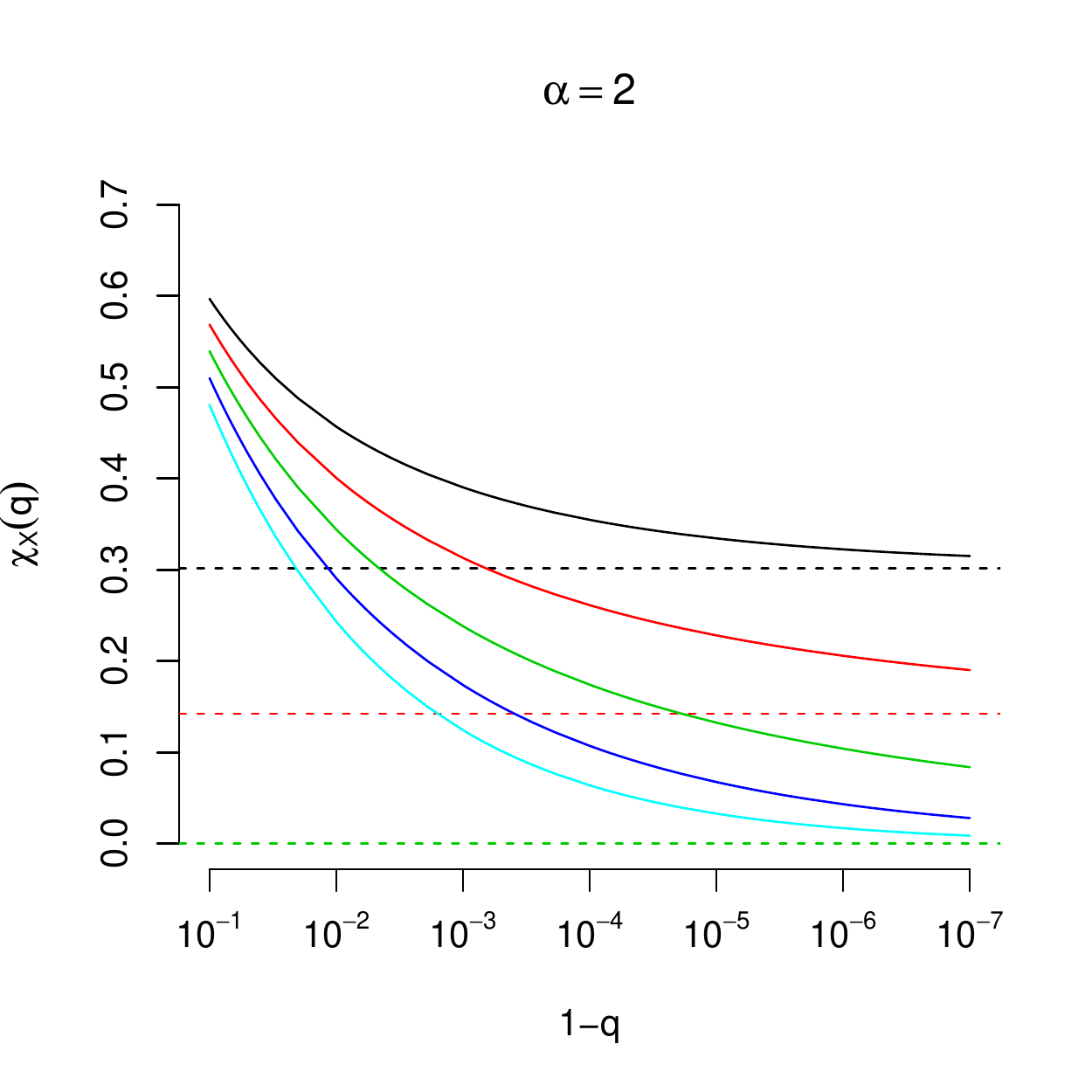}
 \caption{Theoretical $\chi_X(q)$ for Model~\ref{mod:2} plotted against $1-q$ on a logarithmic scale for $q \in [1-10^{-1},1-10^{-7}]$. Different columns show different values of $\alpha$; dashed horizontal lines show the true limiting values of $\chi_X$, which depends on $\xi$ (and is zero for $\xi\leq 0$).}
 \label{fig:model2chis}
\end{figure}

\section{Discussion}
\label{sec:Conclusions}

The paper studies the extremal properties of copulas $\X = R \W$ and determines the tail and residual tail dependence coefficients $\chi_X$ and $\eta_X$, respectively.

In Section~\ref{sec:Constrained}, where $\W$ is constrained to the sphere of some norm, classical results on multivariate Pareto copulas are recovered for regularly varying $R$, whereas new structures are obtained for distributions of $R$ with light tails or finite upper endpoint. In particular, for the Gumbel MDA we get a large variety of behaviors for asymptotically independent $\X$ that strongly depend on the auxiliary function $b$ of $R$ and the shape of the $\nu$-sphere. This extends the results of \citet{Wadsworthetal2017} who considered only the exponential distribution in this class.

For unconstrained distributions of both $R$ and $W$, Section~\ref{sec:Unconstrained} formalizes the general intuition that heavier tails of $R$ introduce more additional dependence in $\X$. The results summarized in Table \ref{tab:indep} for the special case of the independence model allow for several conclusions. The most interesting (and involved) situations figure along the main diagonal where $R$ and $W$ have similar tail behavior. Above this main diagonal, $R$ is so heavy that it mostly dominates the extremal dependence in $\X$. On the other hand, below the diagonal, $R$ is too light tailed, relatively to $W$, to have an impact on the tail dependence coefficients $\chi_X$ and $\eta_X$. Similar observations hold true for the more general case
of arbitrary dependence in $\W$ summarized in Table \ref{tab:unconstrained}. 

We note that there is a clear overlap between the results obtained in Sections~\ref{sec:Constrained} and~\ref{sec:Unconstrained}. If one considers $\chi_W$ as derived from the shape of $\mathcal{W}$, then many results in Section~\ref{sec:Constrained} are obtained from Section~\ref{sec:Unconstrained}, just as Proposition~\ref{prop:Frechet_chi} is relevant in both sections. However, the separate treatment seems justified on the grounds of the importance of such constructions, and the additional insight gained in focusing on the shape of $\mathcal{W}$. 

Multivariate analogs of the upper and residual tail dependence coefficients are obtained by considering the $d$-variate survival function $\P(X_1\geq x_1,\ldots,X_d\geq x_d)$ in~\eqref{eq:chi} and~\eqref{eq:eta}. For random scale constructions in $d$ dimensions, the results from Section~\ref{sec:Unconstrained} are all directly applicable if the $W_j$ components have common margins, since similarly to the bivariate case, we only need to consider the two variables $X_\wedge = R \min(W_1,\ldots,W_d)$ and $X_j = R W_j$. An assumption of common margins is more realistic in spatial models, where dependence is often analyzed in terms of bivariate margins anyway.

The above results provide a general and unifying framework to analyze bivariate extremal dependence, and Section~\ref{sec:Examples} shows that they cover many of the existing examples in the copula and the extreme value literature. Most importantly, combining the insights from different sections enables the construction of numerous new statistical models that smoothly interpolate between asymptotic dependence and independence; see Section~\ref{sec:NewExamples} for two instances. 

Although our focus was on dependence, knowledge on how the marginal scales of $R$ and  $W$ and the dependence properties of $\W$ influence the dependence of $\X$ makes it easier to construct models $\X$ that naturally accommodate both marginal distributions and dependence of multivariate data. Such modeling avoids what may be construed as the artificial separation of modeling of margins and dependence known as copula modeling. For example, in factor constructions based on independent random variables, such as the ones with independent $W_1$ and $W_2$ discussed throughout, our results give guidance on the relative tail heaviness of $R$ with respect to $\W$ necessary to transition from asymptotic independence to asymptotic dependence in $\X$, and both heavy- or light-tailed marginal distributions are possible by considering the distribution of either $\X$ or $\log \X$ as a model for data. 

In Sections~\ref{sec:Constrained} and~\ref{sec:Unconstrained}, we often considered the simplification $W_1\overset{d}{=}W_2$, yielding $X_1\overset{d}{=}X_2$, which allows the coefficients $\chi_X$ and $\eta_X$ to be calculated without reference to marginal quantile functions. A weaker sufficient condition for this is $\bar{F}_{X_1}(x) \sim \bar{F}_{X_2}(x)$ as $x \to x^\star$, with $x^\star$ a common upper end point. To see this sufficiency, define $x_q = \min\{F_{X_1}^{-1}(q),F_{X_2}^{-1}(q)\}$, $x^q = \max\{F_{X_1}^{-1}(q),F_{X_2}^{-1}(q)\}$, and note that 
\begin{align}
 \frac{\P(X_1 \geq x^q,X_2 \geq x^q)}{\max\{\bar{F}_{X_1}(x^q),\bar{F}_{X_2}(x^q)\}}  \leq  \frac{\P(X_1 \geq F_{X_1}^{-1}(q),X_2 \geq F_{X_2}^{-1}(q))}{1-q} \leq \frac{\P(X_1 \geq x_q,X_2 \geq x_q)}{\min\{\bar{F}_{X_1}(x_q),\bar{F}_{X_2}(x_q)\}}, \label{eq:chiineq}
\end{align}
where $\max\{\bar{F}_{X_1}(x^q),\bar{F}_{X_2}(x^q)\} = \min\{\bar{F}_{X_1}(x_q),\bar{F}_{X_2}(x_q)\} = 1-q$. Consequently, the tail dependence coefficient $\chi_X$ of $(X_1,X_2)$, if it exists, is bounded between the limit superior of the left-hand side and the limit inferior of the right-hand side in \eqref{eq:chiineq}, respectively, for $q \to 1$. Whilst these bounds hold in general for common upper end point, they deliver the precise coefficient $\chi_X$ only if $\bar{F}_{X_1}(x) \sim \bar{F}_{X_2}(x)$, $x\to x^\star$ or both limits are zero.
Similar arguments apply to the residual tail dependence coefficient $\eta_X$, where the corresponding bounds
determine $\eta_X$ under the weaker requirement $\log\bar{F}_{X_1}(x) \sim \log\bar{F}_{X_2}(x)$, $x\to x^\star$.

Whilst our focus has been on the coefficients $\chi_X$ and $\eta_X$, we note that there are important aspects of the dependence structure that are not described by these coefficients. For example, in Section~\ref{sec:Constrained}, we found that when $R$ was in the Gumbel or negative Weibull MDA, $\chi_X$ and $\eta_X$ depended only on the shape of $\nu$ and the distribution of $R$, but not at all on the distribution of $Z$. Nonetheless, the latter plays an important role in the behavior of the slowly varying function $\ell$ in~\eqref{eq:eta}, which was exemplified in Figure~\ref{fig:chis}.


\section{Proofs}
\label{sec:Proofs}

This section contains proofs of propositions from Sections~\ref{sec:Constrained}, \ref{sec:Unconstrained} and~\ref{sec:NewExamples}. Proofs of lemmas are deferred to Appendix~\ref{app:AProofs}.
Recall that in the case of common marginal distributions $F_X$ with upper endpoint $x^\star$, the tail dependence coefficient is $ \chi_X=\lim_{x\to x^{\star}} \bar{F}_{X_\wedge}(x)/\bar{F}_X(x)$, whilst the residual tail dependence coefficient is $\eta_X = \lim_{x\to x^\star}\log \bar{F}_{X}(x)/\log \bar{F}_{X_\wedge}(x)$.

\subsection{Proofs for Section~\ref{sec:Constrained}}

\begin{proof}[Proof of Proposition~\ref{prop:Frechet_chi}]
 Since $\E(W_j^{\alpha+\varepsilon})<\infty$, $j=1,2$, 
 Breiman's lemma \citep[][see also Lemma~\ref{lem:breiman} in Appendix~\ref{app:AProofs}]{Breiman1965}
 gives
 \begin{align}
  \bar{F}_{X_j}(x)\sim \E(W_j^\alpha) \bar{F}_R(x),\qquad x\to\infty, \label{eq:Breiman}
 \end{align}
so that $\bar{F}_{X_j} \in \RV_{-\alpha}^\infty$. Now consider the quantile functions of $X_j$ and $R$; denote these by $F_{X_j}^{-1}(q), F_R^{-1}(q)$. Suppose firstly that $\alpha>0$. Taking the reciprocal of relation~\eqref{eq:Breiman}, and using Proposition 2.6~(vi) of \citet{Resnick07}, we have
\begin{align}
 F_{X_j}^{-1}(q) &\sim F_R^{-1}(q)\E(W_j^\alpha)^{1/\alpha}, \qquad q\to 1. \label{eq:quant}
\end{align}
Consider now 
\begin{align*}
 \P(X_1 \geq F_{X_1}^{-1}(q),X_2 \geq F_{X_2}^{-1}(q))
 &= \P\left(R \min\left\{\frac{W_1}{\E(W_1^\alpha)^{1/\alpha}}[1+o(1)],\frac{W_2}{\E(W_2^\alpha)^{1/\alpha}}[1+o(1)]\right\}  \geq F_R^{-1}(q)\right).
\end{align*}
Since $\E(W_j^{\alpha+\varepsilon})<\infty$, $j=1,2$, and $\min(W_1/\E(W_1^\alpha)^{1/\alpha},W_2/\E(W_2^\alpha)^{1/\alpha})^{\alpha+\varepsilon} \leq [W_j/\E(W_j^\alpha)^{1/\alpha}]^{\alpha+\varepsilon}$ for $j=1,2$ we have $\E[\min(W_1/E(W_1^\alpha)^{1/\alpha},W_2/E(W_2^\alpha)^{1/\alpha})^{\alpha+\varepsilon}]<\infty$. By dominated convergence we therefore also have
\[
  \E\left(\min\left\{\frac{W_1}{\E(W_1^\alpha)^{1/\alpha}}[1+o(1)],\frac{W_2}{\E(W_2^\alpha)^{1/\alpha}}[1+o(1)]\right\}^\alpha\right) \to  \E[\min\{W_1^\alpha/E(W_1^\alpha),W_2^\alpha/E(W_2^\alpha)\}],
\]
and so as $q\to 1$
\begin{align*}
\P\left(R \min\left\{\frac{W_1}{\E(W_1^\alpha)^{1/\alpha}}[1+o(1)],\frac{W_2}{\E(W_2^\alpha)^{1/\alpha}}[1+o(1)]\right\}  \geq F_R^{-1}(q)\right) \sim \E\left[\min\left\{\frac{W_1^\alpha}{\E(W_1^\alpha)},\frac{W_2^\alpha}{\E(W_2^\alpha)}\right\}\right] \bar{F}_R(F_R^{-1}(q)),
\end{align*}
from which the result follows. For $\alpha=0$ we have $\bar{F}_{X_j}(x) \sim \bar{F}_R(x) \in\RV_0^{\infty}$, as well as $\bar{F}_{X_\wedge}(x) \sim \bar{F}_R(x)$. Using the bounds in~\eqref{eq:chiineq} and taking limits, we get $\chi_X=1$. As noted after the proposition, the conditions ensure $\chi_X>0$ and hence $\eta_X=1$.
\end{proof}


Before proceeding to the proofs of Propositions~\ref{prop:RGumbel} and~\ref{prop:RNegWeib}, Lemma~\ref{lem:taylor} provides detail on the tail behavior of $W$ and $W_\wedge$, whilst Lemma~\ref{lem:HashorvaThm3} is a reformulation of the relevant components of Theorem~3.1 of \citet{Hashorvaetal2010}, that will be repeatedly useful.

\begin{lemma}\label{lem:taylor}
Assume (Z1), (N1) and (N2). 
  \begin{enumerate}
      \item If $b_1=b_2=1$ then $\bar{F}_W(1-\cdot) \in \RV_{\alpha_W}^0$ with $\alpha_W=\alpha_Z \gamma$. 
    \item If $b_1<1$ then 
    \begin{enumerate}
    \item[a)] $\bar{F}_W(1-\cdot) - \P(W=1) \in \RV_{\alpha_W}^0$ with $\alpha_W=\gamma$. 
      \item[b)] In particular, if $b_1 < 1$, $\tau_1'(b_{1-}), \tau_2'(b_{2+}) \neq 0$, then $\bar{F}_W(1-s) - \P(W=1) = s \ell(s)$, with $\ell$ satisfying
    \begin{align}
      \label{eq:L} \lim_{s\to 0} \ell(s) = f_Z(b_1)/ \tau_1'(b_{1-}) - f_Z(b_2)/\tau_2'(b_{2+})\in (0,\infty),
    \end{align}
    where for $b_2 = 1$ we put $\tau_2'(b_{2+}) = \infty$. In this case, $\bar{F}_W(1-\cdot) - \P(W=1) \in\RV_{\alpha_W}^0$ with $\alpha_W=\gamma=1$.
    \end{enumerate}
  \item
    $\bar{F}_{W_\wedge}(\zeta(1-\cdot)) \in\RV_{1}^0$, with slowly varying function
    $\ell_{\wedge}$ satisfying 
    \begin{align}
      \label{eq:Ltilde}\lim_{s\to 0} \ell_{\wedge}(s) &= 2f_Z(1/2)\zeta/\tau_1'(1/2_-). 
    \end{align}
  \end{enumerate}
\end{lemma}

\begin{lemma}[\citet{Hashorvaetal2010}]
\label{lem:HashorvaThm3}
Let $Y=RS$, where $R\in(0,r^\star)$, $r^\star\in(0,\infty]$, $S\in(0,1)$ and $\bar{F}_S(1-s) = \ell_S(s)s^{\alpha_S}$, $\ell_S \in \RV_0^0$, $\alpha_S>0$. Then
\begin{itemize}
 \item If $R$ is in the Gumbel MDA with auxiliary function $b$,
 \begin{align}
 \bar{F}_Y(x) \sim \Gamma(1+\alpha_S)\bar{F}_S(1-(xb(x))^{-1})\bar{F}_R(x), \qquad x \to r^\star. \label{eq:HashorvaGumbel} 
\end{align}
\item If $R$ is in the negative Weibull MDA,
 \begin{align}
 \bar{F}_Y(r^\star-s) \sim \frac{\Gamma(1+\alpha_S)\Gamma(1+\alpha_R)}{\Gamma(1+\alpha_S+\alpha_R)} \bar{F}_S(1-s/r^\star)\bar{F}_R(r^\star-s), \qquad s \to 0. \label{eq:HashorvaNegWeib} 
\end{align}
\end{itemize}
\end{lemma}
\begin{remark}
\label{rmk:diffendpoint}
 Lemma~\ref{lem:HashorvaThm3} is easily extended to the case where $S$ has upper endpoint $s^\star \in (0,\infty)$, by writing $RS = s^\star R \times S/s^\star$, and noting that: (i) if $R$ is in the Gumbel MDA, with upper endpoint $r^\star$ and auxiliary function $b(t)$, then $s^\star R$ is in the Gumbel MDA with upper endpoint $s^\star r^\star$ and auxiliary function $b(t/s^\star)/s^\star$; (ii) if $R$ is in the negative Weibull MDA, so is $s^\star R$, with $\bar{F}_{s^\star R}(s^\star r^\star - s) = \bar{F}_{R}(r^\star - s/s^\star)$.
\end{remark}

\begin{proof}[Proof of Proposition~\ref{prop:RGumbel}]

For the marginal tail $\bar F_X$ of $X_1$ and $X_2$, we have
 \begin{align*}
 \bar{F}_X(x) &= \int_{[0,1]} \bar{F}_R(x/v) \d F_{W}(v) =\bar{F}_R(x)\P(W=1) + \int_{[0,1)}\bar{F}_R(x/v) \d F_{W}(v),
\end{align*}
with $\P(W=1) = \P(Z \in I_{\nu})$.  Letting $S=W|W<1$, $\bar{F}_S(1-s)=\ell(s)s^{\alpha_W}$, $s\to 0$, where $\alpha_W$ is given by the relevant part of Lemma~\ref{lem:taylor}, and $\ell \in \RV_0^0$. By Lemma~\ref{lem:HashorvaThm3}, we have 
\begin{align}
 \bar{F}_X(x) - \bar{F}_R(x)\P(W=1) \sim \Gamma(1+\alpha_W) \frac{\ell(\{xb(x)\}^{-1})}{\{xb(x)\}^{\alpha_W}}\bar{F}_R(x),\qquad x\to r^\star. \label{eq:marg}
\end{align}
We always have $xb(x)\to\infty$ as $x \to r^\star$, \citep[][Lemma 1.2]{Hashorvaetal2010,Resnick1987}, hence if $\P(W=1) = 0$, $\bar{F}_X(x) = o(\bar{F}_R(x))$, $x\to r^\star$, whilst if $\P(W=1)>0$, $\bar{F}_X(x) \sim \P(W=1)\bar{F}_R(x)$, $x\to r^\star$.

For the joint distribution, $\bar{F}_{X_\wedge}(x) = \P(RW_\wedge \geq x)$, with $W_\wedge \in [0,\zeta]$ for $\zeta \in[1/2,1]$. By Lemma~\ref{lem:taylor}, $\bar{F}_{W_\wedge}(\zeta(1-s)) = \ell_{\wedge}(s)s$, $s\to 0$, whilst by Lemma~\ref{lem:HashorvaThm3} and Remark~\ref{rmk:diffendpoint},
\begin{align*}
 \bar{F}_{X_\wedge}(x) \sim \Gamma(2) \ell_{\wedge}(\{(x/\zeta)b(x/\zeta)\}^{-1})\{(x/\zeta)b(x/\zeta)\}^{-1} \bar{F}_{R}(x/\zeta),\qquad x\to \zeta r^\star. 
\end{align*}
Therefore,
\begin{align}
\chi_X = \lim_{x\to r^\star} \frac{\bar{F}_{X_\wedge}(x)}{\bar{F}_{X}(x)} =\lim_{x \to r^\star} \frac{\Gamma(2) \ell_{\wedge}(\{(x/\zeta)b(x/\zeta)\}^{-1})\{(x/\zeta)b(x/\zeta)\}^{-1}\bar{F}_R(x/\zeta)}{\bar{F}_R(x)\P(W=1) + \Gamma(1+\alpha_W) \ell(\{xb(x)\}^{-1})\{xb(x)\}^{-\alpha_W}\bar{F}_R(x)}, \label{eq:Gumbel_chi}
\end{align}
whilst for $\zeta r^\star = r^\star$, i.e., $\zeta = 1$ or $r^\star = \infty$,
\begin{align}
\eta_X = \lim_{x\to r^\star} \frac{\log \bar{F}_{X}(x)}{\log\bar{F}_{X_\wedge}(x)} &=\lim_{x \to r^\star} \frac{\log \left[\bar{F}_R(x)\P(W=1) + \Gamma(1+\alpha_W) \ell(\{xb(x)\}^{-1})\{xb(x)\}^{-\alpha_W}\bar{F}_R(x)\right] }{ \log \left[\Gamma(2) \ell_{\wedge}(\{(x/\zeta)b(x/\zeta)\}^{-1})\{(x/\zeta)b(x/\zeta)\}^{-1}\bar{F}_R(x/\zeta)\right]} \notag \\ &= \lim_{x\to r^\star} \frac{\log \bar{F}_R(x)}{\log\bar{F}_R(x/\zeta)}. \label{eq:Gumbel_eta}
\end{align}
The latter equality follows by definition~\eqref{eq:GumbelMDA} and l'H\^opital's rule, providing $\lim_{x\to r^\star} \log(x b(x)) / \log\bar{F}_R(x) = 0$.
Now if $\zeta<1$ and $r^\star = \infty$, then $\chi_X = 0$ by~\eqref{eq:DR}, with $\eta_X$ given by~\eqref{eq:Gumbel_eta}. If $r^\star<\infty$, $X_\wedge$ has upper bound $\zeta r^\star < r^\star$ and so $\chi_X = 0$ also, and $\eta_X$ is not defined.

If $\zeta=1$ then $\chi_X = 0$ if $\P(W=1) > 0$ by~\eqref{eq:Gumbel_chi}, whilst $\eta_X = 1$ by~\eqref{eq:Gumbel_eta}. If $\P(W=1)=0$ then $b_1=b_2=1/2$, with Lemma~\ref{lem:taylor} giving $\alpha_W=1$ and
\begin{align*}
\chi_X = \lim_{x \to r^\star} \frac{ \ell_{\wedge}(\{xb(x)\}^{-1})}{\ell(\{xb(x)\}^{-1})} = \frac{2f_Z(1/2)/\tau_1'(1/2_-)}{f_Z(1/2)/ \tau_1'(1/2_{-}) - f_Z(1/2)/\tau_2'(1/2_{+})}.
\end{align*}
The function $\tau$ is not differentiable at $1/2$ when $\zeta=1$ so $\tau_1'(1/2_-) \neq \tau_2'(1/2_+)$ and rearrangement gives the result.
\end{proof}


\begin{proof}[Proof of Proposition~\ref{prop:RNegWeib}]
 Following a similar line to Proposition~\ref{prop:RGumbel}, the marginal distribution satisfies
 \begin{align*}
 \bar{F}_X(r^\star-s) - \bar{F}_R(r^\star-s)\P(W=1) &\sim \frac{\Gamma(1+\alpha_W)\Gamma(1+\alpha_R)}{\Gamma(1+\alpha_W+\alpha_R)} \ell (s/r^\star)(s/r^\star)^{\alpha_W}\bar{F}_R(r^\star-s)\\
 &\sim \tilde{\ell}(s) s^{\alpha_W + \alpha_R},\qquad s\to 0,
\end{align*}
for a new slowly varying function $\tilde{\ell}$, whilst for the joint distribution
\begin{align*}
  \bar{F}_{X_\wedge}(\zeta r^\star-s) &\sim \frac{\Gamma(2)\Gamma(1+\alpha_R)}{\Gamma(2+\alpha_R)} \ell_{\wedge} (s/(\zeta  r^\star))(s/(\zeta r^\star))\bar{F}_R(r^\star-s/\zeta)\\
  &\sim \tilde{\ell}_{\wedge}(s) s^{1 + \alpha_R}, \qquad s\to 0.
\end{align*}

If $\zeta <1$, then $X_\wedge$ has upper bound $\zeta r^\star < r^\star$ and so $\chi_X = 0$, with $\eta_X$ not defined. If $\zeta=1$ and $\P(W=1)> 0$, then $\chi_X = 0$, and
\begin{align*}
\eta_X = \lim_{s \to 0} \frac{\log\bar{F}_X(r^\star-s)}{\log\bar{F}_{X_\wedge}(r^\star-s)} = \lim_{s \to 0} \frac{\log\bar{F}_R(r^\star-s)}{\log\{s\bar{F}_{R}(r^\star-s)\}} = \alpha_R/(1+\alpha_R).
\end{align*}
Otherwise, if $\zeta=1$ and $\P(W=1)=0$, then since $\alpha_W = 1$
\begin{align*}
\bar{F}_{X_\wedge}(r^\star-s) / \bar{F}_{X}(r^\star-s) \sim \ell_{\wedge}(s/r^\star) / \ell(s/r^\star), \qquad s\to 0,
\end{align*}
 with the behavior of $\ell$ and $\ell_{\wedge}$ as given in~\eqref{eq:L} and~\eqref{eq:Ltilde}.
\end{proof}


\subsection{Proofs for Section~\ref{sec:Unconstrained}}
We often write $S=\log R$ and $V=\log W$ and $V_\wedge=\log W_\wedge$ in the following. We denote the common marginal distribution of $X_1$ and $X_2$ by $F_X$. We first recall a result on the convolution with a convolution-equivalent distribution. 

\begin{lemma}[Convolution with a distribution in $\mathrm{CE}_\alpha$, see Theorem~1 of \citet{Cline.1986} and Lemma~5.1 of \citet{Pakes.2004}]\label{lem:conveq}
	Let $Y_1\sim F_1$, $Y_2\sim F_2$ be two random variables. If $F_1\in\mathrm{CE}_\alpha$ with $\alpha \geq 0$ and 
	$\overline{F}_2(x) / \overline{F}_1(x) \rightarrow c\geq 0$, $x\rightarrow\infty$, 
    then 
	\begin{equation}\label{eq:convtail}
	\overline{F_1\star F_2}(x) /\overline{F}_1(x)\rightarrow \E\left(e^{\alpha Y_2}\right)+c\E\left(e^{\alpha Y_1}\right), \qquad x\rightarrow\infty. 
	\end{equation}
\end{lemma}

\begin{proof}[Proof of Proposition \ref{prop:RSHT}]\ \\
1.~The assumption $\overline{F}_{W}(x)/\overline{F}_{R}(x)\rightarrow c$, $x\rightarrow \infty$, is equivalent to $\overline{F}_{V}(x)/\overline{F}_{S}(x)\rightarrow c$. We then apply \eqref{eq:convtail} and obtain $\bar{F}_{X} \sim (1+c)\bar{F}_{R}$. Since $\bar{F}_{V_\wedge}(x)/\bar{F}_{V}(x)\rightarrow \chi_W$, $x\rightarrow\infty$, we get $\bar{F}_{X_\wedge}(x)=\bar{F}_{S+V_\wedge}(\log x)\sim (1+c\chi_W)\bar{F}_S(\log x)=(1+c\chi_W)\bar{F}_R(x)$, and the value $\chi_X$ in \eqref{eq:chiSHT} follows. Since $\chi_X>0$ in all cases, we have $\eta_X=1$. \\
\noindent
2.~Let $F_V\in\mathrm{CE}_0$ and $\overline{F}_R=o(\overline{F}_{W})$. Then, $\overline{F}_S=o(\overline{F}_{V})$, and    \eqref{eq:convtail} gives  $\bar{F}_{X}\sim \bar{F}_{W}$. If $\chi_W>0$, then $\overline{F}_S=o(\overline{F}_{V_\wedge})$, and \eqref{eq:convtail} yields the tail equivalence $\bar{F}_{X_\wedge}(x)\sim \bar{F}_{W_\wedge}(x)$, which entails $\chi_X=\chi_W$. In the case where $\chi_W=0$, we have  $\bar{F}_{V_\wedge}=o(\bar{F}_{V})$ and $\bar{F}_{S}=o(\bar{F}_{V})$, such that \citet[][Corollary 5(ii)]{Fossetal.2009} establishes $\bar{F}_{X_\wedge}\sim \bar{F}_{R}+\bar{F}_{W_\wedge}+o(\bar{F}_{W})$, and $\chi_X=\chi_W=0$ follows.\\
2.a) When further $F_S,F_{V_\wedge}\in\mathrm{CE}_0$ and $\overline{F}_R(x)/\overline{F}_{W_\wedge}(x)\leq C$ with $C>0$ as $x\rightarrow\infty$, we consider the two boundary cases $\overline{F}_R=o(\overline{F}_{W_\wedge})$ and $\overline{F}_R/\overline{F}_{W_\wedge}\sim C$. Using \eqref{eq:convtail}, we get
\begin{align*}
\frac{\log \bar{F}_{X}(x)}{\log\bar{F}_{X_\wedge}(x)}\sim \frac{\log \bar{F}_{W}(x)}{\log(\bar{F}_{R}(x)+\bar{F}_{W_\wedge}(x))}&= \frac{\log \bar{F}_{W}(x)}{\log \bar{F}_{W_\wedge}(x)+\log\left(1+\frac{\bar{F}_{R}(x)}{\bar{F}_{W_\wedge} (x)}\right)} \\
&\sim  \frac{\log \bar{F}_{W}(x)}{\log \bar{F}_{W_\wedge}(x)} \sim \eta_W, \qquad x\rightarrow \infty, 
\end{align*}
in both cases, which proves $\eta_X=\eta_W$. \\
2.b) When further $F_S\in \mathrm{CE}_0$ and $\overline{F}_{W_\wedge}=o(\overline{F}_R)$, we get $\bar{F}_{X_\wedge}(x)\sim \bar{F}_{R}(x)$ from \eqref{eq:convtail}, such that 
$$
\log \bar{F}_{X}(x) / \log \bar{F}_{X_\wedge}(x)\sim \log \bar{F}_{W}(x) / \log \bar{F}_{R}(x), \qquad x\rightarrow \infty, 
$$
and the limit is $\eta_X$ if it exists. 
\end{proof}


\begin{proof}[Proof of Proposition~\ref{prop:WRV}]

  Since $\E(R^{\alpha_W+\varepsilon})<\infty$ for a small $\varepsilon>0$, by  Breiman's lemma,
  the marginal distributions satisfy 
 \begin{align*}
  \bar{F}_{X}(x)\sim \E(R^{\alpha_W})\bar{F}_{W}(x),\qquad x\to\infty, 
 \end{align*}
 so that $X$ is also regularly varying with index $\alpha_W$. 
The coefficient $\eta_W$ is not defined if $W_\wedge$ has a finite upper endpoint, in which case $\E(W_{\wedge}^{\alpha_R+\varepsilon})<\infty$. Otherwise, we have that $W_\wedge$ is regularly varying with index $\alpha_W / \eta_W$, or if $\eta_W=0$, $\bar{F}_{W_\wedge}$ decays faster than any power. Again, by   Breiman's lemma \citep{Breiman1965}, we obtain that 
 \begin{align*}
   \bar{F}_{X_\wedge}(x)
   &\sim 
   \begin{cases}
     \E(W_\wedge^{\alpha_R})\bar{F}_{R}(x) , \quad & \text{if } \alpha_R  < \alpha_W / \eta_W, \eta_W=0 \mbox{ or $\eta_W$ not defined},\\
     \E(R^{\alpha_W / \eta_W})\bar{F}_{W_\wedge}(x) , \quad & \text{if } \alpha_R  > \alpha_W / \eta_W \text{ or } \alpha_R = +\infty.
   \end{cases}
\end{align*}
If $\chi_W>0$ then $\eta_W=1$, so we conclude that $\chi_X = \lim_{x\to \infty} \bar{F}_{X_\wedge}(x) / \bar{F}_{X}(x) = \chi_W$. 
For the coefficient of residual tail dependence between $X_1$ and $X_2$ we obtain
 \begin{align*}
 1 / \eta_X &= \lim_{x\to\infty}\log \bar{F}_{X_\wedge}(x)/\log\bar{F}_{X}(x)\\
 &=
   \begin{cases}
     \lim_{x\to\infty} \log \bar{F}_{R}(x) / \log \bar{F}_{W}(x) = \alpha_R/\alpha_W , \quad & \text{if }   \alpha_R  < \alpha_W / \eta_W, \mbox{ or $\eta_W$ not defined},\\
      \lim_{x\to\infty} \log \bar{F}_{W_\wedge}(x) / \log \bar{F}_{W}(x) =  1 /\eta_W  , \quad & \text{if }  \alpha_R  > \alpha_W / \eta_W \text{ or } \alpha_R = +\infty.
   \end{cases}
\end{align*}
\end{proof}
Before the proof of Proposition~\ref{prop:samealpha}, we recall two lemmas from the literature.
\begin{lemma}[Convolution of distributions in $\mathrm{ET}_{\alpha,\beta>-1}$, see Theorem~4(v) of \citet{Cline.1986}]\label{lem:gamma}
	For two distributions $F_i\in\mathrm{ET}_{\alpha,\beta_i}$, $i=1,2$, possessing gamma-type tail
	\begin{equation*}
	\overline{F}_i(x) \sim \ell_i(x) x^{\beta_i}\exp(-\alpha x), \quad x\rightarrow\infty \quad \alpha>0,\ \beta_i>-1, \quad i=1,2,
	\end{equation*}
	with slowly varying $\ell_i(x)$, 
	we get 
	\begin{equation*}
	\overline{F_1\star F_2}(x)\sim  \alpha \frac{\Gamma(\beta_1+1)\Gamma(\beta_2+1)}{\Gamma(\beta_1+\beta_2+2)} \ell_1(x)\ell_2(x)x^{\beta_1+\beta_2+1}\exp(-\alpha x), \quad x\rightarrow\infty.  
	\end{equation*}
\end{lemma}

\begin{lemma}[Ratio of convolutions with a distribution in $\mathrm{ET}_{\alpha}$]\label{lem:etratio}
Let $F\in \mathrm{ET}_\alpha$ with $\alpha>0$, and let $G_1,G_2$ be distributions satisfying $\bar{G}_2(x)/\bar{G}_1(x)\rightarrow c$, $x\rightarrow\infty$, with $c\geq 0$. Given a distribution $H$, we write $M_{H}(\alpha)=\int \exp(\alpha x) H(\mathrm{d}x)\in(0,\infty]$. 
\begin{enumerate}
\item If $G_1\in\mathrm{ET}_\alpha$, $c>0$ and  $M_{F}(\alpha)=M_{G_1}(\alpha)=\infty$, then
\begin{equation*}
\bar{F\star G_2}\sim c\,\bar{F\star G_1}, \qquad x\rightarrow\infty.
\end{equation*}
\item If  $M_{G_1}(\alpha)<\infty$ and $\bar{F\star G_1}\sim M_{G_1}(\alpha) \bar{F}$, then
\begin{equation*}
\bar{F\star G_2} / \bar{F\star G_1}\sim M_{G_2}(\alpha) / M_{G_1}(\alpha), \qquad x\rightarrow\infty.
\end{equation*}
\end{enumerate}
\end{lemma}

\begin{proof}[Proof of Proposition~\ref{prop:samealpha}]\ \\
1.~To show $\eta_X=1$, we exploit the closure of $\mathrm{ET}_\alpha$  under convolution  \citep[Lemma 2.5,][]{Watanabe.2008}; equivalently, $\mathrm{RV}_\alpha$ is closed under product convolutions. Applying \citet[][Proposition~2.6(i)]{Resnick07} yields
\begin{equation}\label{eq:eta1}
\log \bar{F}_{S+V}(x) / \log \bar{F}_{S}(x) \sim \alpha x / (\alpha x)=1.
\end{equation}
It remains to show that the limit in \eqref{eq:eta1} does not change when we substitute $ \log \bar{F}_{S+V_\wedge}(x)$ for $\log \bar{F}_{S}(x)$ in the denominator, where no additional assumption on the distribution of $V_\wedge$ is made. 
 If $p^+_\wedge=\bar{F}_{V_\wedge}(0)>0$, we get $\overline{F}_{S+V_\wedge}(x)\geq p_\wedge^+\bar{F}_{S}(x)$ for $x>0$, and then
 $$
 \log \overline{F}_{S+V}(x)/\log \overline{F}_{S+V_\wedge}(x) \geq \log \overline{F}_{S+V}(x)/(\log \overline{F}_{S}(x)+\log p_\wedge^+) \sim 1,
 $$
such that $\eta_X=1$. If $p^+_\wedge=0$, then  Breiman's lemma \citep{Breiman1965} gives $\overline{F}_{S+V_\wedge}\sim \E(W_\wedge^\alpha)\overline{F}_{S}$, and  $\eta_X=1$ follows by analogy with \eqref{eq:eta1}.\\ 
\noindent
2.~ In the case where $F_S\in \mathrm{CE}_\alpha$ and $\overline{F}_V(x)/\overline{F}_S(x)\rightarrow c \geq 0$, $x\rightarrow\infty$, we can use Lemma~\ref{lem:conveq} with $F_1=F_S$ and $F_2=F_V$, which yields
$$
\bar{F}_{S+V}(x)/\bar{F}_S(x) \sim \E\left(e^{\alpha V}\right)+c\E\left(e^{\alpha S}\right)=\E\left(W^\alpha\right)+c\E\left(R^\alpha\right)
$$
and by setting $F_2=F_{V_\wedge}$, we get 
$$
\bar{F}_{S+V_\wedge}(x)/\bar{F}_S(x) \sim \E\left(e^{\alpha V_{\wedge}}\right)+c\chi_W\E\left(e^{\alpha S}\right)=\E\left(W_\wedge^\alpha\right)+c\chi_W\E\left(R^\alpha\right).
$$
Combining these two results yields the value of $\chi_X$.\\
\noindent
3.~If $F_V\in\mathrm{CE}_\alpha$ and $\overline{F}_S(x)/\overline{F}_V(x)\rightarrow 0$, $x\rightarrow\infty$, we use  Lemma~\ref{lem:conveq} to show
$\bar{F}_{S+V}(x)/\bar{F}_V(x)\sim \E(e^{\alpha S})=\E(R^\alpha)$. If $\chi_W>0$, then also $F_{V_\wedge} \in\mathrm{CE}_\alpha$ and we have $\bar{F}_{S+V_\wedge}(x)/\bar{F}_{V_\wedge}(x)\sim \E(R^\alpha)$ by analogy. By combining these two results, we get  $\bar{F}_{S+V_\wedge}(x)/\bar{F}_{S+V}(x) \sim \bar{F}_{V_\wedge}(x)/\bar{F}_{V}(x) \sim \chi_W$, such that $\chi_X=\chi_W$.\\
If $\chi_W=0$, then $\bar F_{V_\wedge}(x) / \bar F_V(x) \to  0, x\to\infty$. Consider the copula $(V_1^\epsilon,V_2^\epsilon)$ defined as the mixture
of $(V_1,V_2)$ and $(V,V)$ with probabilities $(1-\epsilon)$ and $\epsilon$, respectively, for some $0<\epsilon < 1$.
The marginal distribution of $(V_1^\epsilon,V_2^\epsilon)$ is still $F_V$, 
and the induced $V^\epsilon_\wedge$ satisfies
$\bar F_{V_\wedge^\epsilon}(x) = (1-\epsilon)\bar F_{V_\wedge}(x) + \epsilon \bar F_{V}(x) \geq \bar F_{V_\wedge}(x)$.
Therefore, $\bar F_{V_\wedge^\epsilon}(x) / \bar F_V(x) \to  \epsilon >0$ and $\bar F_S(x)  /\bar F_{V_\wedge^\epsilon}(x)\to  0$ as  $x\to\infty$.
Since $V_\wedge^\epsilon$ is stochastically larger than $V_\wedge$, this means that also $Y + V_\wedge^\epsilon$ is stochastically larger than $Y + V_\wedge$ for any
random variable $Y$, by a coupling argument. Thus,
$$ \chi_X = \lim_{x\to\infty} \frac{\bar F_{S + V_\wedge}(x)}{\bar F_{S+V}(x)} \leq \lim_{x\to\infty} \frac{\bar F_{S + V_\wedge^\epsilon}(x)}{\bar F_{S+V}(x)} =  \lim_{x\to\infty}  \frac{\bar F_{V_\wedge^\epsilon}(x)}{\bar F_{V}(x)} = \epsilon,$$
where the last but one equation follows from the former case where $\chi_W > 0$ (with $V_\wedge^\epsilon$ taking the role of $V_\wedge$). Since $\epsilon>0$ is arbitrary, the result follows.\\
\noindent
4.a)~We consider the case $\chi_W>0$ where $\beta_W>-1$ or $\beta_W=-1$ and $\E(W^{\alpha})=\infty$. We can apply Lemma~\ref{lem:etratio}(1), which shows $\chi_X=\chi_W$.\\
4.b)~We consider the case $\chi_W\geq 0$ where $\beta_W<-1$ or $\beta_W=-1<\beta_R$ and $\E\left(W^\alpha\right)<\infty$. Then, Theorem 4(iv) of \citet{Cline.1986} yields $\bar{F}_{S+V}(x)/\bar{F}_S(x)\sim \E\left(W^\alpha\right)$. Therefore, we can apply Lemma~\ref{lem:etratio}(2) with $c=\chi_W$, such that 
$
\chi_X= \E\left(W_\wedge^\alpha\right)/\E\left(W^\alpha\right)
$.
4.c)~Here, $\beta_R,\beta_W>-1$ so $S$, $V$ are gamma-tailed. Since we assume 
$\E\left(W_\wedge^{\alpha+\varepsilon}\right)<\infty$, 
  Breiman's lemma \citep{Breiman1965} provides $\bar{F}_{S+V_\wedge}(x) \sim \E\left(W_\wedge^{\alpha}\right)\bar{F}_S(x)$, whilst by Lemma~\ref{lem:gamma}, \[\bar{F}_{S+V}(x) \sim \alpha \frac{\Gamma(\beta_{W}+1) \Gamma(\beta_R+1)}{\Gamma(\beta_{W}+\beta_R+2)} \ell_{W}(x) x^{\beta_{W}+1} \bar{F}_S(x),\]
  so as $x\to \infty$, \[\chi_X \sim \frac{\E\left(W_\wedge^{\alpha}\right)\Gamma(\beta_R + \beta_W+2)}{\alpha \Gamma(\beta_{W}+1) \Gamma(\beta_R+1) \ell_{W}(x) x^{\beta_{W}+1} } \to 0.\]
\end{proof}



\begin{proof}[Proof of Proposition \ref{prop:lightS}]
 	 1. If  $\chi_W>0$, then necessarily $\alpha_{W}=\alpha_{W_\wedge}$, $\gamma_{W}=\gamma_{W_\wedge}$ and  $\chi_W=c_{W_\wedge}/c_W$ where $\ell_{W} \sim c_W$ and $\ell_{W_\wedge}\sim  c_{W_\wedge}$. Using \citet[Theorem 3.1 and Theorem 4.1(iii)]{Asmussen.al.2017}, we therefore get $\chi_X=k_{X_\wedge}/k_{X}$, where $k_{.}$ are the constants made explicit in the cited paper for the sums of the Weibull-tailed random variables $S+V$ and $S+V_\wedge$ respectively. By simplifying the resulting expression of $\chi_X$, which follows from the equalities $\alpha_{W}=\alpha_{W_\wedge}$, $\gamma_{W}=\gamma_{W_\wedge}$, we obtain $\chi_X=c_{W_\wedge}/c_W=\chi_W$.\\
 	 \noindent
2. When $\chi_W=0$, \citet[Theorem 3.1 and Theorem 4.1(iii)]{Asmussen.al.2017} yield
 \begin{equation}\label{eq:weibsum}
 \frac{\bar{F}_{\log X}(x)}{c_{X}x^{\gamma_{X}}\exp\left(-\alpha_{X}x^{\beta}\right)}\sim 1, \  \gamma_{X}=\gamma_R+\gamma_W+\beta/2, \  \alpha_{X}=\frac{\alpha_R\alpha_W^{\beta/(\beta-1)}+\alpha_W\alpha_R^{\beta/(\beta-1)}}{\left(\alpha_R^{1/(\beta-1)}+\alpha_W^{1/(\beta-1)}\right)^{\beta}},
 \end{equation}
 with a constant $c_{X}>0$ that can be made explicit, see \citet[Theorem 3.1]{Asmussen.al.2017}.  
 Moreover, replacing the symbol $W$ in \eqref{eq:weibsum} with $W_\wedge$ yields the tail approximation of $\log X_\wedge=S+V_\wedge$. Applying Lemma~\ref{lem:weib} to the random vector $(X_1,X_2)$, we get 
 \begin{align*}
 \eta_X&=\frac{\alpha_{X}}{\alpha_{X_\wedge}}= \frac{\left(\alpha_R\alpha_W^{\beta/(\beta-1)}+\alpha_W\alpha_R^{\beta/(\beta-1)}\right) \left(\alpha_R^{1/(\beta-1)}+\alpha_{W_\wedge}^{1/(\beta-1)}\right)^{\beta}}{\left(\alpha_R\alpha_{W_\wedge}^{\beta/(\beta-1)}+\alpha_{W_\wedge}\alpha_R^{\beta/(\beta-1)}\right) \left(\alpha_R^{1/(\beta-1)}+\alpha_W^{1/(\beta-1)}\right)^{\beta}} \\
& =\frac{\alpha_W}{\alpha_{W_\wedge}}\times \frac{\alpha_W^{1/(\beta-1)}+\alpha_R^{1/(\beta-1)}}{\alpha_{W_\wedge}^{1/(\beta-1)}+\alpha_R^{1/(\beta-1)}} \left(\frac{1+\left(\frac{\alpha_{W_\wedge}}{\alpha_R} \right)^{1/(\beta-1)}}{1+\left(\frac{\alpha_{W}}{\alpha_R}\right)^{1/(\beta-1)}}\right)^{\beta}.
 \end{align*}
By substituting  ${\alpha_W}/{\alpha_{W_\wedge}}=\eta_W$, and simplifying, equation~\eqref{eq:lightS} follows. 

 If $\chi_W=0$ and $\eta_W<1$, then $\eta_X<1$, implying $\chi_X=0$. If $\eta_W=1$ but $\chi_W=0$, then Lemma~\ref{lem:weib}\eqref{chiwlwt} with the assumption of asymptotically constant slowly varying functions implies $\gamma_{W_\wedge}<\gamma_W$.  Combining~\eqref{eq:weibsum} and Lemma~\ref{lem:weib}\eqref{chiwlwt}, $\gamma_{X_\wedge}<\gamma_X$ and $\chi_X=0$ also.
 \end{proof}

For the proof of Proposition~\ref{prop:weib}, we need the following Lemma.

\begin{lemma}[Tail decay of products of Weibull-type variables, see Theorem~2.1(b) of \citet{Debicki.al.2018}]
	\label{lem:weibprod}
	If two independent random variables $Y_1\geq 0$ and $Y_2\geq 0$ are Weibull-tailed such that 
	\begin{equation}
	\bar{F}_{Y_j}(x) \sim r_j(x) \exp(-\alpha_{j} x^{\beta_j}), \quad x\to\infty, \quad j=1,2, \label{eq:genWeibtail}
	\end{equation}
with $r_j \in \RV_{\gamma_j}^{\infty}$.	Then, as $x\rightarrow\infty$, 
	\begin{equation*}
	\bar{F}_{Y_1Y_2}(x) \sim  C_0 x^\frac{\beta_1\beta_2}{2(\beta_1+\beta_2)} r_1\left(C_1^{-1}x^{\frac{\beta_2}{\beta_1+\beta_2}}\right) r_2\left(C_1 x^\frac{\beta_1}{\beta_1+\beta_2}\right) \exp\left(-C_2 x^{\frac{\beta_1\beta_2}{\beta_1+\beta_2}} \right),
	\end{equation*}
	with constants $C_1=(\alpha_1\beta_1/(\alpha_2\beta_2))^{1/(\beta_1+\beta_2)}$, $C_0=\left(2\pi\alpha_2\beta_2/(\beta_1+\beta_2)\right)^{1/2}C_1^{\beta_2/2}$  and $C_2=\alpha_1C_1^{-\beta_1}+\alpha_2C_1^{\beta_2}$. For $\beta_1=\beta_2=\beta$ this simplifies to 
	\begin{equation*}
	\bar{F}_{Y_1Y_2}(x) \sim  \sqrt{\pi}\left(\alpha_1\alpha_2x^\beta\right)^{1/4}  r_1\left((\alpha_2/\alpha_1)^{1/(2\beta)}x^{1/2}\right) r_2\left((\alpha_1/\alpha_2)^{1/(2\beta)} x^{1/2}\right) \exp\left(-2\sqrt{\alpha_1\alpha_2} x^{\beta/2} \right).
	\end{equation*}
    Consequently $\bar{F}_{Y_1 Y_2}$ is of the form~\eqref{eq:genWeibtail} with $r_{12} \in \RV_{\gamma_{12}}^{\infty}$, with $\gamma_{12} = (\beta_1\beta_2 + 2\gamma_1\beta_2 + 2\gamma_2\beta_1)/\{2(\beta_1+\beta_2)\}$.
\end{lemma}

\begin{proof}[Proof of Proposition \ref{prop:weib}]

  Consider the functions $h_j(x)=\ell_j(x)x^{\gamma_j}\exp(-\alpha_jx^{\beta_j})$ with $\ell_j\in \RV_{0}^\infty$,  $j=1,2$. One easily shows that 
  \begin{equation}\label{eq:hlim}
    h_1(x)/h_2(x)\rightarrow 0, \qquad x\rightarrow\infty,
\end{equation}
  if $\beta_1>\beta_2$, or if $\beta_1=\beta_2$ and $\alpha_1>\alpha_2$, or if $\beta_1=\beta_2$ and $\alpha_1=\alpha_2$ and $\gamma_1 < \gamma_2$.

  Since $\bar{F}_{W_\wedge}(x)/\bar{F}_{W}(x)\leq 1$, we always have $\beta_W\leq \beta_{W_\wedge}$; moreover, $\beta_W=\beta_{W_\wedge}$ implies $\alpha_W\leq \alpha_{W_\wedge}$, and   $\beta_W=\beta_{W_\wedge}$, $\alpha_W= \alpha_{W_\wedge}$ implies $\gamma_W \leq \gamma_{W_\wedge}$.
  
  Using Lemma \ref{lem:weibprod} for $X=RW$, we find that $F_X\in\mathrm{WT}_{\beta_X}$ with parameters
  $\beta_{X} = \beta_R \beta_W/(\beta_R+\beta_W)$, 
  $ \alpha_{X} = \alpha_R C_1^{-\beta_R} + \alpha_W C_1^{\beta_W}$, 
  where $C_1 = \{\alpha_R\beta_R/(\alpha_W\beta_W)\}^{1/(\beta_R+\beta_W)}$, 
  \begin{equation}
  \label{eq:prodWeibgamma}
  \gamma_{X}=(\beta_R\beta_W + 2\gamma_R\beta_W + 2\gamma_W\beta_R)/\{2(\beta_R+\beta_W)\},
  \end{equation}
  and slowly varying function $\ell_{X}> 0$.  The same result applies to $X_\wedge=RW_\wedge$ with constants $\beta_{X_\wedge}$, $\alpha_{X_\wedge}$, $\gamma_{X_\wedge}$, and slowly varying $\ell_{X_\wedge}>0$.
  \\
\noindent
1.~In this case, it follows from Lemma~\ref{lem:weibprod} that also $\beta_{X_\wedge} = \beta_{X}$, $\alpha_{X_\wedge} = \alpha_{X}$,
  and $\gamma_{X_\wedge} = \gamma_{X}$, and thus
  \begin{align}\label{chi1}
    \chi_X = \lim_{x\rightarrow\infty} \bar{F}_{X_\wedge}(x) / \bar{F}_{X}(x) = \lim_{x\rightarrow\infty} \ell_{W_\wedge}(x) / \ell_W(x)
 = \chi_W\in [0,1]
  \end{align}
since all other dominating terms of higher order cancel out, and the further results follow straightforwardly. 
\\
\noindent 
2.~
Since $\gamma_{W_\wedge} < \gamma_W$, equation~\eqref{eq:prodWeibgamma} implies $\gamma_{X_\wedge} < \gamma_X$, whilst $\beta_X = \beta_{X_\wedge}$ and $\alpha_X = \alpha_{X_\wedge}$. Similarly to~\eqref{chi1}, we therefore obtain $\chi_X = \lim_{x\rightarrow\infty} \{\ell_{X_\wedge}(x) / \ell_{X}(x)\} x^{\gamma_{X_\wedge}-\gamma_{X}}= 0 =  \chi_W$.
 On the other hand, we have
  \begin{equation*}
    \eta_X=\lim_{x\rightarrow\infty} \frac{\log \bar{F}_{X}(x)}{\log \bar{F}_{X_\wedge}(x)} = \lim_{x\rightarrow\infty} \frac{\alpha_{X} x^{\beta_{X}} } {\alpha_{X_\wedge} x^{\beta_{X_\wedge}}}  = 1 = \lim_{x\rightarrow\infty} \frac{\alpha_{W} x^{\beta_{W}} } {\alpha_{W_\wedge} x^{\beta_{W_\wedge}}} = \eta_W. 
  \end{equation*}
\\
\noindent
3.~Since $\beta_{W_\wedge} = \beta_W$ and $\alpha_{W_\wedge} > \alpha_W$, it follows from \eqref{eq:hlim}
  that $\chi_W = 0$, and that $\eta_W=\alpha_W/\alpha_{W_\wedge}$. We clearly have that $\beta_{X} = \beta_{X_\wedge}$, and 
  \begin{align*}
    \frac{\alpha_{X}}{\alpha_{X_\wedge}} &= \frac{\alpha_{R}\left(\frac{\alpha_R\beta_R}{\alpha_W\beta_W}\right)^{-\beta_R/(\beta_R+\beta_W)}+\alpha_{W}\left(\frac{\alpha_R\beta_R}{\alpha_W\beta_W}\right)^{\beta_W/(\beta_R+\beta_W)} }{\alpha_{R}\left(\frac{\alpha_R\beta_R}{\alpha_{W_\wedge}\beta_{W_\wedge}}\right)^{-\beta_R/(\beta_R+\beta_{W_\wedge})}+\alpha_{{W_\wedge}}\left(\frac{\alpha_R\beta_R}{\alpha_{W_\wedge}\beta_{W_\wedge}}\right)^{\beta_{W_\wedge}/(\beta_R+\beta_{W_\wedge})}}\\
    &= \left(\frac{\alpha_W}{\alpha_{W_\wedge}}\right)^{\beta_R/(\beta_R+\beta_{W})}\in (0,1),
  \end{align*}
  after some algebra and using that $\beta_{W_\wedge} = \beta_W$. Thus,  ${\alpha_{X_\wedge}}> {\alpha_{X}}$, and by \eqref{eq:hlim} we conclude that $\chi_X = 0$, and further
  \begin{equation*}
    \eta_X=\lim_{x\rightarrow\infty} (\alpha_{X}/\alpha_{X_\wedge}) x^{\beta_{X}-\beta_{X_\wedge}}  =  \left(\alpha_W /\alpha_{W_\wedge}\right)^{\beta_R/(\beta_R+\beta_{W})} \neq \eta_W.
  \end{equation*}
\noindent
4.~For the case $\beta_W < \beta_{W_\wedge}$, we also have $\beta_{X}=\beta_R\beta_W/(\beta_W+\beta_R)<\beta_R\beta_{W_\wedge}/(\beta_{W_\wedge}+\beta_R)=\beta_{X_\wedge}$. It therefore follows from \eqref{eq:hlim} that $\chi_X = \chi_W = 0$ and $\eta_X = \eta_W = 0$.
\end{proof}

\begin{proof}[Proof of Proposition~\ref{prop:negwei}]
\noindent
 1. Applying Lemma \ref{lem:HashorvaThm3} to compute $\bar{F}_{X}(x)$ and $\bar{F}_{X_\wedge}(x)$ gives
    $$\chi_X = \lim_{x\to r^\star} \frac{ \Gamma(1+\alpha_{W_\wedge}) \ell_{W_\wedge}(\{xb_R(x)\}^{-1}) \{xb_R(x)\}^{-\alpha_{W_\wedge}} \bar F_R(x)}{\Gamma(1+\alpha_W) \ell_W(\{xb_R(x)\}^{-1}) \{xb_R(x)\}^{-\alpha_W} \bar F_R(x)} = \lim_{s\to 0} \frac{ \Gamma(1+\alpha_{W_\wedge}) \ell_{W_\wedge}(s) s^{\alpha_{W_\wedge}} }{\Gamma(1+\alpha_W) \ell_W(s) s^{\alpha_W} } =\chi_W,$$
    since $xb_R(x) \to \infty$ for $x\to r^\star$, $\alpha_{W_\wedge}\geq \alpha_{W}$, and if $\chi_W > 0$ then necessarily $\alpha_{W_\wedge} = \alpha_{W}$. Similarly,
    $\eta_X = \lim_{x\to r^\star} \log\bar{F}_{X}(x) / \log\bar{F}_{X_\wedge}(x) = 1$,
    since, by l'H\^opital's rule $\lim_{x\to r^\star} \log(x b_R(x)) / \log\bar{F}_R(x) = 0$, and therefore the term $\log \bar F_R(x)$ dominates both the numerator and the denominator.\\
    \noindent
 2. We observe that for some $z<w^\star$ and $K>0$, 
    $$ \chi_W = \lim_{x\to w^\star} \frac{\bar{F}_{W_\wedge}(x)}{\bar{F}_{W}(x)} = \lim_{x\to w^\star} K 
    \exp\left\{ - \int_z^x b_{W_\wedge}(t)\left[ 1 - b_W(t) / b_{W_\wedge}(t)\right] \mathrm{d} t \right\}. $$
    If $\chi_W>0$ then we must have $\lim_{x\to w^\star} b_W(x) / b_{W_\wedge}(x) = 1$, since
    $\int_z^{w^\star} b_{W_\wedge}(t)\mathrm{d} t = \infty$.
    Therefore, by Lemma~\ref{lem:HashorvaThm3},
    \begin{align*}
      \chi_X = \lim_{x\to w^\star} \frac{\ell_R(\{x  b_{W_\wedge} (x)\}^{-1})\{x b_{W_\wedge}(x)\}^{-\alpha_R}\bar{F}_{W_\wedge}(x)}{\ell_R(\{xb_W(x)\}^{-1})\{xb_W(x)\}^{-\alpha_R}\bar{F}_{W}(x)} = \chi_W.
    \end{align*}
    On the other hand, if $\chi_W = 0$, then 
    $\eta_X = \lim_{x\to w^\star} \log \bar{F}_{W}(x) / \log \bar{F}_{W_\wedge}(x) = \eta_W$.
\noindent
 3. The upper endpoint of both $X$ and $X_{\wedge}$ is $x^\star=r^\star w^\star$. Lemma \ref{lem:HashorvaThm3} and Remark~\ref{rmk:diffendpoint}  yield
    $$\bar{F}_{X}(x^\star - s) \sim \frac{\Gamma(1+\alpha_R)\Gamma(1+\alpha_W)}{\Gamma(1+\alpha_R+\alpha_W)} \bar{F}_W(w^\star-s/r^\star)\bar{F}_R(r^\star-s/w^\star) \sim \frac{\Gamma(1+\alpha_R)\Gamma(1+\alpha_W)}{\Gamma(1+\alpha_R+\alpha_W)}  \ell_R(s) \ell_W(s) s^{\alpha_R + \alpha_W},$$
   as $s\to 0$, and similarly for $\bar{F}_{X_\wedge}(x^\star - s)$. Therefore, if $\alpha_{W_\wedge} = \alpha_W$,
    then 
    $$\frac{\bar{F}_{X_\wedge}(x^\star - s)}{\bar{F}_{X}(x^\star - s)} \sim \frac{\bar{F}_{W_\wedge}(w^\star - s)}{\bar{F}_{W}(w^\star - s)} \sim \frac{\ell_{W_\wedge} (s)}{\ell_{W} (s)} s^{\alpha_{W_\wedge} - \alpha_W} ,$$
    and consequently, $\chi_X = \chi_W$. If  $\alpha_{W_\wedge} > \alpha_W$, then $\chi_X=\chi_W =0$.
    For the residual tail dependence coefficient we only need to keep the dominating terms, and we compute
    $\eta_X = \lim_{s\to 0} \log \bar{F}_{X}(x^\star - s) / \log \bar{F}_{X_\wedge}(x^\star - s) = (\alpha_W + \alpha_R)/(\alpha_{W_\wedge} + \alpha_R).$
\end{proof}

\subsection{Proofs for Section~\ref{sec:NewExamples}}

\begin{proof}[Proof of Proposition~\ref{prop:constrainedmodel}]
 For $\theta>1$, the norm $\nu$ is as in Example~\ref{ex:AD}, and the result for $\chi_X$ is derived using Proposition~\ref{prop:RGumbel}. For $\theta<1$, $\nu$ requires rescaling to meet our requirement that $\nu(x,y) \geq \max(x,y)$ with equality somewhere. Note that scaling by any constant $K\in(0,\infty)$ does not affect the dependence, i.e., $\X = R (\tau(Z),\tau(1-Z))$ has the same copula as $\X = R (\tau(Z),\tau(1-Z))/ K$. For $\theta\in[1/2,1]$ we therefore define $\nu^*(x,y) = \theta^{-1}\nu(x,y)$ which satisfies the required scaling, so that $\zeta=1/\nu^*(1,1) = \theta$. The result for $\eta_X$ is then also given by Proposition~\ref{prop:RGumbel}.
 \end{proof}

\begin{proof}[Proof of Proposition~\ref{prop:unconstrainedmodel}]
 The first claim follows directly from Proposition~\ref{prop:Frechet_chi}, since $\E(W^{1/\xi+\varepsilon})<\infty$. Since $\eta_W\in(0,1)$ is defined, $W_\wedge$ has the same upper endpoint as $W$, and $\bar{F}_{W_\wedge} (w^\star-s) = \ell_{W_\wedge}(\bar{F}_W(s))\bar{F}_W(s)^{1/\eta_W}$, implying $W_\wedge$ is also in the negative Weibull MDA with $\alpha_{W_\wedge} = \alpha_W/\eta_W$. The second and third claims then follow from parts~1 and~3 of Proposition~\ref{prop:negwei}. 
 \end{proof}

\subsection*{Acknowledgments}
Financial support from the Swiss National Science Foundation grant 200021-166274 (Sebastian Engelke) and UK Engineering and Physical Sciences Research Council grant EP/P002838/1 (Jennifer Wadsworth) is gratefully acknowledged. Thomas Opitz was partially funded by the French national programme LEFE/INSU.
\appendix

\section{Lemmas and proofs}
\label{app:AProofs}
\subsection{Additional lemmas}

The following Lemma~\ref{lem:breiman} is widely known as \emph{Breiman's lemma} and is useful in several contexts throughout Section~\ref{sec:Proofs}.

\begin{lemma}[Breiman's lemma, see \citet{Breiman1965,ClineSamorodnitsky1994} and \citet{Pakes.2004}, Lemma~2.1]\label{lem:breiman}
	Suppose $X\sim F$, $Y\sim G$ are independent random variables. If $\overline{F}\in\RV_{-\alpha}^\infty$ with $\alpha\geq 0$ and $Y\geq 0$ with $\E (Y^{\alpha+\varepsilon})<\infty$ for some $\varepsilon>0$, then $$\overline{F}_{XY}(x)\sim \E\left(Y^{\alpha}\right)\, \overline{F}(x), \qquad x\to\infty.$$ 
	Equivalently, if $F\in\mathrm{ET}_{\alpha}$ and $\E(e^{(\alpha+\epsilon)Y})<\infty$, then $\bar{F}_{X+Y}(x)=\overline{F\star G}(x)\sim \E\left(e^{\alpha Y}\right)\, \overline{F}(x)$.
\end{lemma}

The following Lemma~\ref{lem:gamleq1} provides some additional detail on the function $\tau$, as defined in Section~\ref{sec:Constrained}.
\begin{lemma}
	\label{lem:gamleq1}
	If $1-\tau_1(b_1-\cdot) \in \RV_{1/\gamma_1}^0$, for $\gamma_1>0$, then $b_1 - \tau_1^{-1}(1-\cdot) \in \RV_{1/\gamma_1}^0$, and $\gamma_1 \leq 1$. Similarly, if $1-\tau_2(b_2+\cdot) \in \RV_{1/\gamma_2}^0$, for $\gamma_2>0$, then $\tau_2^{-1}(1-\cdot)-b_2 \in \RV_{\gamma_2}^0$ and $\gamma_2 \leq 1$.
\end{lemma}
\begin{proof}
	The argument is similar for both cases, so we focus on the first one. Let $g(t) = 1-\tau_1(b_1-t)$, which is decreasing as $t\to 0$ and invertible, with $g^{-1}(s)=b_1 - \tau_1^{-1}(1-s)$. Since $g \in \RV^0_{1/\gamma_1}$, then $g^{-1} \in \RV^0_{\gamma_1}$ (\citet[][Proposition~2.6(v)]{Resnick07}, adapting to regular variation at zero). Now make the left-sided Taylor expansion
	\[
	\tau_1(b_1-t) = \tau_1(b_1) -t \tau_1'(b_{1-}) + O(t^2), 
	\]
	and note that $\tau_1(b_1) = 1$. Consequently, $1-\tau_1(b_1-t) = t \tau_1'(b_{1-}) + O(t^2)$, and so finiteness of $\tau_1'(b_{1-})$ will imply that the index of regular variation of $g$ is at least 1. Define the convex function $\mu:(0,\infty)\to (0,\infty)$ by $\mu(x) = \nu(1,x)$, so that $\tau(x) = 1/\mu(1/x-1)$. Since $\tau_1$ is increasing on $(0,b_1)$, $\mu$ is increasing on $((1-b_1)/b_1, \infty)$. We have
	\[
	\tau'(x_{-}) = x^{-2} \mu'((1/x-1)_+) / \mu(1/x-1)^2, 
	\]
	and so $\tau_1'(b_{1-}) = b_1^{-2} \mu'([(1-b_1)/b_1]_+) / \mu((1-b_1)/b_1)^2$. For $h\in(0,1)$, convexity entails
	\[
	\mu((1-b_1)/b_1 + h) \leq h \mu((1-b_1)/b_1+1) +(1-h)\mu((1-b_1)/b_1),
	\]
	and so
	\[
	0\leq \frac{\mu((1-b_1)/b_1 + h) - \mu((1-b_1)/b_1)}{h} \leq \mu((1-b_1)/b_1+1)-\mu((1-b_1)/b_1) <\infty.
	\]
	Hence $\mu'([(1-b_1)/b_1]_+)<\infty$, giving $\tau_1'(b_{1-})<\infty$. Thus the index of regular variation of $g$ is at least 1.
\end{proof}

The following Lemma~\ref{lem:negvals} clarifies the influence of negative values in convolutions of exponential-tailed distributions. It allows us to extend certain results from the literature formulated for nonnegative random variables to the real line. 
\begin{lemma}[Convolutions of exponential-tailed distributions with negative values]
	\label{lem:negvals}
	For $i=1,2$ and probability distributions $F_i\in\mathrm{ET}_\alpha$  defined over $\mathbb{R}$ with $\alpha>0$, denote $p_i=\bar{F}_i(0)\in[0,1]$ the probability of nonnegative values. Using the convention $0/0=0$, let $F_i^+$ with $1-F_i^+(x)=\bar{F}_i(x)/p_i$, $x\geq 0$, denote the conditional distribution of $F_i$ over nonnegative values, and let $F_i^-(x)=F_i(x)/(1-p_i)$, $x<0$, denote the conditional distribution of $F_i$ over negative values. We use the notation $M_H(\alpha)=\int_{-\infty}^\infty\exp(\alpha y)H(\mathrm{d}y)$ for a given distribution $H$. Then:
	\begin{enumerate}
		\item If $\bar{F_i^+}(x)/\bar{F_1^+\star F_2^+}(x)\rightarrow 0$ when $x\rightarrow\infty$ for $i=1,2$, then 
		\begin{equation}\label{eq:F1F2bothdominated}
		\bar{F_1\star F_2} \sim p_1p_2\bar{F_1^+\star F_2^+}.
		\end{equation} \label{lemnegvalitem1}
		\item If $\bar{F_1^+\star F_2^+} \sim c_1\bar{F_1^+}+c_2\bar{F_2^+}$ with constants $0\leq c_1,c_2<\infty$, then
		\begin{equation}\label{eq:F1F2conveq}
		\bar{F_1\star F_2}(x) \sim \bar{F}_1(x) \left( p_2c_1+(1-p_2)M_{F_2^-}(\alpha)\right)+\bar{F}_2(x) \left( p_1c_2+(1-p_1)M_{F_1^-}(\alpha)\right).
		\end{equation}
		Specifically, if $c_1=M_{F_2^+}(\alpha)$ and $c_2=M_{F_1^+}(\alpha)$, then 
		$$
		\bar{F_1\star F_2}(x)\sim M_{F_2}(\alpha) \bar{F}_1(x)+M_{F_1}(\alpha) \bar{F}_2(x),
		$$
		and if $c_1=M_{F_2^+}(\alpha)$ and $c_2=0$, then 
		$$
		\bar{F_1\star F_2}(x)\sim M_{F_2}(\alpha) \bar{F}_1(x)+\bar{F}_2(x)(1-p_1)m_{F_1^-}(\alpha).
		$$\label{lemnegvalitem2}
	\end{enumerate}
\end{lemma}
\begin{proof}
	We start with the mixture representation
	$$
	\bar{F_1\star F_2}(x)=p_1p_2\bar{F_1^+\star F_2^+}(x)+p_1(1-p_2)\bar{F_1^+\star F_2^-}(x)+(1-p_1)p_2\bar{F_1^-\star F_2^+}(x), \quad x \geq 0.
	$$
	We can then use the equation $p_i\bar{F_i^+}(x)=\bar{F_i}(x)$, $x\geq 0$, and the following inequalities for $\{i_1,i_2\}=\{1,2\}$, 
	$$
	0=\bar{F_{i_1}^-\star F_{i_2}^-}(x) = \bar{F_{i_1}^-}(x)\leq \bar{F_{i_1}^+\star F_{i_2}^-}(x) \leq \bar{F_{i_1}^+}(x) \leq \bar{F_{i_1}^+\star F_{i_2}^+}(x) , \quad x\geq 0.
	$$
	Moreover, Lemma~\ref{lem:breiman}  can be  applied for mixed terms, yielding $\overline{F_{i_1}^+\star F_{i_2}^-}\sim M_{F_{i_2}^-}(\alpha)\overline{F_{i_1}^+}$, and then \eqref{eq:F1F2bothdominated} and \eqref{eq:F1F2conveq} follow from straightforward calculations.  To determine the behavior for special cases of $c_1$ and $c_2$, observe that $(1-p_i)M_{F_i^-}(\alpha)+p_iM_{F_i^+}(\alpha)=M_{F_i}(\alpha)$. 
\end{proof}

\subsection{Proof of lemmas in Section~\ref{sec:Proofs}}

\begin{proof}[Proof of Lemma~\ref{lem:taylor}]
	1. If $b_1=b_2=1$ then $\P(W=1)=0$ and $\bar{F}_W(1-s)=\bar{F}_Z(\tau^{-1}(1-s))$. By assumption, $\bar{F}_Z(\tau^{-1}(1-s)) = \ell_Z(1-\tau^{-1}(1-s))(1-\tau^{-1}(1-s))^{\alpha_Z}$, where $\ell_Z \in \RV_0^0$. Since  $1-\tau^{-1}(1-s) \in \RV_\gamma^0$ (Lemma~\ref{lem:gamleq1}) with limit zero, results on composition of regularly varying functions \citep[][Proposition~2.6~(iv)]{Resnick07} implies the result, with $\alpha_W = \alpha_Z\gamma$.\\
	
	\noindent
	2.a) If $b_1<1$ then 
	\begin{align}
	\bar{F}_W(1-s) - \P(W=1) = \P(Z\in[\tau_1^{-1}(1-s),b_1))+\P(Z\in (b_2,\tau_2^{-1}(1-s)]), \label{eq:FWFZ}
	\end{align}
	and since $Z$ has a Lebesgue density,
	\begin{align}
	&\P(Z\in[\tau_1^{-1}(1-s),b_1))+\P(Z\in (b_2,\tau_2^{-1}(1-s)])\notag\\ &\sim  f_Z(b_1)(b_1 - \tau_1^{-1}(1-s)) + f_Z(b_2)(\tau_2^{-1}(1-s) - b_2) \in \RV_\gamma^0, \label{eq:fZdecay}
	\end{align}
	again using Lemma~\ref{lem:gamleq1}.\\
	\noindent     
	2.b) A left-sided Taylor expansion of $\tau_1^{-1}$ about $1$ gives 
	\[
	\tau_1^{-1}(1-s) = b_1  - (\tau_1^{-1})'(1_{-}) s + O(s^2),
	\]
	where $O(s^2)/s$ uniformly tends to 0 as $s\to 0$, and similarly we can make a right-sided expansion for $\tau_2^{-1}(1-s)$.
	Hence, using~\eqref{eq:FWFZ} and~\eqref{eq:fZdecay}, $ \bar{F}_W(1-s) - \P(W=1) = s \ell(s)$ with $\lim_{s\to 0}\ell(s) = f_Z(b_1) (\tau_1^{-1})'(1_-) - f_Z(b_2)(\tau_2^{-1})'(1_-)$. Noting the link $1/\tau_1'(b_{1-}) = (\tau_1^{-1})'(1_-)$, similarly for $1/\tau_2'(b_{2+})$, gives Equation~\eqref{eq:L}.\\
	
	\noindent
	3. For the final part, we have 
	$$\bar{F}_{W_\wedge}(\zeta(1-s)) = \bar{F}_Z(\tau_1^{-1}(\zeta(1-s))) - \bar{F}_Z(1-\tau_1^{-1}(\zeta(1-s))) \sim  f_Z(1/2)\left\{ 1 - 2\tau_1^{-1}(\zeta(1-s))\right\}.$$
	Again by left-sided Taylor expansion of $\tau_1^{-1}$ about $\zeta = \tau(1/2)$, we have
	$$\tau_1^{-1}(\zeta(1-s)) = 1/2  - (\tau_1^{-1})'(\zeta_-) \zeta s + O(s^2),$$
	and so we obtain
	$\bar{F}_Z(\tau_1^{-1}(\zeta(1-s))) - \bar{F}_Z(1-\tau_1^{-1}(\zeta(1-s))) = s \ell_{\wedge}(s)$ with $\lim_{s\to 0}\ell_{\wedge}(s) = 2 f_{Z}(1/2)\zeta(\tau_1^{-1})'(\zeta_-)$. Noting again that $1/\tau_1'(1/2_-) = (\tau_1^{-1})'(\zeta_-)$, we arrive at Equation~\eqref{eq:Ltilde}.
\end{proof}

\begin{proof}[Proof of Lemma~\ref{lem:gamma}]
	The result for nonnegative $S$ and $V$ is found in Theorem~4(v) of \citet{Cline.1986}. The extension to negative values then follows from Lemma~\ref{lem:negvals}\eqref{lemnegvalitem1}.
\end{proof}

\begin{proof}[Proof of Lemma~\ref{lem:etratio}]
The result  is given in Theorem~6(ii,iii) of \citet{Cline.1986} for $\bar{F}(0)=1$ and $\bar{G}_1(0)=1$.
For point 1, the extension to negative values in $F$ and $G_1$ follows from observing that Theorem~6(iii) of \citet{Cline.1986} implies
\begin{equation}
\bar{F\star G_1}\sim \bar{F^+\star G_1^+}, \qquad x\rightarrow\infty,
\end{equation}
where $F^+$ is obtained from $F$ by setting $F^+(0)=F(0)$ and $\bar{F^+}(0)=1$, and the same construction is taken for $G_1^+$; i.e., $F^+$ and $G_1^+$ arise from projecting negative values to $0$. For point 2, the extension to negative values can be shown using Lemma~\ref{lem:negvals}\eqref{lemnegvalitem2}.  Indeed, the same limit  $M_{G_1}(\alpha)$ arises for $\bar{F\star G_1}(x)/\bar{F}(x)$ if we project negative values in $F$ and $G_1$ to $0$ or not.
\end{proof}

\section{Tail classes and examples}
\label{app:dist}
Definitions of tail classes are given in Section \ref{sec:Notation}. The following lemma summarizes important relationships between such tail classes. In this section, we refer to the class of heavy-tailed distributions by $\mathrm{HT}$, and to superheavy-tailed distributions by $\mathrm{SHT}$.

\begin{lemma}[Relationships between tail classes]
The following relationships between distribution classes hold: 	
	\begin{enumerate}
		\item $\mathrm{RV}_\alpha^\infty\subset \mathrm{CE}_0$ for $\alpha>0$,
        \item $\mathrm{ET}_0\subsetneq \mathrm{HT}$.
		\item For $\mathrm{ET}_\alpha$ with $\alpha>0$, we have:
		\begin{itemize}
					\item $F(\exp(\cdot))\in \mathrm{ET}_\alpha \Leftrightarrow F \in \mathrm{RV}_\alpha^\infty$,
		\item $\mathrm{CE}_\alpha\subset \mathrm{ET}_\alpha$,
		\item $\mathrm{ET}_{\alpha,\beta>-1}\cap \mathrm{CE}=\emptyset$.
	\end{itemize}
   \item For $\mathrm{WT}_\beta$, we have:
   \begin{itemize}
   	\item $\mathrm{WT}_1\subset \bigcup_{\alpha>0}\mathrm{ET}_\alpha$,
   	\item $\mathrm{WT}_\beta\subset\mathrm{CE}_0$ for $\beta<1$,
   	\item $\mathrm{LWT}_\beta\subset \mathrm{SHT}$ for $\beta <1$.
   \end{itemize} 
   \item By denoting $F_1\prec F_2$ if  $\overline{F}_1(x)/\overline{F}_2(x)\rightarrow 0$ for $x\rightarrow\infty$, we have: 
   \begin{itemize}
   	\item If $\tilde{\alpha}<\alpha$, then $\mathrm{WT}_{\beta>1}\prec \mathrm{ET}_\alpha \prec \mathrm{ET}_{\tilde{\alpha}} \prec \mathrm{WT}_{\beta <1}\prec \mathrm{LWT}_{\beta>1}\prec \mathrm{RV}_{\alpha>0}^\infty\prec \mathrm{RV}_{\tilde{\alpha}}^\infty\prec \mathrm{SHT}$.
   	\item $\mathrm{CE}_{\alpha>0}\prec \mathrm{ET}_{\tilde{\alpha},\beta}$ for $\tilde{\alpha}\leq \alpha$ and any $\beta>0$.
   \end{itemize}
	\end{enumerate}
\end{lemma}
We recall the membership in tail classes for well-known parametric distribution families in Table~\ref{tab:dists}, see \citet{Johnson.al.1994,Johnson.al.1995} for reference about parameters. Here we abstract away from the usual parameter symbols of these distributions to avoid conflicting notations with general tail parameters. We refer parameters as $\mathrm{scl}$  and $\mathrm{loc}$ if $\mathrm{scl}\times X+\mathrm{loc}$ has scale $\mathrm{scl}$ and location $\mathrm{loc}$, where $X$ has scale $1$ and location $0$. Another parameter  $\mathrm{shp}$ may be related to shape for some distributions. 
\begin{table} 
\begin{center}
\begin{tabular}{|l|c|c|c|c|c|c|} \hline
  & {$\RV_\alpha^0$} & {$\mathrm{WT}_\beta$} &   {$\mathrm{ET}_\alpha$}   & {$\mathrm{HT}$} & {$\mathrm{LWT}_\beta$} & {$\RV_\alpha^\infty$}  \\
                                 & & & & &  &\\ \hline
normal & & $\beta=2$ & &  & & \\
log-normal & &  & & \ding{51} & $\beta=2$ & \\
exponential & & $\beta=1$ &  $\alpha=\mathrm{scl}$  & & & \\
gamma & & $\beta=1$ &  $\alpha=\mathrm{scl}$ &  & & \\
inverse normal & & $\beta=1$ & $\alpha=\frac{\mathrm{shp}}{2\,\mathrm{mean}^2}$  & & & \\
logistic & & $\beta=1$ &  $\alpha=\mathrm{scl}$ & & & \\
log-logistic & & $\beta=1$ &  & \ding{51} & & $\alpha=\mathrm{shp}$ \\
Gumbel & & $\beta=1$ &  $\alpha=\mathrm{scl}$ &  & & \\
Weibull &  & $\beta=\mathrm{shp}$ &  $\mathrm{shp}=1,\alpha=\mathrm{scl}$ & $\mathrm{shp}<1$ & & \\
$t$ & & & & \ding{51} & $\beta=1$ & $\alpha=\mathrm{shp}$ \\
Pareto & &  & & \ding{51}& $\beta=1$ & $\alpha=\mathrm{shp}$\\
Fr\'echet & & & &  \ding{51} & $\beta=1$ & $\alpha=\mathrm{shp}$ \\
stable & & $\beta=\mathrm{shp}=2$ & &  $\mathrm{shp}<2$ & $\mathrm{shp}<2$,$\beta=1$ & $\alpha=1/\mathrm{shp}>1/2$\\
$F(\mathrm{shp1},\mathrm{shp2})$ &  & & & \ding{51} & & $\alpha=2/\mathrm{shp2}$ \\
uniform & $\alpha=1$  & & & & & \\
Beta$(\mathrm{shp1},\mathrm{shp2})$ & $\alpha=\mathrm{shp2}$ &  & & & & \\
triangular & $\alpha=2$ & &  & & & \\
GEV & $\alpha=1/\mathrm{shp}<0$ & $\mathrm{shp}=0,\beta=1$   & $\mathrm{shp}=0,\alpha=\mathrm{scl}$ & $\mathrm{shp}>0$ & $\mathrm{shp}>0,\beta=1$ & $\alpha=1/\mathrm{shp}$ \\
\hline
\end{tabular}
\end{center}
\caption{Membership in tail classes (columns) for distribution families (rows). The column $\RV_\alpha^0$ refers to the behavior of $\overline{F}(x^\star-\cdot)$ when $x^\star<\infty$. All heavy-tailed distributions in this table are also subexponential. All distributions in $\mathrm{ET}_\alpha$ listed in this table are in $\mathrm{ET}_{\alpha,\beta}$ except for the inverse normal; the inverse normal is in $\mathrm{CE}_\alpha$. 
The parameter $\mathrm{shp}$ of the stable distributions is here chosen as their stability parameter.}
  \label{tab:dists}
\end{table}

\newpage

\section{Additional illustrations}
\label{app:addill}
Figure~\ref{fig:nutaukappa2} illustrates further examples of norms $\nu$ and related functions $\tau(z)$ and $\tau(1-z)$, as defined in Section~\ref{sec:Constrained}.
\label{sec:Supporting}
\begin{figure}[H]
\centering
\includegraphics[width=0.45\textwidth]{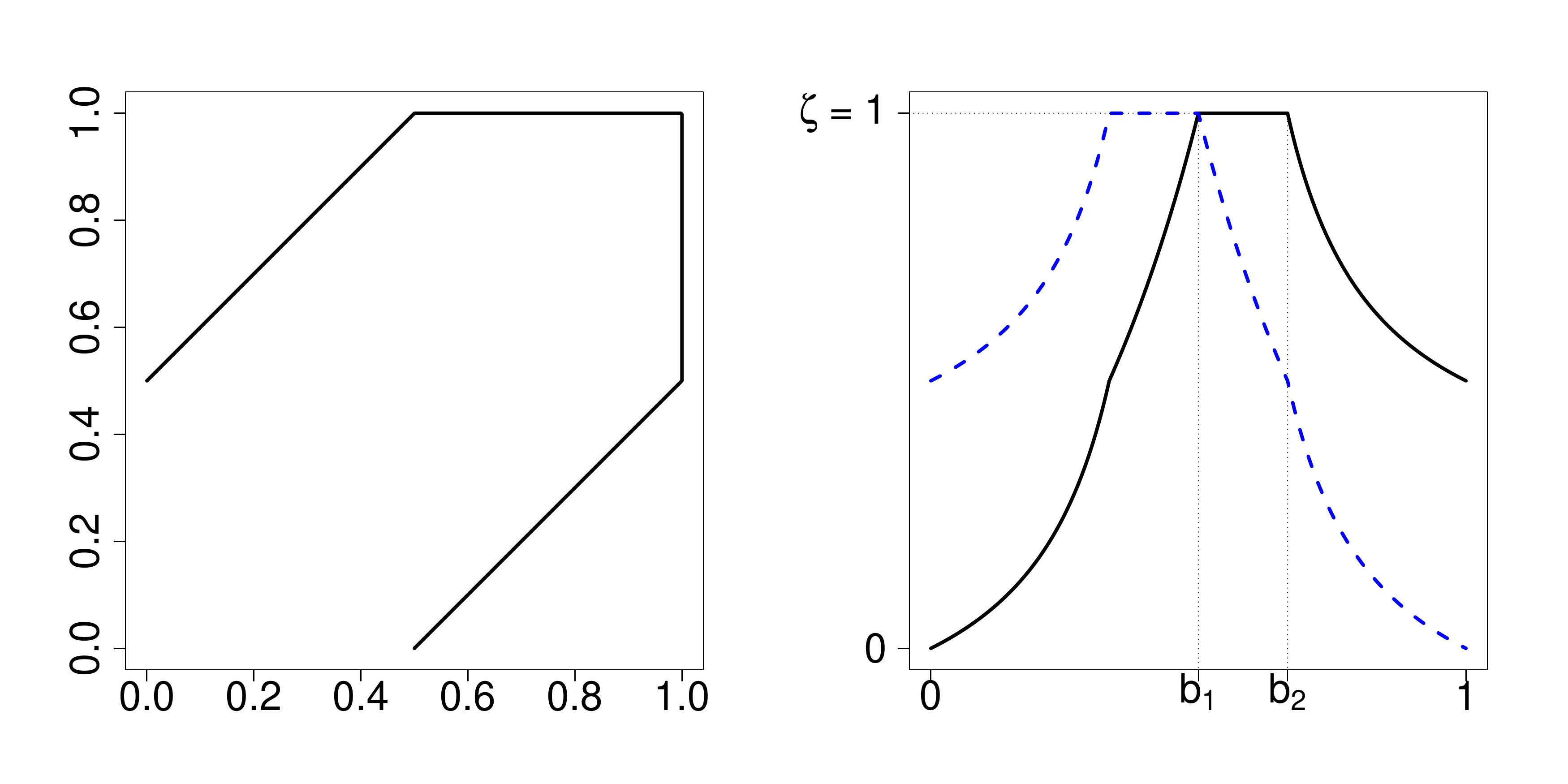}
\includegraphics[width=0.45\textwidth]{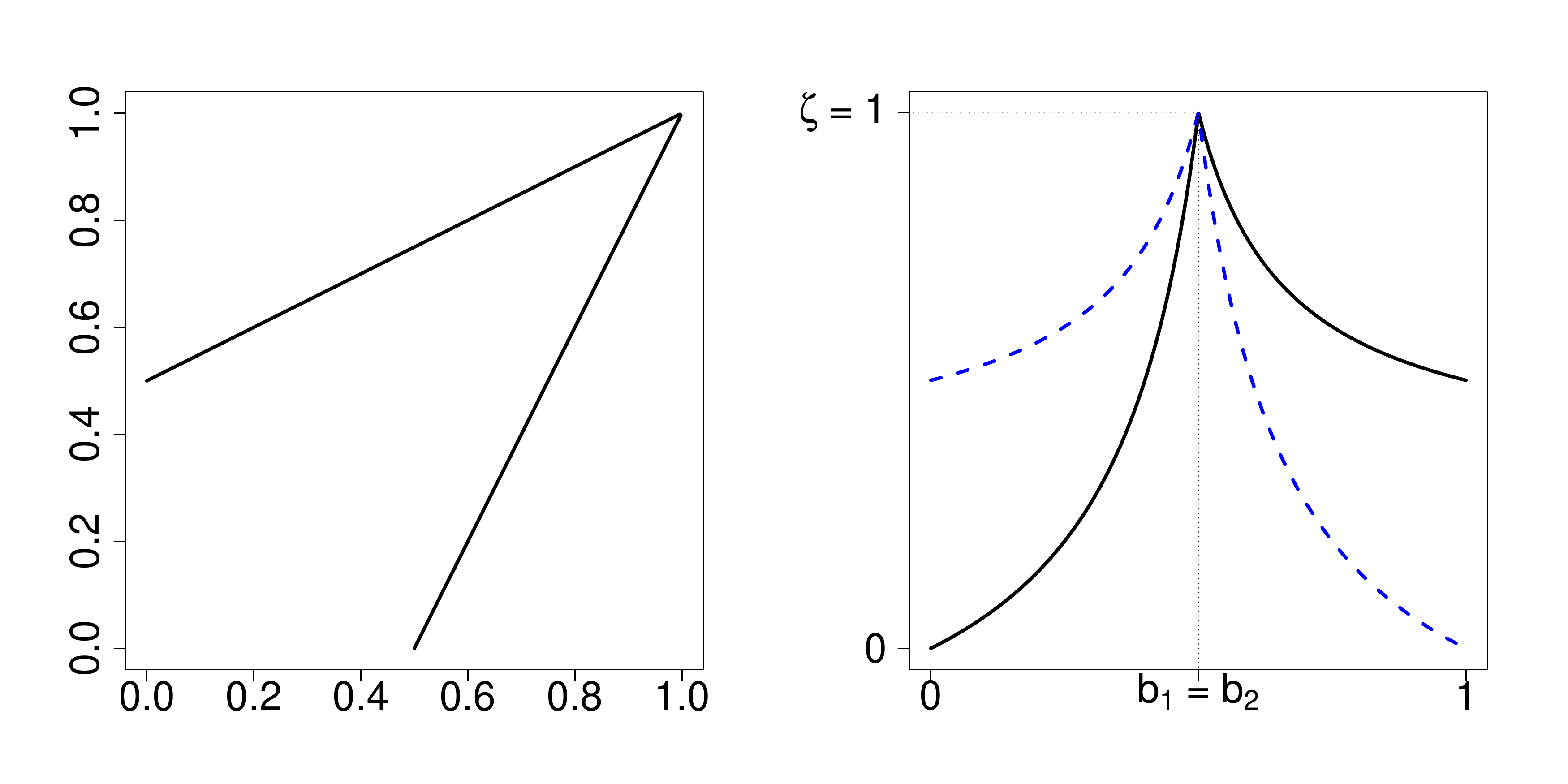}\\
\includegraphics[width=0.45\textwidth]{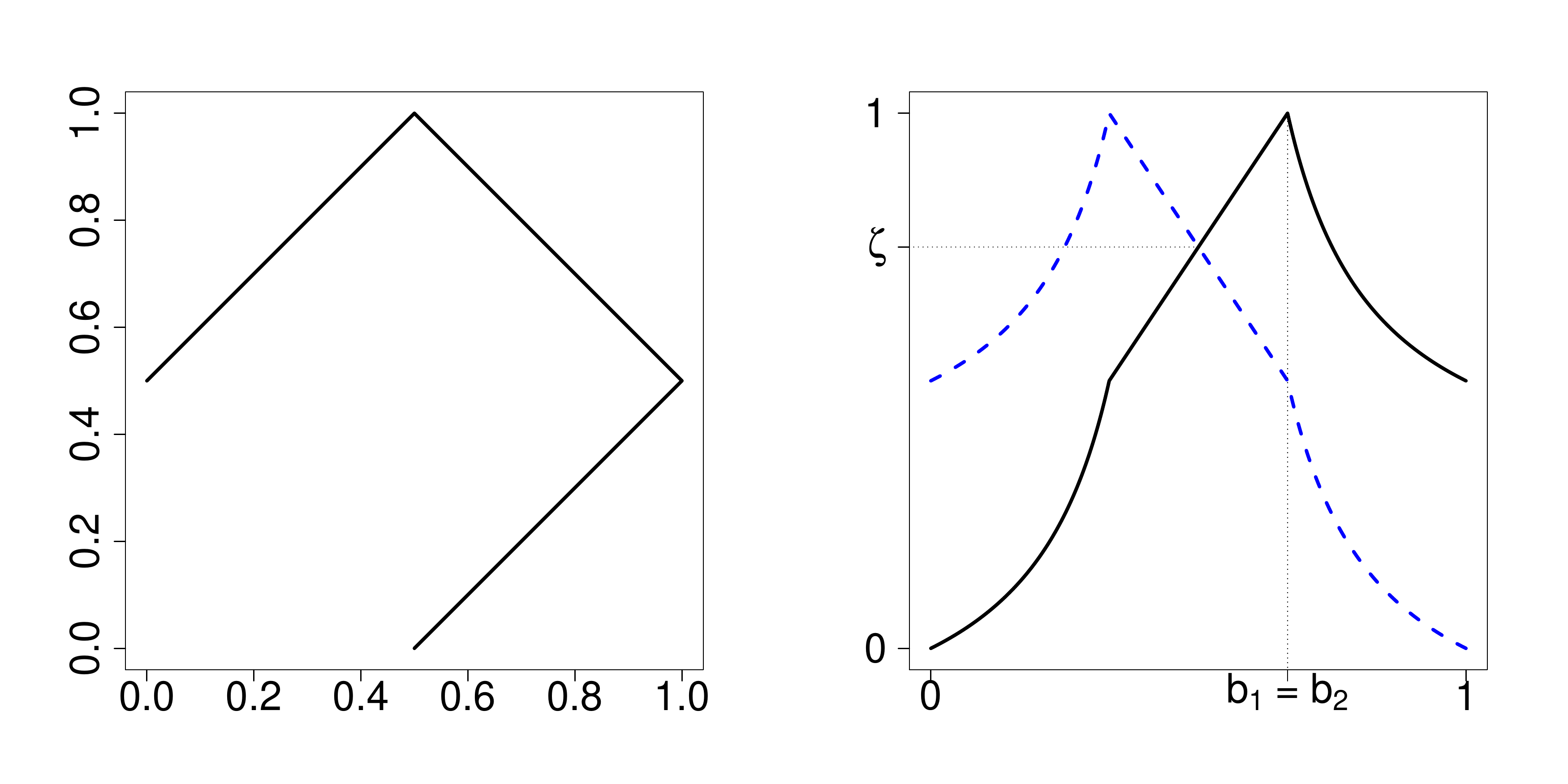}
\includegraphics[width=0.45\textwidth]{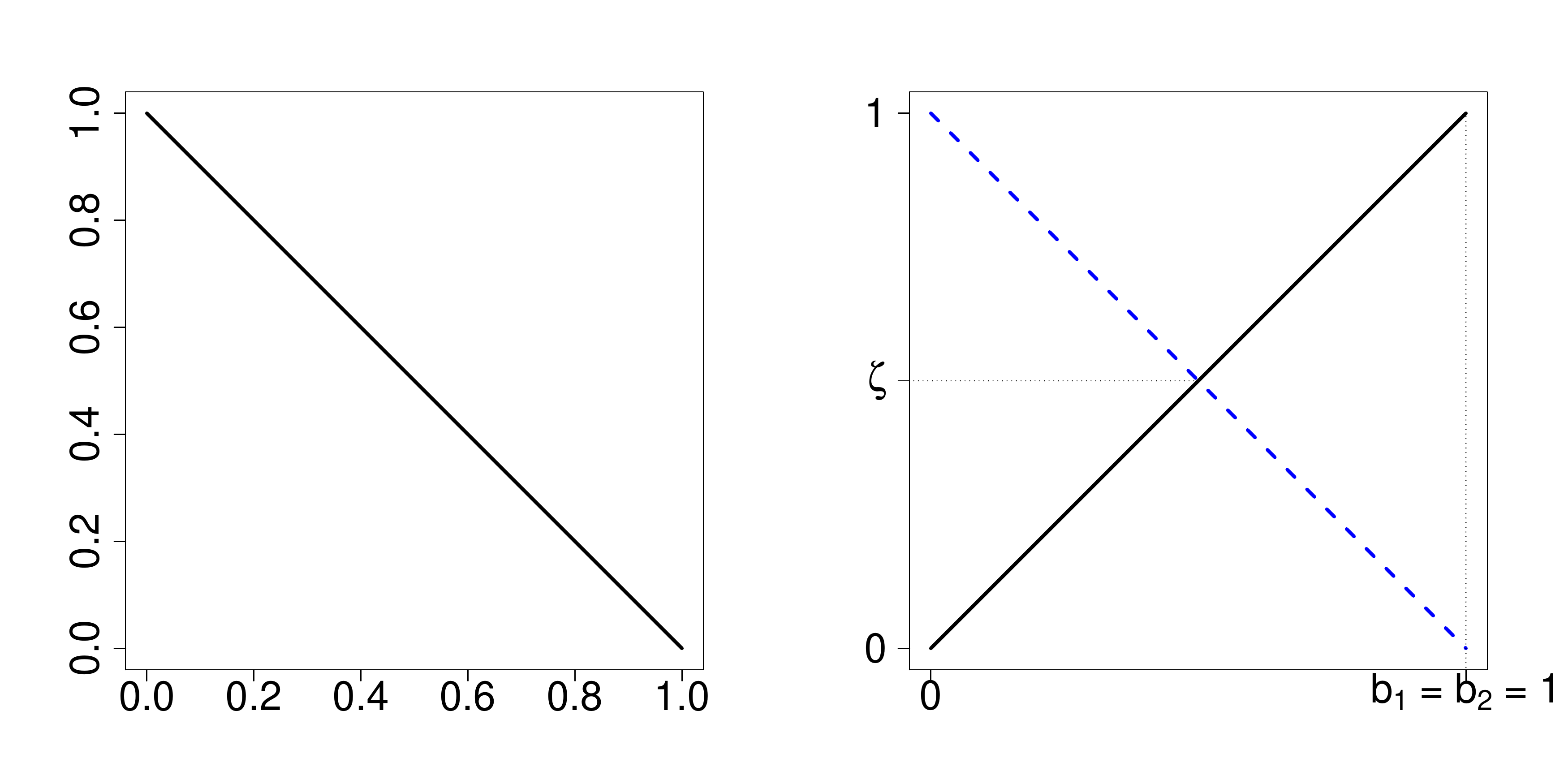}
\caption{Further illustration of different norms $\nu$ and their related functions $\tau(z)$ (solid line) and $\tau(1-z)$ (dashed line).}
\label{fig:nutaukappa2}
\end{figure}

\bibliographystyle{apalike}
\bibliography{RSC}

\end{document}